\documentclass[11pt, letterpaper, 
]{amsart}

\usepackage[margin=30truemm]{geometry}

\usepackage{amsmath}
\usepackage{amsfonts}
\usepackage{amssymb}
\usepackage{amsthm}
\usepackage{amsrefs}
\usepackage{mathrsfs}
\usepackage{graphics}
\usepackage[all]{xy}
\usepackage{mathtools}
\newcommand{\seq}{\coloneqq}
\usepackage[
bookmarks=false,
colorlinks=true,
debug=true,
naturalnames=true,
pdfnewwindow=true,
citecolor=blue,
linkcolor=blue,
urlcolor = blue]{hyperref}
\usepackage{comment}
\mathtoolsset{showonlyrefs=true}
\usepackage{yfonts}

\newtheorem{Thm}{Theorem}[section]
\newtheorem{Cor}[Thm]{Corollary}
\newtheorem{Prop}[Thm]{Proposition}
\newtheorem{Lem}[Thm]{Lemma}
\newtheorem{conj}[Thm]{Conjecture}
\newtheorem{Conj}[Thm]{Conjecture}
\theoremstyle{definition}
\newtheorem{Def}[Thm]{Definition}
\newtheorem{Rem}[Thm]{Remark}
\newtheorem{Ex}[Thm]{Example}

\numberwithin{equation}{section}

\newcommand{\arxiv}[1]{\href{http://arxiv.org/abs/#1}{\texttt{arXiv:#1}}}

\newcommand{\ol}{\overline}
\newcommand{\ul}{\underline}

\newcommand{\ep}{\epsilon}
\newcommand{\bep}{{\boldsymbol{\ep}}}

\newcommand{\Hom}{\mathop{\mathrm{Hom}}\nolimits}

\newcommand{\Aut}{\mathop{\mathrm{Aut}}\nolimits}
\newcommand{\End}{\mathop{\mathrm{End}}\nolimits}
\newcommand{\Ext}{\mathop{\mathrm{Ext}}\nolimits}
\newcommand{\Ker}{\mathop{\mathrm{Ker}}\nolimits}
\newcommand{\Cok}{\mathop{\mathrm{Coker}}\nolimits}
\newcommand{\Image}{\mathop{\mathrm{Im}}\nolimits}
\newcommand{\Res}{\mathop{\mathrm{Res}}\nolimits}
\newcommand{\Ind}{\mathop{\mathrm{Ind}}\nolimits}

\newcommand{\vdim}{\mathop{\ul{\dim}}}

\newcommand{\fg}{\mathfrak{g}}

\newcommand{\SG}{\mathfrak{S}}

\newcommand{\tv}{\mathfrak{T}}
\newcommand{\fo}{\mathfrak{o}}

\newcommand{\N}{\mathbb{N}}
\newcommand{\Z}{\mathbb{Z}}
\newcommand{\Q}{\mathbb{Q}}
\newcommand{\C}{\mathbb{C}}

\newcommand{\kk}{\Bbbk}
\newcommand{\Oo}{\mathbb{O}}
\newcommand{\Kk}{\mathbb{K}}
\newcommand{\bP}{\mathbb{P}}

\newcommand{\bD}{\mathbb{D}}

\newcommand{\mof}{\text{-$\mathsf{mod}$}}
\newcommand{\Mod}{\text{-$\mathsf{Mod}$}}
\newcommand{\sQ}{\mathsf{Q}}
\newcommand{\sW}{\mathsf{W}}
\newcommand{\sR}{\mathsf{R}}
\newcommand{\bB}{\mathsf{B}}

\newcommand{\Cc}{\mathscr{C}}
\newcommand{\Mm}{\mathscr{M}}
\newcommand{\scrQ}{\mathscr{Q}}
\newcommand{\scrP}{\mathscr{P}}

\newcommand{\cA}{\mathcal{A}}
\newcommand{\ocA}{\bar{\cA}}
\newcommand{\cR}{\mathcal{R}}
\newcommand{\cF}{\mathcal{F}}
\newcommand{\cG}{\mathcal{G}}
\newcommand{\cQ}{\mathcal{Q}}
\newcommand{\cL}{\mathcal{L}}
\newcommand{\cC}{\mathcal{C}}

\newcommand{\cM}{\mathcal{M}}

\newcommand{\rH}{\mathrm{H}}

\newcommand{\tL}{\tilde{L}}
\newcommand{\tM}{\tilde{M}}
\newcommand{\tB}{\tilde{B}}
\newcommand{\tE}{\tilde{E}}
\newcommand{\tF}{\tilde{F}}
\newcommand{\tY}{\tilde{Y}}
\newcommand{\tA}{\tilde{A}}
\newcommand{\tm}{\tilde{m}}

\newcommand{\wh}{\widehat}
\newcommand{\hI}{\hat{I}}
\newcommand{\hH}{\widehat{H}}

\newcommand{\rr}{\mathbf{r}}
\newcommand{\gMod}{\text{-$\mathsf{Mod}_\Z$}}

\newcommand{\bd}{{\boldsymbol{d}}}
\newcommand{\ii}{{\mathbf{i}}}
\newcommand{\bdelta}{{\boldsymbol{\delta}}}
\newcommand{\bv}{{\boldsymbol{v}}}

\newcommand{\Rep}{\mathop{\mathrm{Rep}}}

\newcommand{\For}{\mathop{\mathrm{For}}}
\newcommand{\Perv}{\mathop{\mathrm{Perv}}}
\newcommand{\Stab}{\mathop{\mathrm{Stab}}\nolimits}

\newcommand{\dhom}{\mathrm{h}}
\newcommand{\dext}{\mathrm{e}}

\newcommand{\id}{\mathsf{id}}
\newcommand{\op}{\mathrm{op}}

\newcommand{\qv}{\mathfrak{M}^\bullet}
\newcommand{\qvs}{Z^\bullet}
\newcommand{\qvr}{\mathfrak{M}_0^{\bullet \mathrm{reg}}}

\newcommand{\pt}{\mathrm{pt}}

\newcommand{\sM}{M^!}
\newcommand{\sbM}{\bar{M}^!}
\newcommand{\coM}{M^*}
\newcommand{\cbM}{\bar{M}^*}

\newcommand{\wt}{\mathrm{wt}}
\newcommand{\IC}{\mathrm{IC}}
\newcommand{\IH}{\mathrm{IH}}

\newcommand{\Gm}{{\C^{\times}}}

\newcommand{\ev}{\mathop{\mathrm{ev}}\nolimits}
\newcommand{\Gr}{\mathop{\mathrm{Gr}}\nolimits\!}

\newcommand{\fdim}{\mathrm{f}}
\newcommand{\nilp}{\mathrm{nilp}}

\title{Monoidal Jantzen filtrations}
\date{\today}

\makeatletter
\@namedef{subjclassname@2020}{%
  \textup{2020} Mathematics Subject Classification}
\makeatother
\subjclass[2020]{17B37, 20G42, 16T25, 81R50, 17B10, 17B67}

\author[R.~Fujita]{Ryo Fujita}
\address[R.~Fujita]{Research Institute for Mathematical Sciences, Kyoto University, Oiwake-Kitashirakawa, Sakyo, Kyoto, 606-8502, Japan}
\email{rfujita@kurims.kyoto-u.ac.jp}

\author[D.~Hernandez]{David Hernandez}
\address[D.~Hernandez]{Universit\'e Paris Cit\'e, Sorbonne Universit\'e, CNRS, IMJ-PRG, F-75013 Paris, France}
\email{david.hernandez@imj-prg.fr}

\begin{document}

\begin{abstract}
We introduce a monoidal analogue of Jantzen filtrations in the framework of monoidal abelian categories with generic braidings. 
It leads to a deformation of the multiplication of the Grothendieck ring. We conjecture, and we prove 
in many remarkable situations, that this deformation is associative so that our construction yields a 
quantization of the Grothendieck ring as well as analogs of Kazhdan-Lusztig polynomials. 
As a first main example, for finite-dimensional representations of simply-laced quantum loop algebras, 
we prove the associativity and we establish that the resulting quantization coincides with the quantum Grothendieck ring constructed by Nakajima and Varagnolo-Vasserot in a geometric manner. 
Hence, it yields a unified representation-theoretic interpretation of the quantum Grothendieck ring. As a second main example, we establish an analogous result for a monoidal category of finite-dimensional modules over symmetric quiver Hecke algebras categorifying the coordinate ring of a unipotent group associated with a Weyl group element. We obtain various applications, in particular on the homological structure of representations.
\end{abstract}

\maketitle

\tableofcontents

\section{Introduction}

Jantzen filtrations are at the origin of fundamental developments of representation theory. For instance, the celebrated Jantzen conjecture \cite{J} (and its reformulation by Gabber-Joseph \cite{GJ}), originally proved by Beilinson-Bernstein \cite{BB}, implies that the (original) Kazhdan-Lusztig polynomials \cite{KL} are interpreted in terms of Jantzen filtrations of Verma modules in the category $\mathscr{O}$ of a simple Lie algebra. 
 This explains remarkable properties of these polynomials: their coefficients are positive and their evaluation at $1$ are the multiplicities of simple modules in certain distinguished representations. This 
gives rise to the Kazhdan-Lusztig algorithm to compute characters of simple modules in certain important categories by using geometric representation theory.

The definition of Jantzen filtrations relies on an isomorphism of 
$\Kk$-vector spaces
$$\phi \colon V\otimes_{\Oo}\Kk\simeq W\otimes_\Oo\Kk$$
where $\Kk$ is the fraction field of an integral domain $\Oo$, and $V$, $W$ are $\Oo$-modules.
For $\mathfrak{p}$ a maximal ideal of $\Oo$, one has the respective filtrations 
$\mathfrak{p}^iV$ and $\mathfrak{p}^jW$ of $V$ and $W$, $i,j\geq 0$. The Jantzen filtrations
are obtained from their interplay via the isomorphism $\phi$ (see \cite[II.8]{J2} for the precise definition).

\subsection{Main construction}
\label{Ssec:main_construction}
We introduce a monoidal analogue of Jantzen filtrations in the framework of monoidal 
categories with generic braidings, which we call  $R$-matrices, with the following salient points in comparison to ordinary 
Jantzen filtrations:

\begin{enumerate}
\item Instead of one isomorphism $\phi$, our definition of the filtration of $W$ 
 is obtained from two remarkable isomorphisms
\begin{equation}\label{vwv}V\otimes_{\Oo} \Kk \simeq W\otimes_{\Oo}\Kk 
\simeq V'\otimes_{\Oo} \Kk,\end{equation}
by an interplay of the images of three relevant filtrations. 
\item Our filtrations lead to the deformation not only of certain multiplicities, but also of the structure constants of the Grothendieck ring of the monoidal category.
\end{enumerate}

The precise formula for the monoidal Jantzen filtrations is given in (\ref{expfm}).

Our general construction depends on the choice of a PBW-theory in the monoidal category,
that is a choice of a family of simple objects (the cuspidal objects) whose monoidal products (the mixed products) satisfy certain remarkable 
properties. Then the construction involves a deformation of this PBW-theory along a formal parameter together with $R$-matrices, crucial 
isomorphisms between deformations of the mixed products. These are isomorphisms in \eqref{vwv} where $W$ is a mixed product and $V$, $V'$ are distinguished mixed products, called respectively standard and costandard.

Our monoidal Jantzen filtrations are filtrations by subobjects
$$F_\bullet M\colon \quad  M \supset \cdots \supset F_{-1}M \supset F_0M \supset F_1M \supset \cdots \supset \{ 0 \}.$$

We establish that, under mild conditions, the filtrations are compatible 
with specializations of $R$-matrices and satisfy certain duality properties.

The decategorification version of the filtration $F_\bullet M$ is defined as 
$$ [M]_t \seq \sum_{n \in \Z} [\Gr_n^F M] t^n.$$
It belongs to the Grothendieck group of the category, with the coefficients
extended to $\mathbb{Z}[t^{\pm 1}]$ for a formal variable $t$.

Some of these coefficients are defined to be the analogues of Kazhdan-Lusztig polynomials. We establish the existence of a corresponding canonical basis under reasonable conditions.

Moreover, this decategorification defines a $\mathbb{Z}[t^{\pm 1/2}]$-bilinear operation $*$ (after a slight twist) that deforms the multiplication of the Grothendieck ring. We conjecture that in a general setting, this deformation defines a ring, that is the operation $*$ is associative. This is one of the new salient points in comparison to the original theory of Jantzen filtrations.

In this paper, we apply the general construction of monoidal Jantzen filtrations to the monoidal categories of finite-dimensional modules over quantum loop algebras and symmetric quiver Hecke algebras, and verify the expected associativity in many remarkable situations.
We can expect our theory extends to other frameworks, such as to the coherent Satake category \cite{cw} or to the representation theory of $p$-adic groups.

\subsection{Quantum loop algebras}
Our first main examples for monoidal Jantzen filtrations are realized in categories of finite-dimensional representations 
of the quantum loop algebra $U_q(L\mathfrak{g})$ associated with a complex simple Lie algebra $\mathfrak{g}$ and a generic quantum parameter $q \in \C^\times$.
This is a Hopf algebra whose finite-dimensional modules form an interesting abelian monoidal category  $\Cc$, which is neither semisimple nor braided. 
In particular, the tensor product $V \otimes W$ is not isomorphic to its opposite $W \otimes V$ for general simple modules $V, W \in \Cc$.
Nevertheless, their Jordan-H\"{o}lder factors coincide up to reordering. 
In other words, we have $[V \otimes W] = [W \otimes V]$ in the Grothendieck ring $K(\Cc)$, and hence $K(\Cc)$ is commutative.  
Indeed, this commutativity follows from the injectivity of the so-called $q$-character homomorphism $\chi_q \colon K(\Cc) \to \mathcal{Y} = \Z[Y_{i,a}^{\pm 1} \mid i \in I, a \in \mathbb{C}^\times]$ due to Frenkel-Reshetikhin~\cite{FR}, where $I$ is an index set of the simple roots of $\mathfrak{g}$.
Thus, one may identify $K(\Cc)$ with a subring of $\mathcal{Y}$.  

By the classification result due to Chari-Pressley~\cite{CP}, the set of classes of simple modules in $\Cc$ is in bijection with the set $ \cM^+  \subset \mathcal{Y}$ of monomials in the variables $Y_{i,a}$. 
For each $m \in \cM^+ $, the corresponding simple module $L(m)$ is of highest weight $m$, namely $\chi_q(L(m))$ has $m$ as its highest term.
The problem to compute $\chi_q(L(m))$ for all $m \in \cM^+ $ is of fundamental importance.
At the present moment, a  general closed formula (like the Weyl character formula) is not known.

  One possible strategy is to find an algorithm to compute $\chi_q(L(m))$ recursively, analogous to the Kazhdan-Lusztig algorithm. 
For each $x \in I \times \C^\times$, the $q$-character of the simple module $V_x \seq L(Y_x)$ (called a fundamental module) can be computed by an algorithm due to Frenkel-Mukhin~\cite{FM01}.
For each monomial $m = Y_{x_1}\cdots Y_{x_d} \in  \cM^+ $, if $(x_1, \ldots, x_d)$ is ordered suitably, the corresponding tensor product $M(m) \seq V_{x_1} \otimes \cdots \otimes V_{x_d}$ has a simple head isomorphic to $L(m)$. 
Moreover, there exists a partial ordering of $\cM$ (called the Nakajima partial ordering) such that we have
\[ [M(m)] = [L(m)] + \sum_{m' < m} P_{m,m'} [L(m')] \]
in $K(\Cc)$.  
The module $M(m)$ is called a standard module.
Since we know $\chi_q(M(m))$, it is enough to compute the multiplicities $P_{m,m'}$.
For this purpose, we consider a one-parameter (non-commutative) deformation of $K(\Cc)$, called the quantum Grothendieck ring.
It was introduced by Nakajima~\cite{Nak04} and by Varagnolo-Vasserot~\cite{VV} for $\mathfrak{g}$ of simply-laced type, and by the second author \cite{H1} for general $\mathfrak{g}$.
The quantum Grothendieck ring $K_t(\Cc)$ is a $\Z[t^{\pm 1/2}]$-subalgebra of a quantum torus $\mathcal{Y}_t$ deforming $\mathcal{Y}$, stable under a natural anti-involution $y \mapsto \bar{y}$ of $\mathcal{Y}_t$, and comes with a standard $\Z[t^{\pm 1/2}]$-basis $\{M_t(m)\}_{m \in  \cM^+}$. 
Under the specialization $t \to 1$, $M_{t}(m)$ goes to $[M(m)]$. 
We can prove (see \cite{Nak04, H1}) that there exists the canonical basis $\{L_t(m)\}_{m \in \cM^+}$ satisfying $\overline{L_t(m)} = L_t(m)$ and
\[ M_t(m) = L_t(m) + \sum_{m' < m} P_{m,m'}(t) L_t(m')\]
for some $P_{m,m'}(t) \in t\Z[t]$.
This characterization enables us to compute the polynomials $P_{m,m'}(t)$ recursively.
When $\mathfrak{g}$ is of simply-laced type, the following result was obtained by using perverse sheaves on quiver varieties.

\begin{Thm}[\cite{Nak04, VV}]
When $\mathfrak{g}$ is of simply-laced type, the following properties hold:
\begin{itemize}
\item[(KL)] Analog of Kazhdan-Lusztig conjecture: under the specialization $t \to 1$, $L_t(m)$ goes to $[L(m)]$, or equivalently, we have $P_{m,m'}(1) = P_{m,m'}$.
\item[(P)] Positivity: for any $m' < m$, we have $P_{m,m'}(t) \in \Z_{\ge 0}[t]$.
\end{itemize}
\end{Thm}    

The second author \cite{H1} conjectured that these properties hold for general $\mathfrak{g}$.
Very recently, with Oh and Oya, we obtained some pieces of evidence of this conjecture. 

\begin{Thm}[\cite{FHOO, FHOO2}]
The property {\rm (KL)}  also holds when $\mathfrak{g}$ is of type $B$. For 
general $\mathfrak{g}$, the property {\rm (KL)}  also holds for all simple modules that are reachable (in the sense of cluster algebras).
The property {\rm (P)}  holds for general $\mathfrak{g}$. 
\end{Thm}  

Having these results, we ask: what is representation-theoretic meaning of $K_t(\Cc)$ or $P_{m,m'}(t)$?

Here we propose an answer to this question by introducing monoidal Jantzen filtrations for any tensor products of fundamental modules.
For \emph{any} sequence $\bep = (\ep_1, \ldots, \ep_d)$ of elements of $I \times \C^\times$, let $M(\bep) \seq V_{\ep_1} \otimes \cdots \otimes V_{\ep_d}$ be the corresponding tensor product, which is not necessarily a standard module (we call it is a mixed product).
By using $R$-matrices, we define a monoidal Jantzen filtration $F_\bullet M(\bep)$ 
by $U_q(L\mathfrak{g})$-submodules (in the paper, we will also handle more general PBW-theories in $\Cc$). The decategorification gives a corresponding element $[M(\bep)]_t$ of the $t$-deformed Grothendieck group $K(\Cc)_t \seq K(\Cc)\otimes \Z[t^{\pm 1/2}]$.

Then we define a $\Z[t^{\pm 1/2}]$-bilinear map $* \colon K(\Cc)_t \times K(\Cc)_t \to K(\Cc)_t$ by
\[ [M(m)]_t * [M(m')]_t \seq t^{\gamma(m,m')} [M(m) \otimes M(m')]_t, \]
where $\gamma$ is a certain skew-symmetric bilinear form on $ \cM^+$.  
Also, $K(\Cc)_t$ is endowed with a natural involution $\overline{X \otimes f(t)} = X \otimes f(t^{-1})$.
Now we propose the following:

\begin{Conj}[= Conjecture \ref{qgrconj}] \label{conjF}
The pair $(K(\Cc)_t, *)$ defines a $\Z[t^{\pm 1/2}]$-algebra with anti-involution, and it is isomorphic to the quantum Grothendieck ring $K_t(\Cc)$ identifying the standard basis $\{M_t(m) \}_{m \in  \cM^+ }$ with the basis $\{[M(m)]_t\}_{m \in  \cM^+ }$. 
\end{Conj} 

Note that the associativity of the map $*$ is unclear from the definition. Besides, Conjecture~\ref{conjF} implies the above properties  {\rm (KL)} and {\rm (P)}.  
We prove the Conjecture~\ref{conjF} for  $\fg$ of simply-laced type. This is one of the main results of this paper.

\begin{Thm}[= Theorem \ref{simply-laced}] \label{Intro:simply-laced}
Conjecture~\ref{conjF} is true when $\mathfrak{g}$ is of simply-laced type.
\end{Thm}

As a consequence, we obtain a categorification of the quantum Grothendieck ring in terms of 
finite-dimensional representations enhanced with their monoidal Jantzen filtrations. 
Note also that it was established in \cite{HL15} that, when $\mathfrak{g}$ is of simply-laced type, the quantum Grothendieck ring contains a copy of the positive part $U_q(\mathfrak{n})$ of 
the finite-type quantum group $U_q(\mathfrak{g})$ (it corresponds to the quantum Grothendieck  ring 
of a monoidal subcategory of finite-dimensional representations). Hence we obtain as well a new
categorification of $U_q(\mathfrak{n})$ in terms of our monoidal Jantzen filtrations.

\subsection{Applications} After decades of intensive study, the structure of finite-dimensional
representations of quantum loop algebras
is  still  largely not understood, even in the $\mathfrak{sl}_2$-case. For example,
the classification  of  finite-dimensional
indecomposable representations is not known. Up to the authors
knowledge, the only known general result,
beyond the structure of the Grothendieck ring of the category, is that
standard modules have a unique
simple quotient, and that co-standard modules have a unique simple
submodule \cite{Chari, Kas, VV}. As an application of the results
of our paper, we obtain many new informations on the homological
structure of mixed products, which were
not known even in the $\mathfrak{sl}_2$-case. Indeed, by direct algebraic
computations in the quantum Grothendieck ring,
we can determine the simple constituents of the submodules and
subquotients obtained from the monoidal Jantzen
filtrations. This is illustrated in Examples of Section \ref{exfil2}.

For example, we obtain the following vast generalization of the result
of \cite{Chari, Kas, VV} recalled above 
 (see its proof in Section \ref{pfadd}). 


\begin{Thm}\label{thmapp}
Let $M$ be a mixed product. Let $S\in K(\mathscr{C})$ (resp.\ $Q\in K(\mathscr{C})$) be the coefficient of the highest
(resp.\ lowest) power of $t$ arising in $[M]_t$. Then $M$ admits a submodule
(resp.\ a quotient) whose image in $K(\mathscr{C})$ is $S$
(resp.\ $Q$).
\end{Thm}

Through our approach, we can see quantum Grothendieck rings as a
powerful tool to compute monoidal Jantzen filtrations,
which themselves form a new method to analyze the structure of mixed
tensor products.

Another application, at the moment conjectural, is the extension of the
Kazhdan-Lusztig algorithm
to compute $q$-character of simple modules in non-simply laced cases.
Indeed, it is known that
such an algorithm gives the correct answer for all simple modules in
simply-laced cases \cite{Nak04}, and
for all reachable simple modules in non-simply-laced cases \cite{FHOO2}.
If our associativity conjecture
is correct, then the coefficients of $[M]_t$ for $M$ a standard module
can be computed by an analog of
Kazhdan-Lusztig algorithm, and so we can obtain the result for all
simple representations in non-simply-laced cases.
Hence, the problem of computation of the character of simple modules is
reduced to the study of
the associativity of our bilinear operation $*$. 

\subsection{Strategy of the proof}
\label{Ssec:strategy}
Our proof of Theorem \ref{Intro:simply-laced} uses geometric method due to Nakajima involving perverse sheaves on quiver varieties. 
Actually, our strategy is much inspired by Grojnowski's unpublished note \cite{Groj}, which studies filtrations on standard modules over quantum loop algebras and affine Hecke algebras using perverse sheaves.  

Recall that the first proof of the original Jantzen conjecture for Verma modules by Beilinson-Bernstein \cite{BB} was also geometric, where the Jantzen filtrations are identified with the weight filtrations of some standard $\mathcal{D}$-modules on flag manifolds through the Beilinson-Bernstein localization.  
There is another approach due to Soergel \cite{Soe08} and K\"ubel \cite{Kub12}, which is a Koszul dual picture to Beilinson-Bernstein's proof.
In this second approach, the Jantzen filtrations are related to the Andersen filtrations on the $\Hom$-space from Verma to tilting modules in the category $\mathscr{O}$, which is in turn identified with the degree filtrations of  the local intersection cohomology of Schubert varieties.  
A key ingredient here is the hard Lefschetz theorem applied to the setting of the ``Fundamental Example'' of Bernstein-Lunts \cite{BL94}.  
See the introduction of \cite{Wil16} for more details and recent further development. 

Our proof of Theorem \ref{Intro:simply-laced} has a similar flavor to this second approach.
Based on Nakajima's geometric construction, we identify our monoidal Jantzen filtrations of the mixed products $M(\bep)$ with the degree filtrations of certain hyperbolic localizations (in the sense of Braden \cite{Braden}) of perverse sheaves on graded quiver varieties.
Here, key ingredients are again the hard Lefschetz property and the ``Fundamental Example'' mentioned above.
Since the Poincar\'e polynomials of these hyperbolic localizations serve the structure constants of the quantum Grothendieck ring $K_t(\Cc)$ in its geometric definition \cite{VV03}, we obtain the desired result.

\subsection{Symmetric quiver Hecke algebras}
Our second examples of the monoidal Jantzen filtrations are given by the finite-dimensional modules over symmetric quiver Hecke algebras. 
For any symmetric Kac-Moody algebra $\fg$ and an element $w$ of its Weyl group, one has a monoidal abelian category $\Cc_w$ consisting of finite-dimensional ungraded modules over the quiver Hecke algebras (or rather their completions), which categorifies the coordinate ring $\C[N(w)]$ of a unipotent algebraic group $N(w)$.  
Note that this category $\Cc_w$ is obtained from its graded version $\Cc^\bullet_w$ categorifying the quantized coordinate ring $A_t[N(w)]$ by forgetting the grading. 
When $w$ is the longest element $w_0$ of the Weyl group of finite type, the category $\Cc^\bullet_{w_0}$ is the category of all the finite-dimensional graded modules. 
To each reduced word $\ii = (i_1, i_2, \ldots, i_\ell)$ for $w$, one can associate the dual PBW-basis of $A_t[N(w)]$, which are categorified by ordered products of the so-called cuspidal modules in $\Cc_w^\bullet$ \cite{KKOP}.
Forgetting the grading, we have a basis of standard modules for the Grothendieck ring $K(\Cc_w)$.
Since the category $\Cc_w$ has generic braidings, one can apply the same construction as above to define the monoidal Jantzen filtrations and hence get a deformation $(K(\Cc_w)_t, *)$ of $K(\Cc_w)$.
Thus, it makes sense to expect that $(K(\Cc_w)_t, *)$ defines an associative algebra isomorphic to the quantum coordinate ring $A_t[N(w)]$.
In other words, the monoidal Jantzen filtrations in $\Cc_w$ may recover the forgotten gradings of the Jordan-H\"older multiplicities in $\Cc^\bullet_w$.
This is an analog of Conjecture \ref{conjF} above.
Note also that the same construction applies to the affine Hecke algebras of general linear groups as well, since their central completions are identical to the completions of quiver Hecke algebras of type $A$.
In this paper, we verify the conjecture in the following special case.

\begin{Thm}[= Theorem \ref{Thm:adapted}] \label{Intro:adapted}
The analog of Conjecture \ref{conjF} for $\Cc_w$ is true when the reduced word $\ii$ is adapted to a quiver.  
\end{Thm}

In fact, when the reduced word $\ii$ is adapted to a quiver $Q$, we have a geometric interpretation of the quiver Hecke algebra due to Varagnolo-Vasserot~\cite{VV11} and the relevant mixed product modules in terms of the equivariant perverse sheaves on the space of representations of the quiver $Q$, which appear in the construction of the canonical bases of quantized enveloping algebras due to Lusztig~\cite{LusB}. 
Theorem \ref{Intro:adapted} can be proved by applying the strategy as in Section \ref{Ssec:strategy} to this geometric situation. 
 
\subsection{Further application} 
We end this introduction with a brief discussion on another application of our monoidal Jantzen filtrations and  Theorem \ref{Intro:adapted}. 

Let $F \colon \Cc_1 \to \Cc_2$ be an exact monoidal functor between monoidal categories with generic braidings as in Section \ref{Ssec:main_construction}.
If $F$ sends cuspidal objects of $\Cc_1$ to cuspidal objects of $\Cc_2$ and sends $R$-matrices among cuspidal objects in $\Cc_1$ to those in $\Cc_2$, then it is immediate from the construction that $F$ sends the monoidal Jantzen filtration of a mixed tensor product in $\Cc_1$ to that of its image in $\Cc_2$. 
Therefore, it induces a homomorphism $(K(\Cc_1)_t, *) \to (K(F(\Cc_1))_t, *)$, where $F(\Cc_1) \subset \Cc_2$ denotes the essential image of $F$. 
In a good situation, the associativity of $(K(\Cc_1)_t, *)$ implies the associativity of $(K(F(\Cc_1))_t, *)$,

Examples of such nice functors may be provided by the generalized quantum affine Schur-Weyl duality introduced by Kang-Kashiwara-Kim~\cite{KKK}, which connect the monoidal categories of finite-dimensional modules over symmetric quiver Hecke algebras and quantum affine algebras.
As a remarkable special case, associated with a Q-datum $\mathcal{Q}$ for a finite-dimensional simple Lie algebra $\fg$ in the sense of \cite{FO21}, we have a certain monoidal Serre subcategory $\Cc_{\mathcal{Q}}$ of $\Cc$ for the quantum loop algebra of $\fg$, and a monoidal equivalence $F_\mathcal{Q} \colon \Cc_{w_0} \simeq \Cc_{\mathcal{Q}}$ with $\Cc_{w_0}$ being the category for the symmetric quiver Hecke algebra associated with the unfolding of $\fg$.  
See \cites{KKK15, KO19, OS19, Naoi21}.
To each reduced word $\ii$ for $w_0$, we have the associated PBW-theory for $\Cc_{w_0}$, whose image under $F_{\mathcal{Q}}$ gives a PBW-theory for $\Cc_{\mathcal{Q}}$. 
As discussed in \cite[Section 4]{KKOP}, the functor $F_\mathcal{Q}$ respects the $R$-matrices and hence the monoidal Jantzen filtrations.

When $\fg$ is of simply-laced type, a Q-datum $\mathcal{Q}$ for $\fg$ is the same as a Dynkin quiver $Q$ (plus a choice of height function).
In this case, if the word $\ii$ is adapted to $Q$, the functor $F_Q$ sends the PBW-theory of $\Cc_{w_0}$ associated with $\ii$ to the PBW-theory of $\Cc_Q$ arising from the fundamental modules.
Thus, the functor $F_Q$ directly connects the associativity of $(K(\Cc_{w_0})_t, *)$ established in Theorem \ref{Intro:adapted}, to the associativity of $(K(\Cc_Q)_t, *)$ established in Theorem \ref{Intro:simply-laced}.

When $\fg$ is of non-simply-laced type, the functor $F_\mathcal{Q}$ sends the PBW-theory of $\Cc_{w_0}$ associated with a reduced word $\ii$ adapted to a quiver into a non-standard PBW-theory arising from a collection of simple modules which are not fundamental in general.
In this case, we have the associativity of $(K(\Cc_\mathcal{Q})_t, *)$ with respect to such a non-standard PBW-theory by Theorem \ref{Intro:adapted}.  
Comparing with the HLO-isomorphism $\Phi_\mathcal{Q} \colon A_t[N(w_0)] \simeq K_t(\Cc_\mathcal{Q})$ studied in \cite{FHOO}, our deformation $(K(\Cc_\mathcal{Q})_t, *)$ with respect to such a non-standard PBW-theory gets identified with the quantum Grothendieck ring $K_t(\Cc_\mathcal{Q})$.
Thus, we obtain some evidence of a version of Conjecture \ref{conjF} with a non-standard PBW-theory when $\fg$ is of non-simply-laced type.
Note that the analog of Kazhdan-Lusztig conjecture for the category $\Cc_\mathcal{Q}$ is already verified in \cite{FHOO}.

In any case, one concludes that the functor $F_\mathcal{Q}$ together with our formalism of monoidal Jantzen filtrations gives a representation-theoretic interpretation of the HLO isomorphism $\Phi_\mathcal{Q}$ for any Q-datum $\mathcal{Q}$ for any $\fg$.   
 
\subsection*{Organization}

This paper is organized as follows.  
In Section \ref{secun}, we develop a general theory of monoidal Jantzen filtrations in the setting of monoidal abelian category of representations (modules) over an algebra.
In Section \ref{fqla}, we discuss the case of quantum loop algebras and state our main Conjecture \ref{conjF}. 
 We also provide some concrete examples of monoidal Jantzen filtrations at the end (Section \ref{exfil2}). 
In Section \ref{mjqha}, we discuss the case of quiver Hecke algebras and state the analogous conjecture.
The remaining part of the paper is devoted to the proofs of our main theorems, where we apply some geometric methods including perverse sheaves.
Before going into individual discussions, in Section \ref{sec:preliminary}, we assemble some relevant facts on equivariant perverse sheaves which we commonly use in the proofs.
Finally, we prove our main Theorems \ref{Intro:simply-laced} and \ref{Intro:adapted} above in Sections \ref{ssec:Prqaa} and \ref{ssec:PrqH} respectively.

\subsection*{Acknowledgements}
The first named author is grateful to Hironori Oya for stimulating discussion.  
 We also thank Geoffrey Janssens and Ricardo Canesin for valuable questions.  
R.\ F.\ was supported by JSPS Overseas Research Fellowships and KAKENHI Grant No.\ JP23K12955.  
D.\ H.\ was supported by the Institut Universitaire de France.

\subsection*{Overall conventions}
\begin{enumerate}
\item For a statement $P$, we set $\delta(P)$ to be $1$ or $0$ according that $P$ is true or false.
We often abbreviate $\delta(i=j)$ as $\delta_{i,j}$.
\item For an object $X$ in a category, we denote by $\id_X$ the identity morphism on $X$.  We often abbreviate it as $\id$ suppressing the subscript $X$ when it is clear from the context.
\item We write $\Z$, $\N$, $\Q$, and $\C$ for the sets of integers, non-negative integers, rational numbers, and complex numbers, respectively.
Note that we have $0 \in \N$ in our convention.
\item For a set $J$, we define $\N^{\oplus J}$ to be the subset of $\N^J$ consisting of $J$-tuples $\bd =(d_j)_{j \in J}$ with finite support, i.e., $\#  \{ j \in J \mid d_j > 0\} < \infty$. 
For each $i \in J$, let $\bdelta_i \seq (\delta_{i,j})_{j \in J} \in \N^{\oplus J}$ be the delta function. 
\end{enumerate}

\section{General definitions for monoidal categories of representations}\label{secun}

In this section we explain our general categorical framework to construct monoidal Jantzen filtrations (Definition~\ref{defjmon} and Formula~\eqref{expfm}). They depend on a PBW-theory 
in a monoidal category (Section~\ref{pbw}) and on a deformation of this PBW-theory (Section \ref{defoc}) together with $R$-matrices (Section~\ref{mjf}). We establish in general the compatibility 
of the monoidal Jantzen filtrations with specialized $R$-matrices (Propositions~\ref{compfil}, \ref{compfil2}). We explain in Section~\ref{decato} the decategorification process and the construction of analogs of Kazhdan-Lusztig polynomials. 
We conjecture that we obtain a ring through this process (Conjectures~\ref{Conj:assoc}, \ref{Conj:sassoc}). Then we establish a general duality result 
(Proposition~\ref{Prop:Kdual}) between filtrations of standard 
and costandard objects, a Kazhdan-Lusztig type characterization of a canonical 
basis and we state a Duality 
Conjecture~\ref{Conj:duality} related to the existence of a bar involution.

\subsection{PBW-theory for monoidal categories of representations}\label{pbw}

Let $A$ be an associative algebra over a field $\kk$.
In what follows, we abbreviate $\otimes_\kk$ as $\otimes$.
We assume that there is a non-trivial $\kk$-algebra homomorphism $\varepsilon \colon A \to \kk$, through which $\kk$ is regarded as an $A$-module.  
Let $B$ be an $(A, A^{\otimes 2})$-bimodule which is free of finite rank as a right $A^{\otimes 2}$-module and equipped with isomorphisms
\begin{equation} \label{eq:bimodule}
B \otimes_{A^{\otimes 2}} (B \otimes A) \simeq B \otimes_{A^{\otimes 2}} (A \otimes B)
\end{equation} 
of $(A, A^{\otimes 3})$-bimodules, and
\begin{equation} \label{eq:epsilon}
B \otimes_{A^{\otimes 2}} (A \otimes \kk) \simeq B \otimes_{A^{\otimes 2}}(\kk \otimes A) \simeq A
\end{equation} 
of $(A,A)$-bimodules, making the category of (left) $A$-modules into a $\kk$-linear monoidal category with respect to the product 
\begin{equation} \label{eq:star} 
M \star N \seq B \otimes_{A^{\otimes 2}} (M \otimes N). 
\end{equation}
Note that the category $A \mof$ of finite-dimensional left $A$-modules is stable under this monoidal structure, and the Grothendieck group $K(A\mof)$ becomes a ring with a canonical $\Z$-basis formed by the classes of finite-dimensional simple $A$-modules.  

\begin{Ex}
We mainly consider the following case.
Let $A$ be a bialgebra over $\kk$ with coproduct $\Delta \colon A \to A^{\otimes 2}$ and counit $\varepsilon \colon A \to \kk$. 
We regard $B \seq A^{\otimes 2}$ as an $(A, A^{\otimes 2})$-bimodule with the structure map $(\Delta, \id)$.
Then the product $\star$ is the ordinary tensor product of left $A$-modules. The case of 
quantum loop algebras will be of particular interest
in the following (see Section~\ref{fqla}).
\end{Ex}

We can also consider a slight generalization of the above situation.
Now, we may not assume that $A$ is unital, but we assume that there is a collection of mutually orthogonal central idempotents $\{1_{\gamma}\}_{\gamma \in \Gamma} \subset A$ labelled by a  commutative  monoid $\Gamma = (\Gamma, +)$ such that $A = \bigoplus_{\gamma \in \Gamma} A_\gamma$, where $A_\gamma \seq 1_\gamma A$. An $A$-module $M$ is always supposed to satisfy $M = \bigoplus_{\gamma \in \Gamma} 1_\gamma M$. 
Let $\varepsilon \colon A \to \kk$ be a non-trivial $\kk$-algebra homomorphism satisfying $\varepsilon(1_\gamma) = \delta_{\gamma, 0}$.
Let $B$ be an $(A,A^{\otimes 2})$-bimodule, which is $\Gamma$-graded (that is, $B = \bigoplus_{\gamma \in \Gamma} 1_\gamma B$ and $1_\gamma B = \bigoplus_{\gamma' + \gamma'' = \gamma} B(1_{\gamma'}\otimes 1_{\gamma''})$ for all $\gamma \in \Gamma$) 
and locally free of finite rank as a right $A^{\otimes 2}$-module (that is, $B (1_\gamma \otimes 1_{\gamma'})$ is free of finite rank as a right $A_\gamma \otimes A_{\gamma'}$-module for each $\gamma, \gamma' \in \Gamma$).
We assume that these are equipped with isomorphisms as in \eqref{eq:bimodule}, \eqref{eq:epsilon} making the category of (left) $A$-modules into a $\kk$-linear monoidal category with respect to the product $\star$ in \eqref{eq:star}.
Note that the category $A \mof$ in this case is a $\Gamma$-graded monoidal category, that is, we have a natural decomposition $A \mof = \bigoplus_{\gamma \in \Gamma} A_\gamma \mof$ with $(A_\gamma \mof) \star (A_{\gamma'}\mof) \subset A_{\gamma + \gamma'}\mof$.
The situation in the previous paragraph can be thought of  as  a special case where $\Gamma$ is trivial. 

\begin{Ex}
We mainly consider the following case.
Let $A = \bigoplus_{\beta \in \sQ^+} \hH_\beta$, where $\hH_\beta$ is a natural completion of the quiver Hecke algebra $H_\beta$, and $\varepsilon \colon A \to \kk$ the projection to $\hH_0 = \kk$. We take $B = \bigoplus_{\beta, \beta' \in \sQ_+}\hH_{\beta + \beta'}e(\beta, \beta')$ with a natural $(A, A^{\otimes 2})$-bimodule structure. 
Then the product $\star$ is the usual convolution product (or parabolic induction) of left $A$-modules 
(see Section~\ref{Ssec:qH} below for details). 
\end{Ex}

\begin{Rem}
One could develop our theory of monoidal Jantzen filtrations in a more general setting of an abstract monoidal abelian category with an appropriate notion of deformation.
For example, one may employ the notion of affinization in an abstract monoidal abelian category recently studied in \cite{KKOPaff}. 
\end{Rem}

Let $\Cc$ be a monoidal Serre subcategory of $A \mof$. 

\begin{Def} \label{Def:PBW}
Let $\{ L_j \}_{j \in J}$ be a collection of simple objects of $\Cc$ parameterized by a subset $J \subset \Z$, and $\preceq$ a partial ordering of the set $\N^{\oplus J}$.
We say that such a pair $(\{L_j \}_{j \in J}, \preceq)$ gives a \emph{PBW-theory} of $\Cc$ 
if the following conditions are satisfied: 
\begin{enumerate} 
\item For each $\bd = (d_j)_{j \in J} \in \N^{\oplus J}$, the \emph{oppositely} ordered product (here the ordering of $J\subset \mathbb{Z}$ is induced from the natural ordering of $\mathbb{Z}$)
\[ M(\bd) \seq \mathop{\star}^{\leftarrow}_{j \in J} L_j^{\star d_j} \]
has a simple head $L(\bd)$.
 
\item The set $\{ L(\bd)\}_{\bd \in \N^{\oplus J}}$ gives a complete collection of simple objects of $\Cc$ up to isomorphisms.
\item In the Grothendieck ring $K(\Cc)$, for each $\bd \in \N^{\oplus J}$, we have 
\[ [M(\bd)] = [L(\bd)] + \sum_{\bd' \prec \bd} P_{\bd,\bd'} [L(\bd')], \]
where $P_{\bd,\bd'} = [M(\bd):L(\bd')]\in \N$ is the Jordan-H\"older multiplicity.
\end{enumerate}
\end{Def}

We refer to the modules $M(\bd)$ as the \emph{standard modules}. 
Note that their classes $\{ [M(\bd)] \}_{\bd \in \N^{\oplus J}}$ form a $\Z$-basis of $K(\Cc)$. 
On the other hand, we also consider the naturally ordered product 
\[
M^\vee(\bd) \seq \mathop{\star}^{\to}_{j \in J} L_j^{\star d_j},
\] 
which we refer to as the \emph{constandard modules}.

\begin{Rem}
Let $(\{L_j \}_{j \in J}, \preceq)$ be a PBW-theory of $\Cc$. 
For each $i \in J$, let $\bdelta_i = (\delta_{i,j})_{j \in J} \in \N^{\oplus J}$ denote the delta function.
By definition, we have
\[ M(\bdelta_i) = M^\vee(\bdelta_i) = L(\bdelta_i) = L_i. \]
\end{Rem}

\begin{Rem}
In all the examples below, we will only encounter the situation where the partial ordering $\preceq$ of $\N^{\oplus J}$ can be taken to be the bi-lexicographic ordering.
\end{Rem}

\subsection{Generically commutative deformations of simple modules}\label{defoc}
 
For a commutative $\kk$-algebra $R$, we write $A_R \seq A \otimes R$ and $B_R \seq B \otimes R$.
Note that $B_R$ is an $(A_R, A_R \otimes_R A_R)$-bimodule.
Let $A_R \mof$ denote the category of left $A_R$-modules which are finitely generated over $R$. 
This is an $R$-linear monoidal category with respect to the product 
\begin{equation}
M \star_R N \seq B_R \otimes_{(A_R \otimes_R A_R)}(M \otimes_R N).
\end{equation}

Consider an indeterminate $z$. 
Let $\Oo \seq \kk [\![ z ]\!]$ be the ring of formal power series and $\Kk \seq \kk (\!( z )\!)$ its fraction field 
(the ring of Laurent series).
For an $\Oo$-module $M$, we write 
\begin{equation} \label{eq:M_0}
M_\Kk \seq M \otimes_\Oo \Kk \quad \text{ and } \quad M_0 \seq M \otimes_\Oo \kk.
\end{equation}
These operations give the monoidal functors
\[ A \mof \leftarrow A_\Oo \mof \to A_\Kk \mof.\]

\begin{Def} \label{Def:gcdef}
Let $\{L_j\}_{j \in J}$ be a collection of simple objects of $A \mof$ labelled by a subset $J \subset \Z$.
We say that a collection $\{ \tL_j \}_{j \in J}$ of objects of $A_\Oo \mof$ gives a \emph{generically commutative deformation} of $\{L_j\}_{j \in J}$ if the following conditions are satisfied:
\begin{itemize}
\item[(D1)] For each $j \in J$, we have $(\tL_j)_0 \simeq L_j$ and $\tL_j$ is free over $\Oo$. 
\item[(D2)] For any $i, j \in J$, we have an 
isomorphism of $A_\Kk$-modules
\[(\tL_i \star_\Oo \tL_j)_\Kk \simeq (\tL_j \star_\Oo \tL_i)_\Kk \quad\]
and an equality
\[\End_{A_\Kk}\left((\tL_i \star_\Oo \tL_j)_\Kk\right) = \Kk \id. \]
\end{itemize}
\end{Def}

Under the condition (D2), we always find an isomorphism of $A_{\Kk}$-modules
\[R_{i,j} \colon (\tL_i \star_\Oo \tL_j)_\Kk \to (\tL_j \star_\Oo \tL_i)_\Kk\] 
satisfying $R_{i,j}(\tL_i \star_\Oo \tL_j) \subset \tL_j \star_\Oo \tL_i$ and $R_{i,j}(\tL_i \star_\Oo \tL_j) \not \subset z(\tL_j \star_\Oo \tL_i)$. 
Here we naturally regard $\tL_i \star_\Oo \tL_j$ as an $\Oo$-lattice of $(\tL_i \star_\Oo \tL_j)_\Kk$.
Such a morphism $R_{i,j}$ is unique up to 
a multiple in $\Oo^\times$ and is called a \emph{renormalized $R$-matrix}.

\begin{Lem}[cf.~{\cite[5.5.4]{ES}}]
Let $\Cc$ be a monoidal Serre subcategory of $A\mof$.
If there is a PBW-theory $(\{L_j \}_{j \in J}, \prec)$ of $\Cc$ which admits a generically commutative deformation $\{\tL_j\}_{j \in J}$, the Grothendieck ring $K(\Cc)$ is isomorphic to a polynomial ring in $J$-many variables: 
\[K(\Cc) \simeq \Z[X_j \mid j \in J]; \quad [L_j] \mapsto X_j. \]
In particular, $K(\Cc)$ is a commutative ring. 
\end{Lem}
\begin{proof}
For any $M \in A_\Kk \mof$ and $A_\Oo$-lattices $N, N' \subset M$, we have $[N_0] = [N'_0]$ in $K(A \mof)$ (cf.~\cite[Lemma 2.3.4]{CG}). 
Applying this fact to the case when $M = (\tL_j \star_\Oo \tL_i)_\Kk$, $N= R_{ij}(\tL_i \star_\Oo \tL_j)$ and $N' = \tL_j \star_\Oo \tL_i$, we find $[L_i \star L_j] = [L_j \star L_i]$ for any $i, j \in I$.
Since $\{ [M(\bd)]\}_{\bd \in \N^{\oplus J}}$ forms a $\Z$-basis of $K(\Cc)$, we obtain the assertion.
\end{proof}

Let $(\{L_j\}_{j \in J}, \preceq)$ be a PBW-theory of a monoidal Serre subcategory $\Cc \subset A \mof$.
Assume $\{L_j \}_{j \in J}$ admits a generically commutative deformation $\{\tL_j \}_{j \in J}$.   
Then for any $d \in \N$ and any sequence $\bep = (\ep_1, \ldots, \ep_d) \in J^d$, we define the \emph{mixed product} $M(\bep) \in \Cc$ and its deformation $\tM(\bep) \in A_\Oo \mof$ by
\[ M(\bep) \seq L_{\ep_1} \star \cdots \star L_{\ep_d} \quad \text{ and } \quad \tM(\bep) \seq \tL_{\ep_1} \star_\Oo \cdots \star_\Oo \tL_{\ep_d}.\]
By definition, we have $\tM(\bep)_0 = M(\bep)$. 

\begin{Def}
For $J \subset \Z$ and $\bd =(d_j)_{j \in J} \in \N^{\oplus J}$, we set
\begin{equation}
J^{\bd} \seq \{ \bep = (\ep_1, \ldots, \ep_{d}) \in J^{d} \mid  \#\{ k \mid \ep_k = j \} = d_j, \forall j \in J \},
\end{equation}
where $d \seq \sum_{j \in J} d_j$.
A sequence $\bep = (\ep_1, \ldots, \ep_{d})\in J^\bd$ is said to be \emph{standard} (resp.~\emph{costandard}) if it satisfies $\ep_1 \ge \cdots \ge \ep_{d}$ (resp.~$\ep_1 \le \cdots \le \ep_{d}$). 
Given $\bd \in \N^{\oplus J}$, there is a unique standard (resp.~costandard) sequence in $J^\bd$, which we often denote by $\bep_s = \bep_s(\bd)$ (resp.~$\bep_c = \bep_c(\bd)$).    
By definition, we have $M(\bep_s) = M(\bd)$ and $M(\bep_c) = M^\vee(\bd)$. 
\end{Def} 

\subsection{Intertwiners arising from $R$-matrices}
Let $\{ L_j \}_{j \in J}$ be a collection of simple modules in $\Cc$ labelled by a set $J \subset \Z$, and $\{ \tL_j \}_{j \in J}$ its generically commutative deformation.
For any pair $(i,j) \in J^2$, we have a unique non-negative integer $\alpha(i,j)$ satisfying 
\begin{equation} \label{eq:dd}
R_{i,j} \circ R_{j,i} \equiv z^{\alpha(i,j)} \id \mod \Oo^\times
\end{equation} 
by the condition~(D2) in Definition~\ref{Def:gcdef}.
Note that 
$$\alpha(i,j) = \alpha(j,i) \quad \text{ and } \quad \alpha(i,i) = 0$$ 
hold. We have the following three cases: 
\begin{itemize}
\item[(i)] $\alpha(i,j) = 0$;
\item[(ii)] $\alpha(i,j) > 0$ and $i > j$;
\item[(iii)] $\alpha(i,j) > 0$ and $i < j$.
\end{itemize}
We say that the renormalized $R$-matrix $R_{i,j}$ (or an isomorphism of the form $\id \star R_{i,j} \star \id$) is \emph{neutral} (resp.\ \emph{positive, negative}) when the above condition (i) (resp.\ (ii), (iii)) is satisfied.

\begin{Def}
Let $\bd \in \N^{\oplus J}$.
For $\bep, \bep' \in J^\bd$, we write 
$\bep \lesssim \bep'$ if $\bep'$ is obtained from $\bep$ by replacing a consecutive pair $(i,j)$ in $\bep$ satisfying either (i) or (ii) (see above) with the opposite pair $(j, i)$. 
It generates a preorder on the set $J^\bd$, which we denote by the same symbol $\lesssim$. 
Let $\sim$ denote the induced equivalence relation on $J^\bd$. 
In other words, for $\bep, \bep' \in J^\bd$, 
we write 
$\bep \sim  \bep'$ if and only if $\bep \lesssim \bep'$ and $\bep' \lesssim \bep$.
\end{Def}

Now let us assume our generically commutative deformation is consistent in the following sense.

\begin{Def} A generically commutative deformation $\{ \tL_j \}_{j \in j}$ of $\{L_j\}_{j \in J}$ is said to 
be \emph{consistent} if 
\begin{itemize}
\item[(D3)] \label{consistency} For $i < j <k$, we have the \emph{quantum Yang-Baxter relation}:
\begin{equation} \label{eq:YB}
(R_{j,k} \star_\Oo \id) \circ (\id \star_\Oo R_{i,k})\circ (R_{i,j} \star_\Oo \id)  \equiv (\id \star_\Oo R_{i,j}) \circ (R_{i,k}\star_\Oo \id) \circ (\id \star_\Oo R_{j,k}) \mod \Oo^\times
\end{equation}
as morphisms from $\tL_i \star_\Oo \tL_j \star_\Oo \tL_k$ to $\tL_k \star_\Oo \tL_j \star_\Oo \tL_i$.
\end{itemize}
\end{Def}

\begin{Rem} \label{Rem:negYB}
Thanks to \eqref{eq:dd}, the above consistency condition ( D3) ensures the quantum Yang-Baxter relation \eqref{eq:YB} holds for any triple $(i,j,k)$ in $J$. 
For example, if we multiply by $(R_{j,k} \star_\Oo \id)^{-1} \equiv z^{-\alpha(j,k)}(R_{k,j} \star_\Oo \id)$ from the left and by $(\id \star_\Oo R_{j,k})^{-1} \equiv z^{-\alpha(j,k)} (\id \star_\Oo R_{k,j})$ from the right to the relation~\eqref{eq:YB}  with $i < j < k$, we obtain the quantum Yang-Baxter relation for the triple $(i,k,j)$.
\end{Rem}

Assume that $\{ \tL_j \}_{j \in J}$ is a consistent generically commutative deformation of $\{ L_j \}_{j \in J}$.
Let $\bd \in \N^{\oplus J}$ and $\bep, \bep' \in J^\bd$.
When $\bep \lesssim \bep'$ (resp.~$\bep' \lesssim \bep$), we can consider the $A_\Kk$-isomorphism
\begin{equation}
R_{\bep',\bep} \colon \tM(\bep)_\Kk \to \tM(\bep')_\Kk
\end{equation}
obtained by composing the neutral or positive (resp.~negative) renormalized $R$-matrices.
Thanks to the quantum Yang-Baxter relation \eqref{eq:YB} and Remark~\ref{Rem:negYB}, it is well-defined up to multiples in $\Oo^\times$.
If $\bep \lesssim \bep' \lesssim \bep''$ or $\bep'' \lesssim \bep' \lesssim \bep$, we have
$$R_{\bep'', \bep'} \circ R_{\bep', \bep} \equiv R_{\bep'', \bep} \mod \Oo^\times.$$
In particular, we obtain the following.
\begin{Prop}\label{riso1} If $\bep \sim \bep'$, the homomorphism $R_{\bep', \bep}$ induces isomorphisms
\[
\tM(\bep) \simeq \tM(\bep') \quad \text{ and } \quad M(\bep) \simeq M(\bep').
\] 
\end{Prop}

\subsection{Monoidal Jantzen filtrations}\label{mjf}
In what follows, let $(\{ L_j \}_{j \in J}, \preceq)$  be  a PBW-theory of a monoidal Serre subcategory $\Cc \subset A\mof$ and $\{ \tL_j \}_{j \in J}$ a consistent generically commutative deformation of $\{ L_j \}_{j \in J}$.
Fix $\bd \in \N^{\oplus J}$ and write $\bep_s$ and $\bep_c$ for the standard and costandard sequences in $J^\bd$ respectively. 
For any $\bep \in J^\bd$, we have $\bep_s \lesssim \bep \lesssim \bep_c$ and hence the $A_\Kk$-isomorphisms
\[
\tM(\bep_s)_\Kk \xrightarrow{R_{\bep, \bep_s}} \tM(\bep)_\Kk \xrightarrow{R_{\bep_c, \bep}} \tM(\bep_c)_\Kk
\]
constructed in the previous subsection. 
We regard $\tM(\bep)$ as an $\Oo$-lattice of $\tM(\bep)_\Kk$.
Now, we define the decreasing filtration of $A$-submodules
\begin{equation} \label{eq:Filt}
M(\bep) \supset \cdots \supset F_{-1}M(\bep) \supset F_0M(\bep) \supset F_1M(\bep) \supset \cdots 
\end{equation}  
by the formula \begin{equation}\label{expfm} F_nM(\bep) \seq \ev_{z=0} \left( \tM(\bep) \cap \sum_{k \in \Z} \left( z^k R_{\bep, \bep_s} \tM(\bep_s) \cap z^{n-k} R^{-1}_{\bep_c, \bep} \tM(\bep_c)\right)\right)\end{equation}
for each $n \in \Z$, where $\ev_{z=0} \colon \tM(\bep) \to \tM(\bep)_0 = M(\bep)$ is the natural evaluation map (recall \eqref{eq:M_0}). 
By construction, we have $F_{-n}M(\bep) = M(\bep)$ and $F_nM(\bep) = \{0\}$ for $n$ large enough.

\begin{Def}\label{defjmon} We call the filtration $F_\bullet M(\bep) = \{ F_n M(\bep) \}_{n \in \Z}$ in \eqref{eq:Filt} the \emph{(monoidal) Jantzen filtration} of $M(\bep)$. 
\end{Def}

\begin{Ex} \label{Ex:standard}
When $\bep = \bep_s$, we have $R_{\bep, \bep_s} \in \Oo^\times \id$ and hence
\[ F_n M(\bep_s) = \ev_{z=0} \left(\tM(\bep_s) \cap z^n R^{-1}_{\bep_c, \bep_s} \tM(\bep_c) \right)\]
for each $n \in \Z$. 
The filtration $F_\bullet M(\bd)$ of the standard module $M(\bd) = M(\bep_s)$ is given in this way, which is analogous to the usual Jantzen filtration of standard (Verma) modules of Lie algebras.

Dually, when $\bep = \bep_c$, we have $R_{\bep_c, \bep} \in \Oo^\times \id$ and hence
\[ F_n M(\bep_c) = \ev_{z=0} \left(z^n R_{\bep_c, \bep_s} \tM(\bep_s) \cap \tM(\bep_c) \right)\]
for each $n \in \Z$.
The filtration $F_\bullet M^\vee(\bd)$ of the costandard module $M^\vee(\bd) = M(\bep_c)$ is given in this way.
\end{Ex}

\subsection{Specialized $R$-matrices}\label{rnomat}
We keep the assumption from the previous subsection.
Suppose that $\bep \lesssim \bep'$ or $\bep' \lesssim \bep$. 
Then $R_{\bep',\bep}$ is defined and there is a unique integer $\beta(\bep',\bep)\geq 0$ so that 
$$R_{\bep',\bep} \tM(\bep) \subset z^{\beta(\bep',\bep)}\tM(\bep')\quad \text{ and } \quad 
R_{\bep',\bep}\tM(\bep) \not \subset z^{\beta(\bep',\bep) + 1} \tM(\bep').$$
Note that if $R_{\bep',\bep}$ is of the form $\id \star R_{i,j}\star \id$, then $\beta(\bep',\bep) = 0$.
Under the same assumption, there is also a unique integer $\alpha(\bep',\bep)\geq 0$ such that
\begin{equation}\label{sldef}
R_{\bep',\bep} \circ R_{\bep,\bep'} \equiv z^{\alpha(\bep',\bep)}\id_{\tM(\bep')} 
\quad \text{ and } \quad 
R_{\bep,\bep'} \circ R_{\bep',\bep} \equiv z^{\alpha(\bep',\bep)}\id_{\tM(\bep)} \mod \Oo^\times.\end{equation}
These numbers satisfy the following properties:
\begin{enumerate}
\item By definition, we have $\alpha(\bep, \bep') = \alpha(\bep', \bep)$;
\item Recall the notation $\alpha(i,j)$ for $i,j \in J$ from the previous section.
When $R_{\bep', \bep}$ is the composition of homomorphisms of the form $\id \star R_{i_k, j_k} \star \id$ for $1 \le k \le n$, we have $\alpha(\bep', \bep) = \sum_{k=1}^n \alpha(i_k, j_k).$
In particular, if $\bep \lesssim \bep' \lesssim \bep''$, we have the additivity $\alpha(\bep'', \bep) = \alpha(\bep'', \bep') + \alpha(\bep', \bep)$;
\item We have $\bep \sim \bep'$ if and only if $\alpha(\bep, \bep') = 0$ (case of Proposition \ref{riso1});
\item If $\bep \lesssim \bep' \lesssim \bep''$, we have $\beta(\bep'', \bep') + \beta(\bep', \bep) \le \beta(\bep'', \bep)$ and $\beta(\bep, \bep') + \beta(\bep', \bep'') \le \beta(\bep, \bep'')$; 
\item For $\bep \lesssim \bep'$, we have $\alpha(\bep, \bep') \ge \beta(\bep, \bep') + \beta(\bep', \bep)$.
In particular, we always have $\alpha(\bep, \bep') - \beta(\bep, \bep') \ge 0$.
\end{enumerate}

Now $z^{-\beta(\bep',\bep)}R_{\bep',\bep}$ induces a non-zero morphism  of $A$-modules $$\rr_{\bep',\bep} \colon M(\bep) \rightarrow M(\bep')$$
called the \emph{specialized $R$-matrix}, which is uniquely determined up to an invertible element in $\kk$. 
The following propositions are useful to compute examples.

\begin{Prop}\label{compfil} Let $\bep, \bep' \in J^\bd$ satisfying $\bep \lesssim \bep'$. 
For any $N\in\mathbb{Z}$, we have  
$$\mathbf{r}_{\bep',\bep}(F_N M(\bep))\subset F_{N-2\beta(\bep',\bep)} M(\bep').$$
\end{Prop}

\begin{proof} We have the following commutative (up to multiples in $\Oo^\times$) diagram
\begin{equation}
\vcenter{\xymatrix{\tM(\bep_s)_\Kk\ar[dd]_{R_{\bep', \bep_s}}\ar[rr]^{R_{\bep, \bep_s}} && \tM(\bep)_\Kk  \ar[lldd]^{R_{\bep',\bep}}\ar[dd]^{R_{\bep_c, \bep}}
\\
\\\tM(\bep')_\Kk \ar[rr]^{R_{\bep_c, \bep'}} && \tM(\bep_c)_\Kk }}.
\end{equation}
Let $y(z)\in z^k R_{\bep, \bep_s} \tM(\bep_s) \cap z^{N-k} R^{-1}_{\bep_c, \bep} \tM(\bep_c)$ with $k\in\mathbb{Z}$. Then
$$y(z) = R_{\bep, \bep_s}(z^k x(z)) \quad \text{ and } \quad R_{\bep_c, \bep}(y(z)) = z^{N-k} x'(z),$$
for some $x(z)\in \tM(\bep_s)$ and $x'(z)\in \tM(\bep_c)$. 
Then $y'(z) = z^{-\beta(\bep',\bep)}R_{\bep',\bep}(y(z))\in \tM(\bep')$
satisfies
$$y'(z) = z^{-\beta(\bep',\bep)}(R_{\bep',\bep}\circ R_{\bep, \bep_s})(z^k x(z)) = R_{\bep', \bep_s}(z^{k-\beta(\bep',\bep)} x(z))$$
and 
$$R_{\bep_c, \bep'}(y'(z)) = z^{-\beta(\bep',\bep)}R_{\bep_c, \bep}(y(z)) = z^{N-k-\beta(\bep',\bep)} x'(z)$$
up to multiples in $\Oo^\times$.
Then the result follows from
\begin{equation}
y'(z) \in z^{k-\beta(\bep',\bep)}  R_{\bep', \bep_s}\tM(\bep_s) \cap z^{N-k-\beta(\bep',\bep)} R^{-1}_{\bep_c, \bep'}\tM(\bep_c).
\qedhere
\end{equation}
\end{proof}

\begin{Prop}\label{compfil2} 
Let $\bep, \bep' \in J^\bd$ satisfying $\bep' \lesssim \bep$. For any  $N\in\mathbb{Z}$, we have  
$$\mathbf{r}_{\bep',\bep}(F_N M(\bep))\subset F_{N + 2\alpha(\bep',\bep)-2\beta(\bep',\bep)} M(\bep').$$
\end{Prop}

\begin{proof} We have the  same diagram as in the proof of Proposition \ref{compfil} 
and we consider $y(z)$, $x(z)$, $x'(z)$, $y'(z)$ in the same way. But now we have 
$$y'(z) = z^{-\beta(\bep',\bep)}(R_{\bep',\bep}\circ R_{\bep, \bep_s})(z^k x(z)) = z^{\alpha(\bep',\bep)-\beta(\bep',\bep)} R_{\bep', \bep_s}(z^k x(z)), $$
$$R_{\bep_c, \bep'}(y'(z)) = z^{\alpha(\bep',\bep)-\beta(\bep',\bep)} R_{\bep_c, \bep}(y(z)) =  z^{N-k + \alpha(\bep',\bep)-\beta(\bep',\bep)} x'(z)$$
up to multiples in $\Oo^\times$. Then the result follows from 
$$y'(z) \in z^{k+\alpha(\bep',\bep)-\beta(\bep',\bep)} R_{\bep', \bep_s}\tM(\bep_s) \cap z^{N-k+\alpha(\bep',\bep)-\beta(\bep',\bep) } R^{-1}_{\bep_c, \bep'}\tM(\bep_c). 
\qedhere$$
\end{proof}

\begin{Rem} \label{Rem:eed}
For $\bd \in \N^{\oplus J}$, we set 
$$\beta(\bd) \seq \beta(\bep_c(\bd), \bep_s(\bd)),$$ 
where $\bep_s(\bd)$ and $\bep_c(\bd)$ are the standard and costandard sequences in $J^\bd$ respectively.
By Example \ref{Ex:standard}, we have 
\begin{equation}
M(\bd) =  F_{\beta(\bd)} M(\bd) \supsetneq F_{\beta(\bd) + 1} M(\bd). 
\end{equation}
In particular, the simple head $L(\bd)$ of $M(\bd)$ contributes to $F_nM(\bd)/F_{n+1}M(\bd)$ as a composition factor if and only if $n = \beta(\bd)$.
\end{Rem}

\subsection{Decategorification}\label{decato}

We keep the same assumption from the previous subsections. 
Let $t$ be another indeterminate with a formal square root $t^{1/2}$.
Consider the $\Z[t^{\pm 1/2}]$-module 
\[K(\Cc)_t \seq K(\Cc) \otimes_\Z \Z[t^{\pm 1/2}] = \bigoplus_{\bd \in \N^{\oplus J}} \Z[t^{\pm 1/2}] [L(\bd)],\]
where we abbreviate $[M] \otimes 1$ as $[M]$.
For each $\bep \in J^d$ with $d \in \N$, using the Jantzen filtration \eqref{eq:Filt}, we define an element $[M(\bep)]_t \in K(\Cc)_t$ by 
\begin{equation} \label{eq:Mt} [M(\bep)]_t \seq \sum_{n \in \Z} [\Gr_n^F M(\bep)] t^n, 
\end{equation}
where $\Gr_n^F M(\bep) \seq F_nM(\bep) / F_{n+1}M(\bep)$. 
As a special case, for each $\bd \in \N^{\oplus J}$, we have defined the element $[M(\bd)]_t = [M(\bep_s(\bd))]_t$.
By Definition~\ref{Def:PBW}, it comes with the relation
\begin{equation} \label{eq:KtL}
[M(\bd)]_t = 
t^{\beta(\bd)}\left([L(\bd)] + \sum_{\bd' \prec \bd} P_{\bd,\bd'}(t) [L(\bd')] \right),
\end{equation}
where $\beta(\bd) \seq \beta(\bep_c(\bd), \bep_s(\bd)) \in \N$ is as in Remark \ref{Rem:eed}, and $P_{\bd,\bd'}(t) \in \N[t]$ is an analog of Kazhdan-Lusztig polynomial defined by
\[ P_{\bd,\bd'}(t) \seq t^{-\beta(\bd)}\sum_{n \in \Z} [\Gr^F_n M(\bd) : L(\bd')] t^n.\]
Here $[M:L]$ denotes the Jordan-H\"older multiplicity of $L$ in $M$. 
Then, by definition, we have 
\[P_{\bd,\bd'}(1) = P_{\bd,\bd'} = [M(\bd):L(\bd')] \quad \text{ for any $\bd'\prec \bd$.}\]  
Note that $\{ [M(\bd)]_t\}_{\bd \in \N^{\oplus J}}$ forms a $\Z[t^{\pm 1/2}]$-basis of $K(\Cc)_t$ by the relation \eqref{eq:KtL}.

Let $\gamma \colon \N^{\oplus J} \times \N^{\oplus J} \to \frac{1}{2}\Z$ be a skew-symmetric bilinear map.
With the above notation, we define a $\Z[t^{\pm 1/2}]$-bilinear operation $* = *_\gamma$ on $K(\Cc)_t$ in terms of the basis $\{ [M(\bd)]_t\}_{\bd \in \N^{\oplus J}}$ by 
\[ [M(\bd)]_t * [M(\bd')]_t \seq t^{\gamma(\bd,\bd')} [M(\bd) \star M(\bd')]_t, \]
where the RHS is given by \eqref{eq:Mt} with $M(\bep) = M(\bd) \star M(\bd')$ (that is, $\bep \in J^{\bd + \bd'}$ is the concatenation of two standard sequences $\bep_s(\bd)$ and $\bep_s(\bd')$, which is not necessarily standard).
 Be aware that the operation $*$ depends on many choices: a PBW-theory $(\{ L_j\}_{j \in J}, \preceq)$, its consistent, generically commutative deformation $\{ \tL_j\}_{j \in J}$, and a bilinear form $\gamma$.

We may expect the associativity of $*$, but it seems unclear from the construction.
We state it as our general conjecture.

\begin{conj}[Associativity Conjecture] \label{Conj:assoc} The $\Z[t^{\pm 1/2}]$-module $K(\Cc)_t$ with this operation $*$ is a $\Z[t^{\pm 1/2}]$-algebra, and so it gives a (not necessarily commutative) $t$-deformation of the Grothendieck ring $K(\Cc)$. 
\end{conj}

We also write a stronger version of the above Conjecture.
For each integer $n \ge 2$, we consider the $\Z[t^{\pm 1/2}]$-multilinear operation $m_n \colon K(\Cc)_t^{n} \to K(\Cc)_t$ given by
\[ 
m_n\left([M(\bd_1)]_t, \ldots , [M(\bd_n)]_t\right) \seq  t^{\sum_{1 \le k < l \le n}\gamma(\bd_k, \bd_l)}[M(\bd_1)\star \cdots \star M(\bd_n)]_t
\]
for $\bd_1, \ldots, \bd_n \in \N^{\oplus J}$.
Of course, we have $m_2(x,y) = x*y$.
By convention, we set $m_1 \seq \id$.

\begin{conj}[Strong Associativity Conjecture]
\label{Conj:sassoc}
For any integers $n \ge 3$ and $0 < k < n$, we have
\[ m_n(x_1, \ldots, x_n) = m_{k}(x_1,\ldots,x_k)*m_{n-k}(x_{k+1}, \ldots, x_n)\]
for any $x_1, \ldots, x_n \in K(\Cc)_t$.
\end{conj}

\begin{Rem}\label{remmulti}
If Conjecture~\ref{Conj:sassoc} holds, then Conjecture~\ref{Conj:assoc} also holds and moreover, for any $d \in \N$ and any sequence $\bep = (\ep_1, \ldots, \ep_d) \in J^d$, we have
\begin{equation} \label{eq:*mix}
[L_{\ep_1}] * \cdots * [L_{\ep_d}] = t^{\sum_{1 \le k < l \le d}\gamma(\bdelta_{\ep_k}, \bdelta_{\ep_l})}[M(\bep)]_t, 
\end{equation}
where $\bdelta_i \in \N^{\oplus J}$ denotes the delta function.
Note that the converse is true.
Namely, Conjecture \ref{Conj:sassoc} holds if and only if Conjecture \ref{Conj:assoc} and \eqref{eq:*mix} hold for any $d \in \N$ and $\bep \in J^d$. 
\end{Rem}

\subsection{Bar-involution and normality}\label{bars}
Let $\{L_j\}_{j \in J}$ be a PBW-theory for a monoidal Serre subcategory $\Cc \subset A\mof$ which admits a consistent, generically commutative deformation $\{\tL_j\}_{j \in J}$ as above.
We have the following general fact.

\begin{Prop} \label{Prop:Kdual}
For each $\bd \in \N^{\oplus J}$ and $n \in \Z$, we have an isomorphism of $A$-modules: 
\[ \Gr_n^F M(\bd) \simeq \Gr_{-n}^{F} M^\vee(\bd).\]
\end{Prop}
\begin{proof}
Recall Example~\ref{Ex:standard}. 
Let $\bep_s, \bep_c \in J^\bd$ be the standard and costandard sequences respectively. 
For brevity, we write $M \seq \tM(\bep_s)$, $N \seq \tM(\bep_c)$ and $R = R_{\bep_c, \bep_s} \colon M_\Kk \simeq N_\Kk$.
Then, we have isomorphisms of $A$-modules:
\begin{align}
\Gr_n^F M(\bd) &\simeq \frac{M\cap z^{n}R^{-1}N}{(M \cap z^{n+1} R^{-1} N) + (zM\cap z^nR^{-1} N)} \\
&\simeq \frac{z^{-n}RM\cap N}{(z^{-n}RM \cap z N) + (z^{-n+1}RM\cap N)} \simeq \Gr_{-n}^{F} M^\vee(\bd),
\end{align}
where the second isomorphism is induced by the isomorphism $z^{-n}R$.
\end{proof}

Let $\ol{(\cdot)} \colon K(\Cc)_t \to K(\Cc)_t$ be the involution of abelian group given by 
\[  \ol{t^{n}[L(\bd)]} \seq t^{-n}[L(\bd)] \]
for any $n \in \frac{1}{2}\Z$ and $\bd \in \N^{\oplus J}$.
The following is an immediate consequence of Proposition \ref{Prop:Kdual}. 

\begin{Cor}\label{scase}
For each $\bd \in \N^{\oplus J}$, we have
\[ [M^\vee(\bd)]_t = t^{-\beta(\bd)}\left([L(\bd)] + \sum_{\bd' \prec \bd} P_{\bd,\bd'}(t^{-1}) [L(\bd')] \right) = \ol{[M(\bd)]_t}.\]
In particular, for any $i,j \in J$, we have
\[ \ol{[L_i] * [L_j]} = [L_j] * [L_i].\]
\end{Cor}

\begin{Rem}  Note that we need the map $\gamma$ to be skew-symmetric for the second assertion of the above Corollary \ref{scase}. This justifies this condition on $\gamma$, which will be satisfied in all examples below.
\end{Rem}

As a generalization, we also conjecture the following.

\begin{conj}[Duality Conjecture] \label{Conj:duality}
For any $d \in \N$ and $\bep = (\ep_1, \ldots, \ep_d) \in J^d$, we have
\[ \ol{[M(\bep)]_t} = [M(\bep^\op)]_t,\]
where $\bep^\op = (\ep_d, \ldots, \ep_1)$ is the opposite sequence. 
\end{conj} 

Note that, if both the Strong Associativity Conjecture (= Conjecture~\ref{Conj:sassoc}) and the Duality Conjecture (= Conjecture~\ref{Conj:duality}) are true, the involution $\ol{(\cdot)}$ defines an anti-algebra involution of $(K(\Cc)_t, *)$.

We  also introduce  the notion of normality, following  \cite[Definition 2.5]{KK19}.
\begin{Def}[Normality]
We say that our deformation $\{ \tL_j\}_{j \in J}$ as above is \emph{normal} if 
\begin{itemize}
\item[(N1)] we have $\beta(\bd) = 0$ for each $\bd \in \N^{\oplus J}$ (cf.~Remark~\ref{Rem:eed}), and 
\item[(N2)] $\Gr^F_0 M(\bd) \simeq L(\bd)$ for each $\bd \in \N^{\oplus J}$. 
\end{itemize}
\end{Def}

If $\{ \tL_j \}_{j \in J}$ satisfies the condition (N1), the non-zero homomorphism 
\[\rr_\bd \seq \rr_{\bep_c(\bd),\bep_s(\bd)} \colon M(\bd) \to M^\vee(\bd) \]
is induced by $R_{\bep_c, \bep_s}$ (no rescaling here).
Therefore, assuming (N1), the condition (N2) is equivalent to the condition
\begin{itemize} 
\item[$($N2$)'$] $\Image(\rr_\bd) \simeq L(\bd)$ for each $\bd \in \N^{\oplus J}$. 
\end{itemize}
Note that this condition $($N2$)'$ is automatically satisfied if $M^\vee(\bd)$ has a simple socle  isomorphic to $L(\bd)$  for each $\bd \in \N^{\oplus J}$. 
If $\{ \tL_j \}_{j \in J}$ is normal, we have
\begin{equation}\label{pos1}
P_{\bd,\bd'}(t) \in t \N[t]
\end{equation}
for any $\bd' \prec \bd$.
Therefore, we obtain the following Kazhdan-Lusztig type characterization of $\{[L(\bd)]\}_{\bd \in \N^{\oplus J}}$ that can be seen as a canonical basis of $K(\Cc)_t$.

\begin{Prop}
Assume that our deformation $\{\tL_j\}_{j \in J}$ as above is normal.
Then, the $\Z[t^{\pm 1/2}]$-basis $\{[L(\bd)]\}_{\bd \in \N^{\oplus J}}$ of $K(\Cc)_t$ is characterized by the following two properties:
\begin{enumerate}
\item $\ol{[L(\bd)]} = [L(\bd)]$, and
\item $[L(\bd)] - [M(\bd)]_t \in \sum_{\bd' \prec \bd}t\Z[t] [M(\bd')]_t$.
\end{enumerate}
\end{Prop}

\subsection{Proof of Theorem \ref{thmapp}}\label{pfadd}
We use the notations as in Theorem \ref{thmapp}. It is proved in the following way.

 By definition of $S$ and $Q$, we have
$$[M]_t = t^r Q + t^{r+1}M_{r+1} + \cdots + t^{s-1}M_{s-1} + t^s S$$
for integers $r \leq s$ and classes $M_l\in  K(\mathscr{C})$ defined for $r < l < s$.
In particular, the monoidal Jantzen filtration
$$F_\bullet M\colon \quad  M \supset \cdots \supset F_{-1}M \supset F_0M \supset F_1M \supset \cdots \supset \{ 0 \}$$
satisfies
$$F_{s+1} M = 0 \text{ and }[F_s M] = S,$$
$$F_{r} M = M \text{ and }[F_{r+1} M] + Q = [M].$$
This implies that $F_sM$ is a submodule of $M$ whose image in $K(\mathscr{C})$ is $S$. Also 
$$M/(F_{r+1} M)$$ 
is a quotient of $M$ whose image in $K(\mathscr{C})$ is $Q$.


\section{Monoidal Jantzen filtrations for quantum loop algebras}\label{fqla}

We study our first main examples for monoidal Jantzen filtrations, 
realized in categories of finite-dimensional representations 
of quantum loop algebras. 
More precisely, we first give general reminders on these representations. Then we introduce the ordinary PBW-theory arising from fundamental modules (Theorem \ref{fundpbw}) and more general PBW-theories of affine cuspidal modules from \cite{KKOP}. We recall the relevant $R$-matrices, we 
introduce relevant deformations of the PBW-theory in Section \ref{defocons} and we check it fits into our general framework (Theorem \ref{normt}). 
Hence we obtain monoidal Jantzen filtrations.
Independently, we recall the construction of quantum Grothendieck rings, the 
corresponding Kazhdan-Lusztig polynomials which are now known to be positive 
(Theorem \ref{pospol}). 
We conjecture that our 
decategorified monoidal Jantzen filtrations recover this 
quantum Grothendieck ring (Conjectures \ref{qgrconj}, \ref{qgrmix}). So this gives an explanation for the positivity of Kazhdan-Lusztig polynomials in this context.

\subsection{Quantum loop algebras and their representations}
Let $\fg$ be a complex finite-dimensional simple Lie algebra.
Let $C = (c_{ij})_{i,j \in I}$ denote the Cartan matrix of $\fg$, where $I$ is the set of Dynkin nodes. 
Let $r \in \{1,2,3\}$ be the lacing number of $\fg$, and $(r_i)_{i \in I} \in \{1,r\}^I$ the left symmetrizer of $C$, i.e., satisfying $r_i c_{ij} = r_j c_{ji}$ for all $i, j \in I$.

Let $U_q(L\fg)$ be  the  quantum loop algebra associated to $\fg$.
It is a Hopf algebra defined over an algebraic closed field $\kk = \ol{\Q(q)}$, where $q$ is a formal parameter. It has a family of Chevalley generators $e_i, f_i, k_i^{\pm 1}$ where $i\in I \sqcup\{0\}$.
 In this paper, we use the coproduct $\Delta$ given by 
\[\Delta(e_i) = e_i \otimes k_i^{-1} + 1 \otimes e_i, \quad 
\Delta(f_i) = f_i \otimes 1 + k_i \otimes f_i, \quad 
\Delta(k_i^{\pm 1}) = k_i^{\pm 1} \otimes k_i^{\pm 1}\]
for each $i\in I \sqcup\{0\}$.

Let $\Cc$ denote the rigid monoidal category of finite-dimensional $U_q(L\fg)$-modules (with the standard type $1$ condition).   
Recall that the isomorphism classes of simple modules of the category $\Cc$ are parameterized by the set $(1+z\kk[z])^I$ of $I$-tuples of monic polynomials (the Drinfeld polynomials). Such a $I$-tuple encode the eigenvalues of distinguished operators on a highest weight vector of the 
simple representation \cite{CP}.

We will focus on the monoidal subcategory $\Cc_\Z$ of $\Cc$ introduced by Hernandez-Leclerc~\cite{HL10} and so that every prime simple module of $\Cc$ (that is every simple module which can not be factorized into a non trivial tensor product of modules) is in $\Cc_\Z$ after a suitable spectral parameter shift. Precisely, we fix a parity function $\varepsilon \colon I \to \{0,1\}$ satisfying the condition 
\begin{equation} \label{eq:ep}
\varepsilon_i \equiv \varepsilon_j + \min(r_i, r_j) \pmod 2 \quad \text{if $c_{ij}<0$}, 
\end{equation} 
and let $$\hI \seq \{(i,p) \in I \times \Z \mid p \equiv \varepsilon_i \pmod 2 \}.$$
 We introduce a formal variable $Y_{i,p}$ for each $(i,p) \in \hI$ 
 and $\cM$ be the group of all the Laurent monomials
\begin{equation} \label{eq:monomial}
m = \prod_{(i,p) \in \hI} Y_{i,p}^{u_{i,p}(m)}.
\end{equation}
We say that $m \in \cM$ is \emph{dominant} if $u_{i,p}(m) \ge 0$ for all $(i,p) \in \hI$, and denote the set of dominant monomials by $\cM^+$. 
For each such dominant monomial, we have a simple module $L(m) \in \Cc$ corresponding to the Drinfeld polynomials $(\prod_{p}(1-q^pz)^{u_{i,p}(m)})_{i \in I}$.
The category $\Cc_\Z$ is defined to be the Serre subcategory of $\Cc$ generated by these simple modules.
It is closed under taking tensor products and left/right duals. In other words, $\Cc_\Z$ is a rigid monoidal subcategory of $\Cc$. 

\subsection{Standard modules and PBW-theory}\label{spbwt}

For $(i,p)\in\hI$, consider the fundamental representation defined by 
\[V_{i,p} = L(Y_{i,p}).\] 
We choose a numbering $I = \{1,\cdots, n\}$ where $n$ is the rank of $\fg$ and we define an embedding $e \colon \hI\rightarrow \mathbb{Z}$ by setting 
\begin{equation} \label{eq:e}
e \colon (i,p)\mapsto i + np.
\end{equation}
This induces an ordering on $\hI$ so that $p < q$ implies $(i,p) < (j,q)$. We will just denote $V_{e(i,p)} = V_{i,p}$ so that we have a family of simple modules $\{V_j\}_{j\in J}$ as in 
Section \ref{pbw}, where $J$ is the image $e(\hI)\subset \mathbb{Z}$.

\begin{Rem}
More generally, one can work with any embedding $e \colon \hI \to \Z$ satisfying the condition
\[ e(i,p) < e(j,s) \quad \text{if $\fo(V_{i,p}, V_{j,s}) > 0$, }\]
where the number $\fo(M,N) \in \N$ is the pole order of the normalized $R$-matrix defined below.  
It follows that the resulting deformed product $*$ on $K(\Cc_\Z)_t$ does not depend on the choice of such an embedding at least when $\fg$ is of simply-laced type from the proof of Theorem \ref{simply-laced} given in Section \ref{ssec:Prqaa} below.
\end{Rem}

 In what follows, we often identify $\cM$ with $\Z^{\oplus \hI}$ by the correspondence $m \mapsto (u_{i,p}(m))$.
Then, the set $\cM^+$ is identified with $\N^{\oplus \hI}$.
We define a partial ordering on  $\cM^+ \simeq \mathbb{N}^{\oplus\hI}$ in the following way.
For each $(i,p) \in I \times \Z$ with $(i, p- r_i ) \in \hI$, following \cite{FR}, we define the loop analog of a simple root
\[ A_{i,p} =  Y_{i, p-r_i} Y_{i,p+r_i} \prod_{(j,s) \in \hI \colon c_{i,j} < 0, |s-p| < r_i} Y_{j,s}^{-1}\in \cM. \]
For $m,m' \in \cM$, we write $m \preceq  m'$ if $m'm^{-1}$ is a monomial in various $A_{i,p}$ for  $(i,p-r_i)\in \hI$.
This defines a partial ordering on $\cM$, called the \emph{Nakajima partial ordering}. 
As one also can view an element in 
$\cM \simeq \Z^{\oplus \hI}$ as an element in $\mathbb{Z}^{\oplus J}$ through the map $e$, 
this induces a partial ordering $\preceq$ on $\mathbb{N}^{\oplus J}$.

The following is a reformulation of well-known results by various authors, in particular \cite{CP, Chari, Kas, VV}.

\begin{Thm}\label{fundpbw} The pair $(\{V_j\}_{j\in J},\preceq)$ gives a PBW-theory of $\Cc_{\Z}$. 
\end{Thm}

We will call the corresponding standard modules the ordinary standard modules as they were studied by many authors, in particular from the point of view of geometric
representation theory for simply-laced quantum loop algebras.

A generalization of this PBW-theory is proposed  by  Kashiwara-Kim-Oh-Park in \cite{KKOP}. 
Consider a strong complete duality datum in the sense of \cite{KKOP} (such a family can be obtained from a $Q$-datum in 
the sense of \cite{FO21}). Then there is a corresponding collection of simple representations $(S_k)_{k\in\mathbb{Z}}$ in $\Cc_{\Z}$ called the 
affine cuspidal modules, see \cite[Section 5.2]{KKOP} 
(in the particular case above, the affine cuspidal modules are fundamental representations, 
now parameterized by $\mathbb{Z}$, that is we have fixed an increasing bijection between $J$ and $\mathbb{Z}$). 
Then let $\preceq$ be the bi-lexicographic ordering on $\N^{\oplus \Z}$.

\begin{Thm}[\cite{KKOP}] The pair $(\{S_k\}_{k\in\mathbb{Z}},\preceq)$ gives a PBW-theory of $\Cc_\Z$.
\end{Thm}

The ordinary PBW-theory given by fundamental representations in Theorem \ref{fundpbw} is a particular case of this result (see \cite[Remark 6.4]{KKOP}), but there are more general PBW-theories arising in this form.

\subsection{$R$-matrices}\label{Rmat}

The algebra $U_q(L\fg)$ has a $\mathbb{Z}$-grading defined on Chevalley generators by $\text{deg}(e_i)
= \text{deg}(f_i) = \text{deg}(k_i^{\pm 1}) = 0$ for $i\in I$ and
$\text{deg}(e_0) = - \text{deg}(f_0) = 1$. There is a corresponding algebra morphism 
$\tau_u \colon U_q(L\fg)\rightarrow U_q(L\fg)[u^{\pm 1}]$ such that a homogeneous element $g$ of degree
$m\in\mathbb{Z}$ satisfies $\tau_u(g) = u^m g$. 

Let $V$ be a representation of $U_q(L\fg)$. 
Consider the ring $\Oo = \kk[\![ z ]\!]$ as above with the formal variable $z = u - 1$. Then the $\Oo$-module $(V)_{u} = V\otimes \Oo$ has a structure 
of $U_q(L\fg)_{\Oo}$-module obtained as the twist of the module structure of $V$ by $\tau_u$. The morphism 
$\tau_u$ is compatible with the coproduct of $U_q(L\fg)$, and so for two $U_q(L\fg)$-modules $V$ and $W$ we have
$$(V\otimes W)_u\simeq (V)_u\otimes (W)_u.$$
We can also consider the tensor product $V_u \otimes W_v$ and its scalar extension 
\[((V)_u \otimes (W)_v)_{\kk(\!( z,w )\!)} \seq ((V)_u \otimes (W)_v) \otimes_{\kk[\![ z ]\!] \otimes \kk [\![w]\!]}\kk(\!( z,w )\!)\] 
to the ring of Laurent formal power series with two variables $z = u - 1$ and $w = v - 1$.

\begin{Thm}\label{ext} 
Let $M$, $N$, $P$ be simple modules in $\Cc_{\Z}$. 
There is a unique isomorphism of $U_q(L\fg)$-modules
$$T_{M,N}(u,v) \colon ((M)_u \otimes  (N)_v)_{\kk(\!(z,w)\!)}\rightarrow ((N)_v \otimes (M)_u)_{\kk(\!(z,w)\!)},$$
normalized so that for $y\in M$, $y'\in N$ highest weight vectors, the image of $y\otimes y'$ by $(T_{M,N}(u,v))$ is $y'\otimes y$. Moreover $T_{M,N}(u,v) = T_{M,N}(u/v)$ depends only on $u/v$ and is rational
$$T_{M,N}(u,v)(M\otimes N) \subset (N\otimes M)\otimes  \kk (u/v).$$ 
It satisfies the quantum Yang-Baxter equation, that is we have
\begin{align}
&(T_{N,P}(v)\otimes \id) \circ (\id\otimes T_{M,P}(u)) \circ (T_{M,N}(u/v)\otimes \id)\\
&  = (\id \otimes T_{M,N}(u/v)) \circ 
(T_{M,P}(u)\otimes \id) \circ (\id \otimes T_{N,P}(v)). 
\end{align}
\end{Thm}

The isomorphism $T_{M,N}(u,v)$ is obtained by the specialization of the universal $R$-matrix normalized on tensor products of highest weight vectors (see \cite{ifre} and \cite[Proposition 9.5.3]{efk}). 

Let us consider the order of $1$ as a pole of $T_{M,N}(u)$: \
$$\fo(M,N)\in\N.$$ 
The renormalized $R$-matrix is defined as
$$R_{M,N}(u) = (u-1)^{ \fo(M,N) }T_{M,N}(u).$$
Its limit at $u \rightarrow 1$ is a non zero morphism of $U_q(L\fg)$-modules (considered in \cite{KKK}): 
\begin{equation}\label{rmn}\mathbf{r}_{M,N} \colon M\otimes N\rightarrow N\otimes M.\end{equation}

\begin{Rem} 
It is not clear how to define the quantity $\fo(M,N)$ for general categories as considered in Section \ref{secun}.
However, for the ordinary PBW-theory $(\{ V_j \}_{j \in J}, \preceq)$ in Theorem \ref{fundpbw} and its generically commutative deformation introduced in the next subsection, we have 
\[ \fo(V_i, V_j) = \begin{cases*}
\alpha(i,j) & if $i < j$, \\
0 & otherwise,
\end{cases*}
\]
where $\alpha(i,j)$ is the number defined in Section 2.
The operators $\mathbf{r}_{M,N}$ defined as the limits of operators $R_{M,N}$ coincide with the specialized $R$-matrices in the general framework of Section \ref{mjf}. 
In the situations considered below, these notations 
will not lead to confusions because, as explained above, they are well-defined up to multiples in $\kk^\times$.
\end{Rem}

\begin{Ex}\label{exsl2} Let $\mathfrak{g} = \mathfrak{sl}_2$ and $M = L(Y_{1,a})$, $N = L(Y_{1,b})$ be fundamental representations. The structure of $M\otimes N$ is well-known. We have $ \fo(M,N) = \delta_{b-a, 2}$ and $R_{M,N}$ is an isomorphism if $|b-a|\neq 2$. If $b = a -2$, its image is simple of dimension $3$ isomorphic to $L(Y_{1,a}Y_{1,b})$ and its kernel is the trivial module of dimension $1$. If $b = a + 2$, its image is simple of dimension $1$ and its kernel is isomorphic to $L(Y_{1,a}Y_{1,b})$. All this can be checked by direct computations. Indeed, there are respective bases $(v_a^+,v_a^-)$ and $(v_b^+,v_b^-)$ of weight vectors of $M$ and of $N$, so that 
in the basis $(v_a^+\otimes v_b^+, v_a^+ \otimes v_b^-, v_a^-\otimes v_b^+, 
v_a^-\otimes v_b^-)$, we see that 
$$\begin{pmatrix}1&0&0&0\\ 0&\frac{u(1 - q^{-2})}{u-q^{b-a-2}}&\frac{q^{-1}(u-q^{b-a})}{u-q^{b-a-2}}&0\\0&\frac{q^{-1}(u - q^{b-a})}{u - q^{b-a-2}}&\frac{q^{b-a}(1-q^{-2})}{u - q^{b-a-2}}&0\\0&0&0&1\end{pmatrix}$$ 
is the matrix of $T_{M,N}(u)$. From the basis 
$$(v_a^+\otimes v_b^+,  v_a^+\otimes v_b^- +  q^{-1}  v_a^-\otimes v_b^+, 
q^{-1} v_a^+\otimes v_b^- - v_a^-\otimes v_b^+,  v_a^-\otimes v_b^-)$$ to the basis 
$$(v_b^+\otimes v_a^+, v_b^+\otimes v_a^- + q^{-1} v_b^-\otimes v_a^+, q^{-1} v_b^-\otimes 
v_a^+ - q^{-1} v_b^+\otimes v_a^-,  v_b^-\otimes v_a^-)$$ the matrix is diagonal
$$T_{M,N}(u) = \text{diag}(1,1,\delta(u),1),$$
where $\delta(u) = \frac{q^{b-a} - uq^{-2}}{u - q^{b-a-2}}$. When $a = b+2$, at the limit $u \rightarrow 1$ one obtains 
$$\mathbf{r}_{ M,  N}  = \text{diag}(1,1, 0 ,1).$$ 
When $b = a+2$, multiplying by $u - 1$, we obtain at the limit
$$\mathbf{r}_{M,  N } = \text{diag}(0,0, q^2 - q^{-2} ,0).$$
We note that  in these cases $|b-a| = 2$  we have 
$$R_{M,  N}(u) \circ R_{N,  M }(u) = (u - 1)\id.$$
\end{Ex}

\subsection{Deformation}\label{defocons}
We fix a PBW-theory $(\{S_k\}_{k\in J },\preceq)$ of $\Cc_\Z$ as above ($J = e(\hI)$ or $\Z$). We set
$$\tilde{S}_k \seq (S_k)_{\text{exp}(k z)}.$$
It is a $U_q(L\fg)_{\Oo}$-module. 
For any $k, k' \in J$, we have an isomorphism
$$R_{k,k'} = R_{S_k,S_{k'}}(\text{exp}((k-k')z)) \colon (\tilde{S}_k \star_\Oo \tilde{S}_{k'})_\Kk \simeq (\tilde{S}_{k'} \star_\Oo \tilde{S}_{k})_\Kk.$$

\begin{Thm}\label{normt} The collection $\{\tilde{S}_k\}_{k\in J}$ is a normal, consistent, generically commutative deformation of $\{S_k\}_{k\in J }$.
\end{Thm}

\begin{proof}
The statement follows from the results recalled above, and \cite[Proposition 5.7(iii)]{KKOP} for the normality. 
\end{proof}

\begin{Rem} Recall $\beta(\bep', \bep)$ defined in Section \ref{rnomat}. For the ordinary  PBW-theory of the quantum loop algebras in Theorem~\ref{fundpbw}, we have $\beta(\bep',\bep) = 0$ if $\bep \lesssim \bep'$ by considering highest weight vectors as in the proof above. 
If $\bep'\lesssim \bep$, we may have $\beta(\bep',\bep) > 0$, but we have $\alpha(\bep',\bep) = \beta(\bep',\bep)$.\end{Rem}

As a consequence of Theorem \ref{normt}, we obtain a generalization of \eqref{eq:KtL} and Corollary \ref{scase} in the situation of this section. 
The class $[L]$ of the simple quotient $L$ 
of $M(\bep_s)$ occurs with multiplicity $1$ in $M(\bep)$: 
\begin{equation}\label{mxtf}[M(\bep)]_t = [L] + \sum_{L' \prec L} P_{L',\bep}(t) [L']\end{equation}
where $P_{L',\bep}(t)\in \N [t^{\pm 1}]$ and $\prec$ is the Nakajima partial ordering on simple classes.

\begin{Ex}\label{exfil} We continue Example \ref{exsl2} and we compute the corresponding monoidal Jantzen filtrations. We consider 
$$\bep = \bep_s = (3,1)$$ 
with 
$$S_3 = L(Y_{1,2})\text{ and }S_1 = L(Y_{1,0}).$$ 
Then $M(\bep_s) = S_3\otimes S_1$ has a unique proper submodule $S$ of dimension $1$ and $M(\bep_c) = S_1\otimes S_3$ has a unique proper submodule $L$ of dimension $3$.

For $M(\bep_s)$, we are in the first situation of Example \ref{Ex:standard}. We have $R_{\bep_c, \bep_s} = R_{3,1}$ and 
\[\begin{split}z^N R_{3,1}^{-1}\tM(\bep_c)\cap \tM(\bep_s)
=\begin{cases}\tM(\bep_s)&\text{ if $N\leq 0$,} \\z^N\tM(\bep_s) + z^{N-1} \Oo S&\text{ if $N \geq 1$,} \end{cases}\end{split}\]
$$F_\bullet M(\bep_s)\colon \quad  \cdots \supset F_{0}M(\bep_s) = M(\bep_s) \supset F_1M(\bep_s) = S \supset F_2M(\bep_s) = 0 \supset \cdots. $$
For $M(\bep_c)$, we are in the second situation of Example \ref{Ex:standard}. We have 
\[\begin{split} z^NR_{3,1}\tM(\bep_s)\cap \tM(\bep_c) = \begin{cases} \tM(\bep_c)&\text{ if $N\leq -1$,} 
\\z^{N+1}\tM(\bep_c) + z^{N} \Oo L&\text{ if $N \geq 0$,}\end{cases}\end{split}\]
$$F_\bullet M(\bep_c)\colon \quad \cdots \supset F_{-1}M(\bep_c) = M(\bep_c) \supset F_0M(\bep_c) = L \supset F_1M(\bep_c) = 0 \supset \cdots. $$
\end{Ex}

\begin{Ex} Let us illustrate Proposition \ref{compfil} with the filtrations computed in Example \ref{exfil} for the morphism
$$\mathbf{r}_{\bep_c,\bep_s}\colon M(\bep_s)\rightarrow M(\bep_c).$$
Then we have: 
$$\mathbf{r}_{\bep_c,\bep_s}(F_0 M(\bep_s))\subset \Image(\mathbf{r}_{\bep_c,\bep_s}) = L = F_0 M(\bep_c),$$
$$\mathbf{r}_{\bep_c,\bep_s}(F_1 M(\bep_s)) = \mathbf{r}_{\bep_c,\bep_s}(S) = 0 =   F_1 M(\bep_c).$$
\end{Ex}

\subsection{Quantum Grothendieck ring} \label{Ssec:qGr}
We recall the construction of the quantum Grothendieck ring.  For a representation $M$ in $\Cc_\Z$ we have its $q$-character defined in \cite{FR}. 
It can be proved \cite{HL10} that as $M$ is in $\Cc_\Z$, we have
$$\chi_q(M)\in \mathcal{Y} = \mathbb{Z}[Y_{i,p}^{\pm 1}]_{(i,p)\in\hI}.$$
It defines the $q$-character morphism on the Grothendieck ring $K(\Cc_\Z)$ of $\Cc_\Z$ 
$$\chi_q \colon K(\Cc_\Z) \rightarrow \mathcal{Y}.$$

Consider the quantum Cartan matrix $C(z)= (C_{i,j}(z))_{i,j\in I}$
defined by $C_{i,j}(z) = [c_{i,j}]_z$ if $i\neq j$ and
$C_{i,i}(z) = [2]_{z^{r_i}}$ for $i\in I$, where $[k]_z \seq (z^k-z^{-k})/(z-z^{-1})$ is the standard quantum integer. 
We will denote $\tilde{C}_{i,j}(z) = \sum_{m\geq 1}\tilde{c}_{i,j}(m) z^m$ the expansion 
of the $(i,j)$-entry of
the inverse $\tilde{C}(z)$ of the quantum Cartan matrix $C(z)$ at $z=0$.
We also extend the definition of $\tilde{c}_{i,j}(m)$ to
every $m\in\Z$ by setting $\tilde{c}_{i,j}(m) = 0$ if $m\le 0$.

For $(i,p), (j,s)\in \hI$, following \cite{H1}, we set 
\begin{equation}\label{defcN}
\mathcal{N}(i,p;j,s) \seq \tilde{c}_{i,j}(p-s-r_i) - \tilde{c}_{i,j}(p-s+r_i)
-\tilde{c}_{i,j}(s-p-r_i) + \tilde{c}_{i,j}(s-p+r_i). 
\end{equation} 
As $\mathcal{N}(i,p;j,s) = - \mathcal{N}(j,s;i,p)$, 
this defines a skew-symmetric bilinear form
\begin{equation}\label{caln}\mathcal{N} \colon\N^{\oplus \hI}\times\N^{\oplus \hI}\rightarrow \mathbb{Z}.\end{equation}

\begin{Def}[\cite{H1}] 
We define the quantum torus $\mathcal{Y}_{t}$ as the $\Z[t^{\pm1/2}]$-algebra presented by 
the set of generators $\{\tY_{i, p}^{\pm 1} \mid (i,p) \in \hI \ \}$ and the following relations:
\begin{enumerate}
\item 
$\tY_{i,p}\tY_{i,p}^{-1}= \tY_{i,p}^{-1} \tY_{i,p}=1$ for each $(i,p) \in \hI$,
\item 
$\tY_{i, p}\tY_{j, s} = t^{-\mathcal{N}(i,p;j,s)}\tY_{j,s}\tY_{i,p}$ for each $(i,p), (j,s) \in \hI$.
\end{enumerate}
\end{Def}

\begin{Rem}\label{remqt} See \cite[Remark 3.1]{HL15} for comments on the relations with the quantum torus in \cite{VV03} and \cite{Nak04} for simply-laced quantum loop algebras. \end{Rem}

\begin{Ex} Let $\mathfrak{g} = \mathfrak{sl}_2$. Then $\tY_{1,2}\tY_{1,0} = t^{-2}\tY_{1,0}\tY_{1,2}$. 
\end{Ex}

The evaluation at $t = 1$ is the $\Z$-algebra homomorphism $\ev_{t=1}   \colon \mathcal{Y}_{t} \to \mathcal{Y}$
given by
$$
t^{1/2} \mapsto 1, \qquad \tY_{i,p} \mapsto Y_{i,p}.
$$

An element $\tm \in \mathcal{Y}_{t}$
is called a {\em monomial} if it is a product of the generators
$\tY_{i,p}$ for $(i,p) \in \hI$ and $t^{\pm 1/2}$. 
For a monomial $\tm \in \mathcal{Y}_{t}$, we set $u_{i,p}(\tilde{m})$ 
to be  the  power of $Y_{i,p}$ in $ \ev_{t=1} (\tm)$.
A monomial $\tm$ in $\mathcal{Y}_{t}$ is said to be {\em dominant} if $ \ev_{t=1}  (\tm)\in \cM^+ $. 
Moreover, for monomials $\tm, \tm^{\prime}$ in $\mathcal{Y}_{t}$, set
$$
\text{$\tm \preceq \tm^{\prime}$ if and only if $\ev_{t=1} (\tm) \preceq \ev_{t=1}  (\tm^{\prime})$},
$$
with the ordering on $\cM$ defined above.
Following~\cite[Section~6.3]{H1}, we define the $\Z$-algebra anti-involution $\ol{(\cdot)}$ on $\mathcal{Y}_{t}$ by
$$
t^{1/2} \mapsto t^{-1/2}, \qquad \tY_{i,p} \mapsto t\tY_{i,p}.
$$
For any monomial $\tm$ in $\mathcal{Y}_{t}$,
there uniquely exists $a \in \Z$ such that $\ul{\tm} = t^{a/2}\tm$ is $\ol{(\cdot)}$-invariant. As $\ul{\tm}$ depends only on $\ev_{t=1} (\tm)$, for every monomial $m\in\cM$, the element $\ul{m}$
is well-defined in $\mathcal{Y}_{t}$.
These elements form the free $\Z[t^{\pm 1/2}]$-basis 
of $\mathcal{Y}_{t}$ called the basis of {\em commutative monomials}.
For example, for $(i,p) \in \hI$, we set
$$
\tA_{i,p + r_{i}} \seq \ul{A_{i, p+r_{i}}}.
$$

For each $i \in I$, denote by $K_{i,t}$ the $\Z[t^{\pm 1/2}]$-subalgebra
of $\mathcal{Y}_{t}$ generated by 
$$
\{ \tY_{i,p} (1+t^{-1}\tA_{i, p+ r_{i}}^{-1}) \mid (i,p)\in \hI \} 
\cup \{ \tY_{j,s}^{\pm 1} \mid (j,s) \in \hI, j \neq i \}.
$$
Following \cite{Nak04, VV03, H1}, {\em the quantum Grothendieck ring of} $\Cc_{\Z}$ is defined as
$$
K_{t}(\Cc_{\Z}) \seq \bigcap_{i \in I} K_{i,t}. 
$$
By construction, the quantum Grothendieck is stable by the  $\ol{(\cdot)}$-involution.

\begin{Thm}[{\cite[Theorem 5.11]{H1}}] \label{Thm:Ft}
For every dominant monomial $\tm$ in $\mathcal{Y}_{t}$, 
there uniquely exists an element $F_{t}(\tm)$ of $K_{t}(\Cc_\Z)$ such that $\tm$ is
the unique dominant monomial occurring in $F_{t}(\tm)$. The monomials $\tm^{\prime}$ occurring in $F_{t}(\tm) - \tm$
satisfy $\tm^{\prime}  \prec  \tm$.
In particular, the set $\{ F_{t}(\ul{m}) \mid m \in  \cM^+ \}$
forms a $\Z[t^{\pm 1/2}]$-basis of $K_{t}(\Cc_\Z)$.  
\end{Thm}

Note that the $F_{t}(\ul{m})$ are $\ol{(\cdot)}$-invariant.

For a dominant monomial $\tm$ in $\mathcal{Y}_{t}$ and $u_{i,p}(\tm)$ the power of $Y_{i,p}$ in $ \ev_{t=1}  (\tm)$, set
$$
E_{t}(\tm) \seq \tm \left( \prod_{p \in \Z}^{\leftarrow}
\left( \prod_{i \in I: (i,p) \in \hI} \tY_{i,p}^{u_{i,p}(\tm)} \right)
\right)^{-1} 
\prod_{p \in \Z}^{\leftarrow}
\left( \prod_{i \in I: (i,p) \in \hI} F_{t}(\tY_{i,p})^{u_{i,p}(\tm)} \right).
$$
Note that by \cite{H1}, the products are well-defined.

The element $E_{t}(\ul{m})$ is called \emph{the $(q,t)$-character of the ordinary standard module} $M(m)$ associated to $m$ as above. 
By \cite{FM01, Her05}, the image by $ \ev_{t=1}  $ is $\chi_q(M(m))$.

We consider another kind of elements $L_{t}(\ul{m})$ in $K_{t}(\Cc_\Z)$ which is conjecturally a $t$-quantum version of the 
$q$-character of simple modules.

\begin{Thm}[{\cite[Theorem 8.1]{Nak04}}, {\cite[Theorem 6.9]{H1}}]
\label{Thm:qtch}
For a dominant monomial $m \in  \cM^+ $, there exists a unique element $L_{t}(\ul{m})$ in  $K_{t}(\Cc_\Z)$ 
such that 
\begin{itemize}
\item[(S1)] $\ol{L_{t}(\ul{m})} = L_{t}(\ul{m})$, and
\item[(S2)] $L_{t}(\ul{m}) = E_{t}(\ul{m}) + \sum_{m^{\prime} \in \cM^+} Q_{m, m^{\prime}}(t) E_{t}(\ul{m^{\prime}})$
with $Q_{m,m^{\prime}}(t) \in t\Z[t]$.
\end{itemize}
Moreover, we have $Q_{m, m^{\prime}}(t) = 0$ unless $m^{\prime}  \prec  m$. 
In particular, the set $\{ L_{t}( \ul{m} ) \mid m \in \cM^+\}$ forms a $\Z[t^{\pm 1/2}]$-basis of  $K_{t}(\Cc_\Z)$.
\end{Thm} 

The element $L_{t}(\ul{m})$ is called \emph{the $(q,t)$-character of the simple module $L(m)$}.

In what follows, for a dominant monomial $m \in \cM^+$, we will write for simplicity
$$
F_{t}(m) \seq F_{t}(\ul{m}), \qquad 
E_{t}(m) \seq E_{t}(\ul{m}), \qquad
L_{t}(m) \seq L_{t}(\ul{m}) .
$$

\begin{Conj}[{\cite[Conjecture 7.3]{H1}}] \label{Conj:KL}
For all $m \in \cM^+$, we have
$$
 \ev_{t=1}  (L_{t}(m)) = \chi_{q}(L(m)).
$$
\end{Conj}
A fundamental theorem of Nakajima~\cite[Theorem 8.1]{Nak04} states that this holds true when $\fg$ is of simply-laced type. The proof used the geometry of quiver varieties. This was the main motivation for this conjecture. This conjecture is now proved for type $B$ in \cite{FHOO} and  for all simple modules that are reachable (in the sense of cluster algebras) for general types in \cite{FHOO2}.

Thanks to the unitriangular property (S2), we can write
\begin{equation} \label{eq:KL}
E_{t}(m) = L_{t}(m) + \sum_{m^{\prime} \in \cM^+ \colon m'  \prec  m } P_{m, m^{\prime}}(t) L_{t}(m^{\prime})
\end{equation}
with some $P_{m,m'}(t) \in t\Z[t]$ for each $m \in \cM^+$.
The polynomials $P_{m,m'}(t)$ are analogs of Kazhdan-Lusztig polynomials for finite-dimensional representations of quantum loop algebras. The following was proved by Nakajima \cite{Nak04} for simply-laced types, and by the authors of \cite{FHOO} for general types.

\begin{Thm}[\cite{Nak04, FHOO}] \label{pospol}
The polynomials $P_{m,m'}(t)$ are positive.\end{Thm}

\begin{Ex}\label{fet} Let $\mathfrak{g} = \mathfrak{sl}_2$ and $m = Y_{1,0}Y_{1,2}$. Then $E_t(m)$ equals
$$ t (\underline{Y_{1,2}} + \underline{Y_{1,4}^{-1}})(\underline{Y_{1,0}} + \underline{Y_{1,2}^{-1}})
= (\underline{Y_{1,0}Y_{1,2}} + \underline{Y_{1,0}Y_{1,4}^{-1}} + \underline{Y_{1,2}^{-1}Y_{1,4}^{-1}}) + t = L_t(m) + t L_t(1).$$
The specialization at $t = 1$ corresponds to $[L(Y_{1,2})\otimes L(Y_{1,0})] = [L(m)] + 1$ in $K(\Cc_\Z)$.
\end{Ex}

\subsection{Quantum Grothendieck ring conjecture}

Recall $\mathcal{N}$ defined in the previous  section 
By considering the powers of the variable of dominant monomials $m$, $m'$ 
in $\cM^+$, \eqref{caln} also defines $\mathcal{N}(m,m')\in\mathbb{Z}$.

\begin{Rem} There is an interpretation of $\mathcal{N}$. Let $M$ and $N$ be simple modules in $\Cc_\Z$. Then set
$$\Lambda(M,N) = \mathcal{N}(M,N) + 2 \fo(M,N), $$
where $\mathcal{N}(M,N) = \mathcal{N}(m_M,m_N)$ with 
$m_M,m_N\in \cM^+$ dominant monomials parameterizing 
$M$ and $N$ respectively. As proved in \cite{FO21}, $\Lambda(M,N)$ coincides with the invariant defined in \cite{kkop0}.
\end{Rem}

We continue with a PBW-theory as in Section \ref{defocons}. Let us denote $m_k\in \cM^+$ the dominant monomial so that $S_k = L(m_k)$.
We consider the skew-symmetric bilinear form $\gamma$ defined on $\mathbb{N}^{\oplus J}\times \mathbb{N}^{\oplus J}$ and so that for 
any $k,k'\in  J $: 
$$\gamma(\bdelta_k,\bdelta_{k'}) = -\mathcal{N}(m_k,m_{k'})/2.$$ 
We consider the associated bilinear operation $* = *_\gamma$ on $K(\Cc_{\mathbb{Z}})_t$.
Be aware that this operation $*$ also depends on our choice of PBW-theory and its deformation. 

Let us define the $\Z[t^{\pm 1/2}]$-linear isomorphism $\phi \colon K(\Cc_\Z)_t \simeq K_t(\Cc_\Z)$ by 
$\phi ([L(m)]) = L_t(m)$
for all $m \in \cM^+$. Clearly, we have $\phi \circ \ol{(\cdot)} = \ol{( \cdot )} \circ \phi$. 

\begin{conj}[Quantum Grothendieck Ring Conjecture] 
\label{qgrconj}
 With a chosen PBW-theory and its deformation,  
Associativity Conjectures \ref{Conj:assoc} and \ref {Conj:sassoc} hold for $(K(\Cc_{\mathbb{Z}})_t,*)$, and the linear isomorphism $\phi$ gives a $\Z[t^{\pm 1/2}]$-algebra isomorphism from $(K(\Cc_{\mathbb{Z}})_t,*)$ to the quantum Grothendieck ring $K_t(\Cc_{\mathbb{Z}})$. 
\end{conj}

\begin{Rem} If Conjecture \ref{qgrconj} is true for any PBW-theory, it implies that the ring structure $(K(\Cc_{\mathbb{Z}})_t,*)$, with its canonical basis, does not depend on the choice of PBW-theory.
\end{Rem} 

\begin{Ex} We can illustrate first with the filtrations computed in Example \ref{exfil} with $\bep_s = (3,1)$. We have
$$[M(\bep_s)]_t = [L] + t \quad \text{ and } \quad [M(\bep_c)]_t = [L] + t^{-1}.$$
As $\mathcal{N}(1,3;1,1) = 2$, we recover the well-known formulas (see Example \ref{fet}):
$$[S_3]*[S_1] = t^{-1}[L] + 1 \quad \text{ and } \quad [S_1]*[S_3] = t [L] + 1.$$
\end{Ex}

Now, let us consider the ordinary PBW-theory of fundamental modules as in Theorem \ref{fundpbw}, and assume that Conjecture \ref{qgrconj} is true for this case. 
By Remark \ref{remmulti}, it implies that, for any $m\in \cM^+$, we have
$$\phi([M(m)]_t) = E_t(m). $$ 
More generally, for each $d \in \N$ and $\bep = (\ep_1,\cdots, \ep_d)\in J^d$, letting
\[ E_t(\bep) \seq t^{\sum_{1 \le k < l \le d}\mathcal{N}(i_k, p_k; i_l, p_l)/2}  F_t(Y_{i_1, p_1}) \cdots F_t(Y_{i_d, p_d}),\]
where $(i_k, p_k)$ denotes the element of $\hI$ such that $e(i_k, p_k) = \ep_k$ for each $1 \le k \le d$, we obtain
\begin{equation}\label{mixedet}
\phi( [M(\bep)]_t) = E_t(\bep).
\end{equation}
Compare with \eqref{eq:*mix}.
The converse statement is true.
Namely, Conjecture \ref{qgrconj} for the ordinary PBW-theory is equivalent to the following.

\begin{conj}\label{qgrmix}
The equality \eqref{mixedet} holds for any $d \in \N$ and  $\bep \in J^d$. 
\end{conj}

\begin{Rem}
Note that Conjecture \ref{qgrmix} also implies  Duality Conjecture \ref{Conj:duality} for the ordinary PBW-theory. Indeed, for each  $\bep\in J^\bd$, we have 
$\overline{\phi([M(\bep)]_t)} = \ol{E_t(\bep)} = E_t(\bep^{\op}) = \phi([M(\bep^\op)]_t)$ and so  $\ol{[M(\bep)]_t} = [M(\bep^\op)]_t$. 
\end{Rem}

\begin{Rem} Conjecture \ref{qgrmix} also implies the analog of Kazhdan-Lusztig conjecture (= Conjecture \ref{Conj:KL}). 
In addition, as \eqref{pos1} is a consequence of Conjecture \ref{qgrmix}, the positivity of Kazhdan-Lusztig polynomials of Theorem \ref{pospol} can be seen as an evidence for this conjecture.
\end{Rem}

When $\fg$ is of simply-laced type, we can actually establish that the conjectures are true for the ordinary PBW-theory with the help of geometry.  
Namely, we have the following, whose proof will be given in Section \ref{ssec:Prqaa} below.

\begin{Thm} \label{simply-laced}
When $\fg$ is of simply-laced type, Conjecture \ref{qgrmix} holds, or equivalently, Conjecture \ref{qgrconj} for the ordinary PBW theory holds.
\end{Thm}

\subsection{Examples}\label{exfil2}

\subsubsection{} 
\label{exfil2_1}
We consider examples as in Example \ref{exfil} but with more factors: 
$$\bep_s = (3,3,1), \quad \bep = (3,1,3), \quad \bep_c = (1,3,3)$$
so that we have
$$M(\bep_s) = S_3\otimes S_3\otimes S_1, \quad
M(\bep) = S_3\otimes S_1\otimes S_3, \quad 
M(\bep_c) = S_1\otimes S_3\otimes S_3.$$
We have the morphisms:
$$ M(\bep_s) \overset{\mathbf{r}_{\bep,\bep_s}}{\longrightarrow} M(\bep)\overset{\mathbf{r}_{\bep_c,\bep}}{\longrightarrow}M(\bep_c).$$
We obtain the monoidal Jantzen filtrations 
$$F_0M(\bep_s) = M(\bep_s)\supset F_1 M(\bep_s) = F_2M(\bep_s) =  \Ker(\mathbf{r}_{\bep,\bep_s}) \supset F_3M(\bep_s) = \{0\},$$
$$F_0 M(\bep) = M(\bep) \supset F_1M(\bep) = \{0\},$$
$$F_{-2}M(\bep_c)  = M(\bep_c) \supset F_{-1}M(\bep_c) = F_0 M(\bep_c) =  \Image(\mathbf{r}_{\bep_c,\bep})\supset F_1(M(\bep_c)) = \{0\}.$$
Let $L$ be the simple quotient of $M(\bep_s)$. We obtain
$$[M(\bep_s)]_t = [L] + t^2[S_3], \quad
[M(\bep)]_t = [L] + [S_3], \quad
[M(\bep_c)]_t = [L] + t^{-2}[S_3],
$$
$$m_3([S_3],[S_3],[S_1]) = t^{-2}[L] + [S_3], \quad 
m_3([S_3],[S_1],[S_3]) = [L] + [S_3],$$
$$m_3([S_1],[S_3],[S_3]) = t^2 [L] + [S_3].
$$
This is completely analogous to the case of $\bep_s = (3,1,1)$.

\subsubsection{}
\label{exfil2_2}
 Now we set 
$$\bep_s = (5,3,1), \quad \bep_c = (1,3,5)$$
so that we have
$$M(\bep_s) = S_5\otimes S_3\otimes S_1, \quad
M(\bep_c) = S_1\otimes S_3\otimes  S_5. $$
As $S_1\otimes S_5\simeq S_5\otimes S_1$, we have two intermediate modules:
$$M(\bep_1) = M(3,5,1) \simeq M(3,1,5) \quad \text{and} \quad M(\bep_2) = M(5,1,3)\simeq M(1,5,3).$$
We have the morphisms:
$$
\xymatrix{ &M(\bep_1) \ar@{->}[dr]^{\mathbf{r}_{\bep_c,\bep_1}}& \\
 M(\bep_s) \ar@{->}[ur]^-{\mathbf{r}_{\bep_1,\bep_s}}\ar@{->}[dr]_{\mathbf{r}_{\bep_2,\bep_s}} && M(\bep_c).  \\
& M(\bep_2) \ar@{->}[ur]_-{\mathbf{r}_{\bep_c,\bep_2}}&}
$$   
We obtain the monoidal Jantzen filtrations 
$$F_0M(\bep_s) = M(\bep_s)\supset F_1 M(\bep_s) = \Ker(\mathbf{r}_{\bep_c,\bep_s}) =  \Ker(\mathbf{r}_{\bep_1,\bep_s}) + \Ker(\mathbf{r}_{\bep_2,\bep_s}) 
\supset F_2M(\bep_s) = \{0\},$$
$$F_{-1} M(\bep_1) = M(\bep_1) \supset F_0 M(\bep_1) = \Image(\mathbf{r}_{\bep_1,\bep_s}) $$
$$\supset F_1 M(\bep_1) = \Ker(\mathbf{r}_{\bep_c,\bep_1})\cap \Image(\mathbf{r}_{\bep_1,\bep_s}) \supset F_2 M(\bep_1) = \{0\},$$
$$F_{-1} M(\bep_2) = M(\bep_2) \supset F_0 M(\bep_2) = \Image(\mathbf{r}_{\bep_2,\bep_s})$$
$$\supset F_1 M(\bep_2) = \Ker(\mathbf{r}_{\bep_c,\bep_2})\cap \Image(\mathbf{r}_{\bep_2,\bep_s}) \supset F_2 M(\bep_2) = \{0\},$$
$$F_{-1}M(\bep_c)  = M(\bep_c) \supset F_0 M(\bep_c)  = \Image(\mathbf{r}_{\bep_c,\bep}) = \Image(\mathbf{r}_{\bep_c,\bep_1}) \cap \Image(\mathbf{r}_{\bep_c,\bep_2})\supset F_1(M(\bep_c)) = \{0\}.$$
Let $L$ be the simple quotient of $M(\bep_s)$.
We obtain
$$[M(\bep_s)]_t = [L] + t[S_1] + t[S_5], \quad 
[M(\bep_1)]_t = [L] + t^{-1}[S_1] +  t [S_5]$$
$$[M(\bep_2)]_t = [L] + t [S_1] + t^{-1} [S_5], \quad 
[M(\bep_c)]_t = [L] + t^{-1}[S_1] + t^{-1} [S_5],
$$
$$m_3([S_5],[S_3],[S_1]) = t^{-1}[L] + [S_1] + [S_5],  \quad
m_3([S_1],[S_3],[S_5]) = t[L] + [S_1] + [S_5],$$
$$m_3([S_3],[S_5],[S_1]) =  t [L] + [S_1] +  t^2 [S_5] = t^2 m_3([S_3],[S_1],[S_5]),$$
$$m_3([S_5],[S_1],[S_3]) = t[L] + t^2 [S_1] + t [S_5] = t^2 m_3([S_1],[S_5],[S_3])
.$$

\subsubsection{} 
\label{exfil2_3}
Now we set 
$$\bep_s = (3,3,1,1), \quad \bep_c = (1,1,3,3)$$
so that we have
$$M(\bep_s) = S_3^{\otimes 2}\otimes S_1^{\otimes 2}, \quad
M(\bep_c) = S_1^{\otimes 2}\otimes S_3^{\otimes 2}.$$
We have four intermediate modules:
$$M(\bep_1) = M(3,1,3,1), \quad  M(\bep_2) = M(3,1,1,3), $$
$$ M(\bep_3) = M(1,3,3,1), \quad M(\bep_4) = M(1,3,1,3).$$
We have the specialized $R$-matrices:
$$\xymatrix{ &&M(\bep_2)\ar[dr]^{\mathbf{r}_{\bep_3,\bep_2}}&   &
\\ M(\bep_s) \ar[r]^{\mathbf{r}_{\bep_1,\bep_s}}&M(\bep_1) \ar[ur]^{\mathbf{r}_{\bep_2,\bep_1}}\ar[dr]_{\mathbf{r}_{\bep_3,\bep_1}}& & M(\bep_4) \ar[r]^{\mathbf{r}_{\bep_c,\bep_4}}&M(\bep_c)
\\ && M(\bep_3) \ar[ur]_{\mathbf{r}_{\bep_4,\bep_3}}& &}$$
as well as the morphism $\mathbf{r}_{\bep_s,\bep_c} \colon M(\bep_c)\rightarrow M(\bep_s)$ obtained 
as in Section \ref{Rmat}.

We obtain the monoidal Jantzen filtrations: 
$$F_0M(\bep_s) = M(\bep_s)\supset F_1 M(\bep_s) = \Ker(\mathbf{r}_{\bep_c,\bep_s}) \supset F_2M(\bep_s) = F_3 M(\bep_s) = \Ker(\mathbf{r}_{\bep_1,\bep_s}) $$
$$\supset F_4 M(\bep_s) = 
\Image(\mathbf{r}_{\bep_s,\bep_c})\supset F_5 M(\bep_s) = \{0\},$$
$$F_0 M(\bep_1) = M(\bep_1) \supset F_1 M(\bep_1) = \Ker(\mathbf{r}_{\bep_2,\bep_1}) + \Ker(\mathbf{r}_{\bep_3,\bep_2})$$
$$\supset F_2 M(\bep_1) = \Ker(\mathbf{r}_{\bep_2,\bep_1}) \cap \Ker(\mathbf{r}_{\bep_3,\bep_1}) \supset  F_3 M(\bep_1) = \{0\},$$
$$F_{-1} M(\bep_2) = M(\bep_2) \supset F_0 M(\bep_2) = \Ker(\mathbf{r}_{\bep_4,\bep_2}) + \Image(\mathbf{r}_{\bep_2,\bep_s})$$
$$\supset F_1 M(\bep_2) = \Ker(\mathbf{r}_{\bep_4,\bep_2}) \cap \Image(\mathbf{r}_{\bep_2,\bep_s}) \supset  F_2 M(\bep_2) = \{0\},$$
$$\supset F_1 M(\bep_3) = \Ker(\mathbf{r}_{\bep_4,\bep_3}) \cap \Image(\mathbf{r}_{\bep_3,\bep_s}) \supset  F_2 M(\bep_3) = \{0\},$$
$$F_{-2} M(\bep_4) = M(\bep_4) \supset F_{-1} M(\bep_4) = \Image(\mathbf{r}_{\bep_4,\bep_2}) + \Image(\mathbf{r}_{\bep_4,\bep_3})$$
$$\supset F_0 M(\bep_4) = \Image(\mathbf{r}_{\bep_4,\bep_2}) \cap \Image(\mathbf{r}_{\bep_4,\bep_3}) \supset  F_1 M(\bep_4) = \{0\},$$
$$F_{-4}M(\bep_c)  = M(\bep_c) \supset F_{-3} M(\bep_c)  = \Ker(\mathbf{r}_{\bep_s,\bep_c})$$
$$\supset F_{-2}(M(\bep_c)) =  F_{-1} M(\bep_c) = \Image(\mathbf{r}_{\bep_c,\bep_4}) \supset F_0 M(\bep_c) = \Image(\mathbf{r}_{\bep_c,\bep_s}) 
\supset F_1 M(\bep_c) = \{0\}.$$
Let $L$ be the simple quotient of $M(\bep_s)$ and $K = L(Y_{1,q^0}Y_{1,q^2})$. We obtain:
$$[M(\bep_s)]_t = [L] + (t^3 + t) [K] + t^4, \quad [M(\bep_1)]_t = [L] + 2 t [K] + t^2,$$
$$[M(\bep_2)]_t = [L] + (t + t^{-1}) [K] + 1 = [M(\bep_3)]_t,$$
$$[M(\bep_4)]_t = [L] + 2 t^{-1} [K] + t^{-2}, \quad [M(\bep_c)]_t = [L] + (t^{-3} + t^{-1}) [K] + t^{-4}.$$
This gives for the products:
$$[S_3^{\otimes 2}]*[S_1^{\otimes 2}] = t^{-4} [L] + (t^{-1} + t^{-3}) [K] + 1, \quad [S_1^{\otimes 2}]*[S_3^{\otimes 2}] = t^4[L] + (t^3 + t) [K] + 1,$$
$$m_4([S_3],[S_1],[S_3],[S_1])= t^{-2}[L] + 2 t^{-1} [K] + 1,$$
$$m_3([S_3],[S_1^{\otimes 2}],[S_3]) = [L] + (t + t^{-1}) [K] + 1 = m_3([S_1],[S_3^{\otimes 2}],[S_1]) ,$$
$$m_4([S_1],[S_3],[S_1],[S_3]) = t^2 [L] + 2 t [K] + 1.$$

\subsubsection{} Let $\mathfrak{g}$ be of simply-laced type, $i\in I$, $r\in\mathbb{Z}$ and  set 
$$\bep_s = (i+(r+2)n, i+rn), \quad \bep_c = (i+rn,i+(r+2)n)$$
so that we have
$$M(\bep_s) = L(Y_{i,r+2})\otimes L(Y_{ i,r}),  \quad
M(\bep_c) = L(Y_{i,r})\otimes L(Y_{i,r+2}).$$
We have the morphism:
$$\mathbf{r}_{\bep_c,\bep_s} \colon M(\bep_s) \rightarrow M(\bep_c).$$
of simple image $L$ isomorphic to $L(Y_{i,r}Y_{i,r+2})$ and kernel $K$ simple isomorphic to 
$\bigotimes_{j \in I \colon c_{i,j} = -1} L(Y_{ j,r+1})$.
The monoidal Jantzen filtrations are characterized by 
$$F_0M(\bep_s) = M(\bep_s)\supset F_1 M(\bep_s) = K\supset F_2M(\bep_s) = \{0\},$$
$$F_{-1}M(\bep_c)  = M(\bep_c) \supset F_0 M(\bep_c)  = L \supset F_1(M(\bep_c)) = \{0\}.$$
We obtain
$$[M(\bep_s)]_t = [L] + t[K], \quad [M(\bep_c)]_t = [L] + t^{-1}[K],
$$
$$[S_{i+(r+2)n}]*[S_{ i + rn }] = t^{\alpha}[L] + t^{\alpha + 1}[K], \quad
[S_{i+rn}]*[S_{i+(r+2)n}] = t^{-\alpha}[L] + t^{-\alpha - 1}[K],$$
where $\alpha  = -1 + (\tilde{c}_{i,i}(1) + \tilde{c}_{i,i}(3))/2$.

\noindent For $\mathfrak{g} = \mathfrak{sl}_3$, $i = 1$ and $r =0$, then $K\simeq S_4$ and $\alpha = -1/2$.

\begin{Rem}
In all the above examples, we find that every filter submodule $F_n M(\bep)$ of $M(\bep)$ can be expressed only in terms of the specialized $R$-matrices. 
 In particular, the monoidal Jantzen filtrations do not depend on the choice of deformations in these examples.
It would be interesting to study in which situation the monoidal Jantzen filtrations are characterized only by the specialized $R$-matrices.
\end{Rem}

\section{Monoidal Jantzen filtrations for symmetric quiver Hecke algebras}\label{mjqha}

We study our second main examples for monoidal Jantzen filtrations, 
realized in categories of representations 
of symmetric quiver Hecke algebras. 

We first give reminders on quantum unipotent coordinate rings with their 
PBW and canonical bases (Section \ref{qucr}). 
In Section \ref{Ssec:qH}, we recall their categorification in terms of representations 
of quiver Hecke algebras (Theorem \ref{Thm:VV}) which are compatible 
with specialization (Corollary \ref{Cor:VV}). Note that quiver Hecke algebras
have a natural grading and categorify the {\it quantum} 
unipotent coordinate rings (whereas quantum loop algebras above produce commutative Grothendieck rings),  although we work with ungraded modules in this paper. We recall the categories $\Cc_w$ of (ungraded) representations together with their PBW-theory from 
\cite{KKOP18} (Theorem \ref{pbwqh}). We construct the deformation of such a PBW-theory (Section \ref{ssec:dPBWqH}).
Hence we obtain monoidal Jantzen filtrations. 
We state the analog of
the quantum Grothendieck ring conjecture, which we
call the quantum unipotent  coordinate  ring conjecture 
(Conjectures \ref{conj:qH}, \ref{qHmix}).
It expresses our expectation that our monoidal Jantzen filtrations for ungraded modules should recover the graded Jordan-H\"older multiplicities for their graded counterparts.   

\subsection{Notation}

Let $C=(c_{ij})_{i,j \in I}$ be a \emph{symmetric} generalized Cartan matrix with $I$ being the set of Dynkin nodes.  
We write $i \sim j$ if $c_{ij} <0$.
We denote the associated Kac-Moody Lie algebra by $\fg$.
Let $\sQ$ be a free abelian  group  with a basis $\{ \alpha_i \}_{i \in I}$ endowed with the symmetric bilinear form $(-,-)$ given by $(\alpha_i, \alpha_j) = c_{ij}$.  
We set $\sQ^+ \seq \sum_{i \in I} \N \alpha_i \subset \sQ$.
For each $i \in I$, the simple reflection $s_i$ is defined by $s_i \alpha_j = \alpha_j - c_{ij}\alpha_i$.
The Weyl group $\sW$ is the subgroup of $\Aut(\sQ)$ generated by the simple reflections $\{ s_i \}_{i \in I}$.
The pair $(\sW, \{s_i\}_{i \in I})$ forms a Coxeter system.
The length of an element $w \in \sW$ is denoted by $\ell(w)$. 
The set of real roots is defined by $\sR \seq \bigcup_{i \in I} \sW \alpha_i$.
We have $\sR = \sR^+ \sqcup (-\sR^+)$ with $\sR^+ \seq \sR \cap \sQ^+$.

\subsection{Quantum unipotent coordinate rings}\label{qucr}
Let $t$ be an indeterminate. 
For $n \in \N$, we set $[n]_t \seq \frac{t^n-t^{-n}}{t-t^{-1}} \in \Z[t^{\pm 1}]$ and $[n]_t! \seq \prod_{k=1}^{n}[k]_t$.
Let $U_t^+(\fg)$ denote the positive half of the quantized enveloping algebra of $\fg$.
By definition, it is the $\Q(t)$-algebra presented by the generators $\{e_i\}_{i \in I}$ and the quantum Serre relations $\sum_{k=0}^{1-c_{ij}} e_{i}^{(k)}e_je_i^{(1-c_{ij}-k)} = 0$
for any $i,j \in I$ with $i \neq j$,
where $e_i^{(n)} \seq e_i^n/[n]_t!$ is the divided power.
The algebra $U_t^+(\fg)$ has the natural $\sQ^+$-grading $U_t^+(\fg) = \bigoplus_{\beta \in \sQ^+} U_t^+(\fg)_\beta$ with $e_i \in U_q^+(\fg)_{\alpha_i}$.
In addition, we have an algebra involution $\iota$ given by 
$\iota(t) = t^{-1}$ and $\iota(e_i) = e_i$ for any $i \in I$.
Let $U_t^+(\fg)_{\Z[t^{\pm 1}]}$ denote the $\Z[t^{\pm 1}]$-subalgebra generated by all the divided powers $\{e_i^{(n)}\}_{i \in I, n \in \N}$.
The algebra $U_t^+(\fg)_{\Z[t^{\pm 1}]}$ is free over $\Z[t^{\pm 1}]$ and has the canonical (or lower global) basis $\bB$ due to Lusztig and Kashiwara. 
Each element in $\bB$ is fixed by the involution $\iota$.
We give a review of Lusztig's construction of $\bB$ later in Section \ref{ssec:Lusztig_construction}.

We equip the tensor product $U_t^+(\fg)\otimes_{\Q(t)}U_t^+(\fg)$ with the structure of  $\Q(t)$-algebra by
\[ (x_1 \otimes x_2) \cdot (y_1 \otimes y_2) = t^{-(\beta_2, \gamma_1)} (x_1y_1 \otimes x_2y_2),  \]
where $x_i \in U_t^+(\fg)_{\beta_i}, y_i \in U_t^+(\fg)_{\gamma_i}$ for $i =1,2$.
There is a unique $\Q(t)$-algebra homomorphism 
\[ \mathrm{r} \colon U_t^+(\fg) \to U_t^+(\fg)\otimes_{\Q(t)}U_t^+(\fg)\]
satisfying $\mathrm{r}(e_i) = e_i \otimes 1 + 1 \otimes e_i$ for each $i \in I$.
Then, we have a unique non-degenerate symmetric bilinear pairing $\langle - , - \rangle$ on $U_t^+(\fg)$ satisfying 
\[ \langle 1,1 \rangle = 1, \quad  \langle e_i, e_j\rangle = \delta_{i,j}/(1-t^2), \quad \langle x, yz \rangle = \langle \mathrm{r}(x), y \otimes z \rangle\]
for any $x,y,z \in U_t^+(\fg)$, where $\langle x_1 \otimes x_2, y_1 \otimes y_2 \rangle \seq \langle x_1, y_1 \rangle \cdot \langle x_2 , y_2 \rangle$. 
Let $\iota'$ be the involution of $U_t^+(\fg)$ dual to $\iota$ with respect to $\langle -,- \rangle$.
By definition, it satisfies $\langle \iota'(x), y\rangle = \overline{\langle x, \iota(y)\rangle}$ for any $x, y \in U_t^+(\fg)$, 
where $\overline{f(t)} \seq f(t^{-1})$ for $f(t) \in \Q(t)$.
It is known that 
\begin{equation} \label{eq:iota*}
\iota'(xy) = t^{(\beta,\gamma)}\iota'(y)\iota'(x) 
\end{equation}
holds for any $x \in U^+_t(\fg)_\beta$ and $y \in U^+_t(\fg)_\gamma$.

Let $A_t[N]_{\Z[t^{\pm 1}]}$ be the dual of $U_t^+(\fg)_{\Z[t^{\pm 1}]}$, that is, 
\[ A_t[N]_{\Z[t^{\pm 1}]} = \{ x \in U_t^+(\fg) \mid \langle x, U_t^+(\fg)_{\Z[t^{\pm 1}]} \rangle \subset \Z[t^{\pm 1}]\}.\]
This is a $\Z[t^{\pm 1}]$-subalgebra of $U_t^+(\fg)$, endowed with the dual canonical basis $\bB^*$.
Each element of $\bB^*$ is fixed by the dual involution $\iota'$. 
The algebra $A_t[N]_{\Z[t^{\pm 1}]}$ is specialized at $t = 1$ to  a commutative ring, identical to the coordinate ring of the (pro-)unipotent group $N$ associated with the positive part of $\fg$.  
We call $A_t[N]_{\Z[t^{\pm 1}]}$ the quantum unipotent coordinate ring.

Fix $w \in \sW$.
We choose a reduced word $\ii = (i_1, \ldots, i_\ell) \in I^\ell$ for $w$, that is, we have  $w = s_{i_1} \cdots s_{i_\ell}$ and $\ell = \ell(w)$. 
In what follows, we set 
\begin{equation}
J \seq \{ j \in \Z \mid 1 \le j \le \ell \}. 
\end{equation}
For each $j \in J$, we define a real positive root $\alpha_{\ii, j} \in \sR^+$ by
\[ \alpha_{\ii,j} \seq s_{i_1}\cdots s_{i_{j-1}}\alpha_{i_j}.\]
Then, we have  $\{ \alpha_{\ii,j} \mid j \in J \} = \sR^+ \cap w(-\sR^+)$.
Correspondingly, we define the root vector $E_{\ii,j}$ and its dual $E^*_{\ii,j}$ for each $k \in J$ to be the elements of $U_t^+(\fg)_{\alpha_{\ii, j}}$ given by
\[ E_{\ii,j} \seq T_{i_1} \cdots T_{i_{j-1}}(e_{i_j}), \qquad E^*_{\ii,j} \seq (1-t^2)E_{\ii,j}, \] 
where $T_i$ denotes Lusztig's braid symmetry ($=  T''_{i,1} $ in Lusztig's notation, see \cite[37.1.3]{LusB} for its precise definition,  noting that our $t$ corresponds to $v^{-1}$ in loc.\ cit.).
We have $E_{\ii, j} \in U^+_t(\fg)_{\Z[t^{\pm 1}]}$ and $E^*_{\ii, j} \in A_t[N]_{\Z[t^{\pm 1}]}$.

Let $A_t[N(w)]_{\Z[t^{\pm 1}]}$ denote the $\Z[t^{\pm 1}]$-subalgebra of $A_t[N]_{\Z[t^{\pm 1}]}$ generated by $\{E^*_{\ii, j} \mid j \in J \}$.
As the notation suggests, this is independent of the choice of reduced word $\ii$, and can be thought of as the quantum coordinate ring of the unipotent group $N(w)$ corresponding to the finite-dimensional nilpotent Lie subalgebra $\bigoplus_{\alpha \in \sR^+ \cap w(-\sR^+)}\fg_\alpha$ of $\fg$. 

For each $\bd = (d_j)_{j \in J} \in \N^{\oplus J}$, we define 
\begin{equation} \label{eq:Ea} E_\ii^*(\bd) \seq t^{\sum_{j \in J} d_j(d_j-1)/2}\prod_{j \in J}^{\leftarrow}(E^*_{\ii,j})^{d_j}.
\end{equation}
Then, the set $\{E_\ii^*(\bd) \mid \bd \in \N^{\oplus J} \}$ forms a free $\Z[t^{\pm 1}]$-basis of $A_t[N(w)]_{\Z[t^{\pm 1}]}$, called the dual PBW basis associated to the reduced word $\ii$.   

\begin{Thm}[{\cite[Theorem 4.29]{Kimura}}] \label{Thm:Kimura}
There exists a unique free $\Z[t^{\pm 1}]$-basis $\{B^*_\ii(\bd)\mid\bd \in \N^{\oplus J}\}$ of $A_t[N(w)]_{\Z[t^{\pm 1}]}$ satisfying $\iota'B^*_\ii(\bd) = B^*_\ii(\bd)$ and
\[ E^*_\ii(\bd) = B^*_\ii(\bd) + \sum_{\bd' \prec \bd}c_\ii[\bd , \bd']B^*_\ii(\bd') \quad \text{for some $c_\ii[\bd,\bd'] \in t\Z[t]$},\]
for each $\bd \in \N^{\oplus J}$, where $\preceq$ is the bi-lexicographic ordering.
Moreover, we have 
\[\bB^*(w) \seq \bB^* \cap A_t[N(w)]_{\Z[t^{\pm 1}]} = \{B^*_\ii(\bd)\mid \bd\in \N^{\oplus J}\}.\]
\end{Thm}  

In particular, we have $E^*_{\ii,j} = E^*_{\ii}(\bdelta_j) = B^*_\ii(\bdelta_j) \in \bB^*(w)$ for each $j \in J$.

\subsection{Symmetric quiver Hecke algebra} \label{Ssec:qH}
Fix $\beta = \sum_{i \in I} b_i \alpha_i \in \sQ^+$ and set $|\beta|\seq \sum_{i \in I} b_i$.
Consider the finite set
\[ I^\beta \seq \{ \nu = (\nu_1, \ldots, \nu_{|\beta|}) \in I^{ |\beta|} \mid \alpha_{\nu_1} + \cdots + \alpha_{\nu_{|\beta|}} = \beta\}.\]
The symmetric group $\SG_{|\beta|}$ acts on $I^\beta$ by place permutations. 
We write $\sigma_k \in \SG_{|\beta|}$ for the transposition of $k$ and $k+1$ for each $1 \le k < |\beta|$.
Let $\kk$ be a field of characteristic $0$.
For each $i,j \in I$,  we choose a polynomial $Q_{ij}(u,v) \in \kk[u,v]$ of the form
\[
Q_{ij}(u,v) = \pm
(1-\delta_{i,j})(u-v)^{-c_{ij}}
\]
satisfying $Q_{ij}(u,v) = Q_{ji}(v,u)$. 

\begin{Def}
Let $\beta \in \sQ^+$ as above.
The symmetric quiver Hecke algebra $H_\beta$ is the $\Z$-graded $\kk$-algebra presented by the three kinds of generators 
$\{ e( \nu ) \mid \nu \in I^{\beta}\}$,
$\{x_{1}, \ldots, x_{|\beta|} \}$,
$\{\tau_{1}, \ldots, \tau_{|\beta|-1} \}
$
and the following relations:   
\[
e(\nu) e(\nu^{\prime}) = 
\delta_{\nu, \nu^{\prime}} e(\nu), \quad
\sum_{\nu \in I^{\beta}} e(\nu) = 1,
\quad
x_{k} x_{l} = x_{l} x_{k}, \quad
x_k e(\nu) = e(\nu) x_{k},  
\]
\[
\tau_{k} e(\nu) = e(\sigma_k\nu) \tau_k,
\quad
\tau_k \tau_l = \tau_l \tau_k \quad \text{if $|k-l| > 1$}, \quad
\tau_{k}^{2} e(\nu) = 
Q_{\nu_k,\nu_{k+1}}(x_k,x_{k+1})e(\nu),
\]
\[
(\tau_k x_l - x_{\sigma_{k}(l)} \tau_k)e(\nu)
= \delta_{\nu_k, \nu_{k+1}}(\delta_{l,k+1} - \delta_{l,k})  e(\nu),
\]
\[
(\tau_{k+1} \tau_{k} \tau_{k+1} - \tau_{k} \tau_{k+1} \tau_{k}) 
e(\nu)
= \delta_{\nu_k, \nu_{k+2}}
\frac{Q_{\nu_k,\nu_{k+1}}(x_k, x_{k+1}) - Q_{\nu_k,\nu_{k+1}}(x_{k+2},x_{k+1})}{x_k - x_{k+2}}e(\nu).
\]
We endow $H_\beta$ with a $\Z$-grading by 
\[ \deg(e(\nu)) = 0, \quad \deg(x_k) = 2, \quad \deg(\tau_k e(\nu)) = - c_{\nu_k, \nu_{k+1}}. \]
\end{Def}

We denote by $\Mm_\beta$ the category of left $H_\beta$-modules, and by $\Mm_\beta^\bullet$ the category of graded left $H_\beta$-modules (whose morphisms are homogeneous). 
We also denote by $\Mm_{\fdim,\beta} \subset \Mm_\beta$ and $\Mm_{\fdim,\beta}^\bullet \subset \Mm_\beta^\bullet$ the full subcategories of finite-dimensional modules.
 For $M \in \Mm_\beta^\bullet$, its $n$-th graded piece is denoted by $M^n$. 
For $k \in \Z$, the grading shift functor $M \mapsto M\langle k \rangle$ on $\Mm_\beta^\bullet$ is defined by $(M\langle k \rangle)^n = M^{n - k}$ for any $n \in \Z$. 

There is an anti-algebra involution of $H_\beta$  fixing all the generators $e(\nu), x_k$ and $\tau_k$. 
For a (graded) $H_\beta$-module $M$, we equip the (graded) dual vector space $M^\vee$ with the structure of left $H_\beta$-module by twisting the natural right module structure with the above anti-involution. 
We say that $M \in \Mm_{\fdim,\beta}^\bullet$ is self-dual if $M \simeq M^\vee$ as graded $H_\beta$-modules.
Every simple module in $\Mm_{\fdim,\beta}^\bullet$ is known to be self-dual after a grading shift.

For $\beta, \beta \in \sQ^+$, we consider an idempotent
\[ e(\beta, \beta') \seq \sum_{\nu \in I^\beta, \nu' \in I^{\beta'}} e(\nu * \nu') \quad \in H_{\beta + \beta'},\]
where $\nu * \nu' \in I^{\beta + \beta'}$ is the concatenation of the sequences $\nu$ and $\nu'$.   
Similarly, for $\beta_1, \ldots, \beta_n \in \sQ^+$, we define $e(\beta_1, \ldots, \beta_n) \in H_{\beta_1 + \cdots + \beta_n}$.

We regard $H_{\beta + \beta'}e(\beta, \beta')$ as a graded $(H_{\beta + \beta'}, H_\beta \otimes H_{\beta'})$-bimodule in a natural way.
For $M \in \Mm_\beta$ and $N \in \Mm_{\beta'}$, we define the convolution product $M \star N \in \Mm_{\beta + \beta'}$ by 
\[ M \star N \seq H_{\beta +\beta'} e(\beta, \beta')\otimes_{H_\beta \otimes H_{\beta'}} (M \otimes N).\]
It endows the category $\Mm \seq \bigoplus_{\beta \in \sQ^+}\Mm_\beta$ (resp.~$\Mm^\bullet \seq \bigoplus_{\beta \in \sQ^+}\Mm_\beta^\bullet$) with a structure of $\kk$-linear monoidal category (resp.~graded monoidal category). 
The subcategories $\Mm_\fdim \seq  \bigoplus_{\beta \in \sQ^+}\Mm_{\fdim,\beta}$ and $\Mm_{\fdim}^\bullet \seq \bigoplus_{\beta \in \sQ^+} \Mm_{\fdim,\beta}^\bullet$ are closed under these monoidal structures.
For any $M_k \in \Mm_{\beta_k}$, $k= 1, \ldots, n$, we have  a natural isomorphism
\[ M_{1} \star \cdots \star M_{n} \simeq H_{\beta} e(\beta_1, \ldots, \beta_n)\otimes_{H_{\beta_1,\ldots,\beta_n}} (M_1 \otimes \cdots \otimes M_n),\]
where $\beta = \sum_{k=1}^n \beta_k$ and $H_{\beta_1,\ldots,\beta_n} \seq H_{\beta_1} \otimes \cdots \otimes H_{\beta_n}$.

Let $K(\Mm_{\fdim}^\bullet)$ denote the Grothendieck ring of the category $\Mm_{\fdim}^\bullet$.
The following celebrated result is due to Khovanov-Lauda \cite{KL09}, Rouquier \cite{Rou08, Rou12} and Varagnolo-Vasserot \cite{VV11}. 
 
\begin{Thm}[\cite{KL09, Rou08, Rou12, VV11}] \label{Thm:VV}
There is an isomorphism of algebras
\begin{equation} \label{eq:VV} 
K(\Mm_\fdim^\bullet) \simeq A_t[N]_{\Z[t^{\pm 1}]},
\end{equation}
where the multiplication of $t^{\pm1}$ on the right hand side corresponds to the grading shift $\langle \pm 1 \rangle$ on the left hand side.
It induces a bijection between the set of the classes of self-dual simple modules and the dual canonical basis $\bB^*$.
\end{Thm}

Let $\Mm_{\fdim}^{\nilp} \subset \Mm_f$ be the full subcategory of modules on which the element $x_k$ acts nilpotently for all $k$.  
This is identical to the monoidal Serre subcategory generated by the image of the forgetful functor $\Mm_\fdim^\bullet \to \Mm_\fdim$.  
We think of $\Z$ as a $\Z[t^{\pm 1}]$-algebra through the specialization  $\Z[t^{\pm 1}] \to \Z$ at $t = 1$.
Let $A_t[N]|_{t=1} \seq A_t[N]_{\Z[t^{\pm 1}]}\otimes_{\Z[t^{\pm 1}]}\Z$.
This is a commutative ring endowed with the specialized dual canonical basis $\bB^*|_{t=1} \seq \bB^* \otimes 1$.

\begin{Cor} \label{Cor:VV}
There is an isomorphism of algebras
\[ K(\Mm_\fdim^{\nilp}) \simeq A_t[N]|_{t=1}\]
through which the basis formed by the classes of simple modules correspond to the specialized dual canonical basis $\bB^*|_{t=1}$.
\end{Cor} 

Let $\hH \seq \bigoplus_{\beta \in \sQ^+} \hH_\beta$, where $\hH_\beta$ denotes the completion of $H_\beta$ along the grading. 
The (non-unital) algebra $\hH$ and the $(\hH, \hH \otimes \hH)$-bimodule $\bigoplus_{\beta, \beta'} \hH_{\beta + \beta'}e(\beta, \beta')$ satisfy the assumptions in Section \ref{pbw}. 
Thus the category $\hH \mof$ is a monoidal category.
In addition, we have a natural isomorphism $\hH \mof \simeq \Mm_\fdim^{\nilp}$ of monoidal categories, and hence
\begin{equation} \label{eq:VV2} 
K(\hH \mof) \simeq A_t[N]|_{t=1} 
\end{equation}
through which the basis formed by the classes of simple modules corresponds to the specialized dual canonical basis $\bB^*|_{t=1}$.
In what follows, we identify $ \Mm_\fdim^{\nilp}$ with $\hH \mof$.

\subsection{Category $\Cc_w$ and PBW-theory} 
\label{ssec:pbwqH}
Let us fix an element $w \in \sW$.
We define the category $\Cc_{w}^\bullet$ (resp.~$\Cc_w$) to be the Serre subcategory of $\Mm_\fdim^\bullet$ (resp.~$\hH \mof$) generated by the simple modules corresponding to the elements of $t^\Z \bB^*(w)$ (resp.~$\bB^*(w)|_{t=1}$) under the isomorphism \eqref{eq:VV} (resp.~\eqref{eq:VV2}).
Theorem~\ref{Thm:Kimura} and Theorem~\ref{Thm:VV} (resp.~Corollary~\ref{Thm:VV}) tell us that the category $\Cc^\bullet_{w}$ (resp.~$\Cc_w$) is closed under the monoidal structure $\star$ and that we have the isomorphism 
\[ K(\Cc_{w}^\bullet) \simeq A_{t}[N(w)]_{\Z[t^{\pm 1}]} \qquad (\text{resp. }K(\Cc_w) \simeq A_t[N(w)]|_{t=1}),\]
where $A_t[N(w)]|_{t=1} \seq A_t[N(w)]_{\Z[t^{\pm 1}]}\otimes_{\Z[t^{\pm 1}]} \Z$ is the specialization at $t=1$.   

Now, let us choose a reduced word $\ii = (i_1, \ldots, i_\ell)$ for $w$. 
For each $j \in J$, let $L^\bullet_{\ii, j} \in \Cc_w^\bullet$ (resp.~ $L_{\ii,j}$) be a simple module whose isomorphism class corresponds to the dual root vector $E^*_{\ii,j}$ (resp.~$E^*_{\ii,j}|_{t=1}$) through the above isomorphism \eqref{eq:VV} (resp.~\eqref{eq:VV2}).
The module $L_{\ii,j}$ is obtained from $L^\bullet_{\ii,j}$ by forgetting the grading.
These modules are called
\emph{cuspidal modules}.
We recall the following fundamental result due to Kashiwara-Kim-Oh-Park \cite{KKOP18}.

\begin{Thm}[\cite{KKOP18}] \label{pbwqh}
For any $w \in \sW$ and any reduced word $\ii$ for $w$, the pair $(\{L_{\ii,j}\}_{j \in J}, \preceq)$ gives a PBW-theory of the monoidal category $\Cc_w$ in the sense of Definition~\ref{Def:PBW}, where $\preceq$ is the bi-lexicographic ordering on the set $\N^{\oplus J}$.
\end{Thm}

In what follows, given a reduced word $\ii$ for $w$, we write $M_\ii(\bd)$ and $L_\ii(\bd)$ respectively for the standard module and its simple head, and write $M_\ii(\bep)$ for the mixed tensor product, associated with the PBW-theory in Theorem \ref{pbwqh}.
Note that the class $[L_\ii(\bd)]$ corresponds to the specialized dual canonical basis element $B^*_\ii(\bd)|_{t=1}$ under the isomorphism \eqref{eq:VV2}.

\subsection{$R$-matrices and deformed PBW-theory}
\label{ssec:dPBWqH}

Let $\beta \in \sQ^+$. We define an element $\varphi_k$ of $H_\beta$ for each $1 \le k < |\beta|$ by
\[ \varphi_k e(\nu) \seq  \delta_{\nu_k, \nu_{k+1}}
(\tau_k x_k - x_k \tau_k)e(\nu) + (1-\delta_{\nu_k,\nu_{k+1}}) \tau_k e(\nu).
\] 
Since $\{\varphi_k\}_{1 \le k <|\beta|}$ satisfy the braid relations, we get a well-defined element $\varphi_{g}$ for each $g \in \mathfrak{S}_{|\beta|}$ by composing them so that we have $\varphi_g = \varphi_{i_1} \cdots \varphi_{i_n}$ if $g = \sigma_{i_1} \cdots \sigma_{i_n}$ is a reduced expression.
For any $M \in \Mm_\beta$ and $M' \in \Mm_{\beta'}$, we have the unique $H_{\beta + \beta'}$-homomorphism 
\[ R_{M,M'} \colon M \star M' \to M' \star M\]
extending the $H_\beta \otimes H_{\beta'}$-homomorphism $M\otimes M' \to M' \star M$ given by $v \otimes v' \mapsto \varphi_{\sigma}e(\beta', \beta)(v' \otimes v)$, where $\sigma \in \mathfrak{S}_{|\beta|+|\beta'|}$ is the permutation defined by $\sigma(k) \seq k + (-1)^{\delta(k > |\beta'|)}|\beta'|$.
Note that $R_{M,M'}$ also yields an $\hH$-homomorphism if $M, M'$ are $\hH$-modules.
By construction, they satisfy the quantum Yang-Baxter equation, that is, we have
\begin{align} \label{eq:qYBR}
&(R_{M', M''} \star \id_M) \circ (\id_{M'} \star R_{M,M''})\circ (R_{M, M' } \star \id_{M''})  \\
&= (\id_{M''} \star R_{M,M'}) \circ (R_{M,M''}\star \id_{M'}) \circ (\id_M \star R_{M',M''})
\end{align}
for any $\hH$-modules $M,M',M''$.

Next, we introduce deformations.
Let $z$ be an indeterminate and set $\Oo = \kk[\![z]\!]$, $\Kk = \kk(\!(z)\!)$ as before.
For $M \in \hH \mof$, we define its deformation $M_{a(z)}$ with $a(z) \in z\Oo$ to be the $\Oo$-module $M \otimes \Oo$ equipped with the $\hH$-action given by 
\begin{align}
e(\nu) \cdot (v \otimes f(z)) & \seq e(\nu) v \otimes f(z), \nonumber \\
x_k \cdot (v \otimes f(z)) & \seq x_k v \otimes f(z) + v \otimes a(z)f(z), \label{eq:affx} \\
\tau_{l} \cdot (v \otimes f(z)) & \seq \tau_l v \otimes f(z) \nonumber 
\end{align} 
for any $v \in M$ and $f(z) \in \Oo$.
Therefore, $M_{a(z)}$ is an $\hH_\Oo$-module such that $(M_{a(z)})_0 \simeq M$.

Although the following result is essentially due to \cite[\S2.3]{KP18}, we shall give a proof for completeness. 
Recall that a simple module $M \in \hH \mof$ is said to be \emph{real} if $M \star M$ is simple.

\begin{Lem} \label{Lem:KP}
Let $M$, $N$ be simple modules in $\hH \mof$, and $a(z), b(z) \in z \Oo$ with $a(z) \neq b(z)$.  
\begin{enumerate}
\item \label{Lem:KP1} We have an isomorphism of $\hH_\Kk$-modules
\[(M_{a(z)} \star_\Oo N_{b(z)})_\Kk \simeq (N_{b(z)} \star_\Oo M_{a(z)})_\Kk, \]
induced from $R_{M_{a(z)}, N_{b(z)}}$.
\item \label{Lem:KP2} Assuming that at least one of $M$ and $N$ is real, we have an isomorphism
\[\End_{\hH_\Kk}((M_{a(z)} \star_\Oo N_{b(z)})_\Kk) \simeq \Kk \id. \]
\end{enumerate}
\end{Lem}
\begin{proof}
By \cite[Proposition 1.10]{KKK}, for any $M' \in H_\beta \Mod$ and $N' \in H_{\beta'} \Mod$, the homomorphism $(R_{N',M'} \circ R_{M',N'})|_{e(\beta, \beta')(M' \otimes N')}$ is given by the multiplication by
\[ X \seq \sum_{\nu \in I^\beta, \nu' \in I^{\beta'}} \left(\prod_{1 \le k \le d, 1 \le l \le d', \nu_k \neq \nu'_l}Q_{\nu_k, \nu'_l}(x_k, x_{d+l})\right) e(\nu * \nu').\]
This element
$X$ is in $Z_+(H_\beta) \otimes H_{\beta'} + H_\beta \otimes Z_+(H_{\beta'})$, where $Z_+(H_\beta)$ denotes the positive degree part of the center of $H_\beta$. 
Consider the case when $M' = M_{a(z)}$ and $N' = N_{b(z)}$.
Since $Z_+(H_\beta)$ acts by zero on a simple module, the action of $X$ on $e(\beta, \beta')(M_{a(z)} \otimes_{\Oo} N_{b(z)})$ becomes the multiplication by  
\[ \pm\sum_{\nu \in I^\beta, \nu' \in I^{\beta'}}\left(\prod_{1 \le k \le d, 1 \le l \le d', \nu_k \neq \nu'_l}\left(a(z)-b(z)\right)^{-c_{\nu_k, \nu'_l}}\right) e(\nu * \nu') = \pm(a(z)-b(z))^N e(\beta, \beta')\]
for some $N \in \N$. 
It is invertible after the localization as $a(z) \neq b(z)$.
Thus, \eqref{Lem:KP1} is proved.

To prove \eqref{Lem:KP2}, it is enough to show $\End_{\hH_\Oo}(M_{a(z)} \star_\Oo N_{b(z)}) = \Oo \id$ as we have 
\[\End_{\hH_\Kk}((M_{a(z)} \star_\Oo N_{b(z)})_\Kk) \simeq \End_{\hH_\Oo}(M_{a(z)} \star_\Oo N_{b(z)})\otimes_\Oo \Kk. \]
For simplicity, put $T \seq M_{a(z)} \star_\Oo N_{b(z)}$.
Then, $T_0 = T/zT \simeq M \star N$.
By the assumption and \cite[Proposition 3.8]{kkko}, we have
\begin{equation} \label{eq:simplicity} 
\End_{\hH}(T_0) \simeq \End_{\hH}(M \star N) \simeq \kk \id. 
\end{equation}
Let $f \in \End_{\hH_\Oo}(T)$ be a non-zero homomorphism.  
There exists a unique integer $s \in \N$ such that $f(T) \subset z^sT$ and $f(T) \not \subset z^{s+1}T$.
By \eqref{eq:simplicity}, there exists a unique $c_s \in \kk^\times$ such that $(f - c_s z^s \id)(T) \subset z^{s+1}T$.
Repeating the same argument, we inductively find for any integer $l \ge s$ a scalar $c_l \in \kk$ such that $(f - \sum_{k=s}^{l}c_kz^k \id)(T) \subset z^{l+1}T$.
Then, we get $f = (\sum_{k \ge s}c_k z^k)\id$ as $\bigcap_{k \in \N} z^k T = \{0\}$, which proves $\End_{\hH_\Oo}(T) = \Oo \id$ as desired.
\end{proof}

Let $(\{L_{\ii,j}\}_{j \in J}, \preceq)$ be the PBW-theory  in Theorem \ref{pbwqh} associated with a reduced word $\ii$ of $w$.
We define a collection $\{\tL_{\ii,j}\}_{j \in J}$ of $\hH_\Oo$-modules by
\begin{equation} \label{eq:tLik} 
\tL_{\ii,j} \seq (L_{\ii, j})_{jz} 
\end{equation}
for each $j \in J$.

\begin{Cor}
The collection $\{\tL_{\ii,j}\}_{j \in J}$ defined in \eqref{eq:tLik} gives a normal, consistent, generically commutative deformation of $\{L_{\ii,j}\}_{k \in J}$ in the sense of Section \ref{secun}.  
\end{Cor}
\begin{proof}
It is known that the simple module $L_{\ii,j}$ is real (cf.~\cite[Proposition 4.2]{KKOP18}).\
Then, it is clear from Lemma~\ref{Lem:KP} that the collection $\{L_{\ii,j}\}_{j \in J}$ gives a generically commutative deformation.
Since the renormalized $R$-matrix $R_{ij}$ in this case is induced from the homomorphism $z^{s_{ij}}R_{\tL_{\ii,i}, \tL_{\ii,j}}$ with $s_{ij}$ being a uniquely defined integer, the consistency follows from the quantum Yang-Baxter equation~\eqref{eq:qYBR}. 
The normality is proved in \cite[Proposition 2.11]{KK19}.
\end{proof}

\subsection{Quantum unipotent coordinate ring conjecture}
We state the analog of the quantum Grothendieck ring conjecture for the quiver Hecke algebras, which we call the quantum unipotent  coordinate  ring conjecture.
For this purpose, we need to introduce a renormalization of the dual canonical basis.

Let $t^{1/2}$ be a formal square root of the indeterminate $t$, and let 
\[ A_t[N(w)]_{\Z[t^{\pm 1/2}]} \seq A_t[N(w)]_{\Z[t^{\pm 1}]} \otimes_{\Z[t^{\pm 1}]}\Z[t^{\pm 1/2}]. \]
Note that we have $(\beta, \beta) \in 2 \Z$ for any $\beta \in \sQ$.
For a homogeneous element $x \in A_t[N(w)]_{\Z[t^{\pm 1/2}]}$ of degree $\beta \in \sQ^+$, we write $\tilde{x} \seq t^{-(\beta, \beta)/4}x$.
In particular, for any reduced word $\ii$ for $w$ and $\bd = (d_j)_{j \in J}\in \N^{\oplus J}$, we write
\[ \tB^*_\ii(\bd) = t^{-(\beta, \beta)/4}B^*_\ii(\bd), \qquad \tE^*_\ii(\bd) = t^{-(\beta,\beta)/4}E^*_\ii(\bd), \]
where $\beta \seq \sum_{j \in J}  d_j \alpha_{\ii, j}$.
We define the renormalized involution $\ol{(\cdot)}$ of $A_t[N(w)]_{\Z[t^{\pm 1/2}]}$ by $\ol{x} \seq t^{-(\beta, \beta)/2}\iota'(x)$ if $x$ is homogeneous of degree $\beta$, 
so that it fixes each renormalized dual canonical basis element $\tB^*_\ii(\bd)$.
The identity~\eqref{eq:iota*} implies that $\ol{(\cdot)}$ is an anti-involution, i.e., we have $\ol{x \cdot y} = \ol{y} \cdot \ol{x}$
for any $x, y \in A_t[N(w)]_{\Z[t^{\pm 1/2}]}$.

\begin{Rem}
By \cite{GLS13} (and \cite{kkko}), the algebra $A_t[N(w)]_{\Z[t^{\pm 1/2}]}$ has the structure of quantum cluster algebra in the sense of \cite{BZ05}.   
The anti-involution $\ol{(\cdot)}$ coincides with the natural bar-involution of the quantum cluster algebra. 
\end{Rem}

In terms of the renormalized elements, the equation~\eqref{eq:Ea} is rewritten as 
\begin{equation} \label{eq:tE}
\tE_\ii^*(\bd) = t^{-\sum_{1 \le j < k \le \ell} d_j d_k (\alpha_{\ii, j}, \alpha_{\ii, k})/2}\prod_{j \in J}^{\leftarrow}(\tE^*_{\ii,j})^{d_j}
\end{equation}
Comparing with \eqref{eq:*mix}, we define the skew-symmetric bilinear map $\gamma_\ii  \colon \N^{\oplus J} \times \N^{\oplus J} \to \frac{1}{2}\Z$ by
\[ 
\gamma_\ii(\bd, \bd') \seq \frac{1}{2}\sum_{1 \le j < k \le \ell} (d_j d'_k-d_kd'_j) (\alpha_{\ii,j}, \alpha_{\ii,k}).
\]
With the consistent deformation $\{\tL_{\ii,j}\}_{j \in J}$ constructed in Section \ref{ssec:dPBWqH} and $\gamma = \gamma_\ii$ defined as above, we obtain the associated bilinear operation $* = *_{\gamma_\ii}$ on $K(\Cc_{w})_t$.
Be aware that it depends on the choice of reduced word $\ii$.

We define a $\Z[t^{\pm 1/2}]$-linear isomorphism $\phi \colon K(\Cc_w)_t \simeq A_t[N(w)]_{\Z[t^{\pm 1/2}]}$ by $\phi([L_\ii(\bd)]) = \tB^*_\ii(\bd)$ for all $\bd \in \N^{\oplus J}$. 
Note that this isomorphism $\phi$ does not depend on the choice of reduced word $\ii$.
Clearly, we have $\phi \circ \ol{(\cdot)} = \ol{(\cdot)} \circ \phi$.

\begin{conj}[Quantum Unipotent Coordinate Ring Conjecture] 
\label{conj:qH}
With a chosen reduced word $\ii$ for $w$,  
Associativity Conjectures \ref{Conj:assoc} and \ref {Conj:sassoc} hold for $(K(\Cc_{w})_t,*)$, and the linear isomorphism $\phi$ gives a $\Z[t^{\pm 1/2}]$-algebra isomorphism from $(K(\Cc_{w})_t,*)$ to  the  quantum unipotent coordinate ring $A_t[N(w)]_{\Z[t^{\pm 1/2}]}$.
\end{conj}

\begin{Rem}If Conjecture \ref{conj:qH} is true for any reduced word $\ii$ for $w$, it implies that the ring structure $(K(\Cc_{w})_t,*)$ does not depend on the choice of $\ii$.
\end{Rem} 

Assume that Conjecture \ref{conj:qH} is true for a chosen reduced word $\ii$ for a while.
Then,  we have
\[ \phi([M_\ii(\bd)]_t) = \tE^*_{\ii}(\bd)\]
for any $\bd \in \N^{\oplus J}$.     
More generally, for any $d \in \N$ and sequence $\bep = (\ep_1, \ldots, \ep_{d}) \in J^d$, 
letting 
\begin{equation} \label{eq:Mit}
\tE^*_{\ii}(\bep) \seq t^{\sum_{1 \le k < l \le  d}\gamma_{\ii}(\bdelta_{\ep_k}, \bdelta_{\ep_l})}  \tE^*_{\ii,\ep_1}  \cdots  \tE^*_{\ii, \ep_{d}},
\end{equation}
we obtain the equality 
\begin{equation} \label{eq:mixqH}
\phi([M_\ii(\bep)]_t) = \tE^*_\ii(\bep). 
\end{equation}
The converse statement is true.
Namely, Conjecture \ref{conj:qH} is equivalent to the following.

\begin{conj}\label{qHmix}
The equality \eqref{eq:mixqH} holds for any $d \in \N$ and  $\bep \in J^d$. 
\end{conj}

\begin{Rem}
Note that Conjecture \ref{qHmix} also implies  Duality Conjecture \ref{Conj:duality} in this case. Indeed, for each $\bep\in J^d$, we have 
$\overline{\phi([M_\ii(\bep)]_t)} = \ol{\tE^*_\ii(\bep)} = \tE^*_\ii(\bep^{\op}) = \phi([M_\ii(\bep^\op)]_t)$, and hence $\ol{[M_\ii(\bep)]_t} = [M_\ii(\bep^\op)]_t$. 
\end{Rem}

\begin{Def} \label{Def:adapted}
Let $Q$ be a quiver.
We understand it as a quadruple $Q = (Q_0, Q_1, \mathrm{s}, \mathrm{t})$, where $Q_0$ is the set of vertices, $Q_1$ is the set of arrows and $\mathrm{s}$ (resp.~$\mathrm{t}$) is the map $Q_1 \to Q_0$ assigning an arrow with its source (resp.~target).    
We say that a quiver $Q$ without edge loops is of type $\fg$ if $Q_0 = I$ and, for any $i, j \in I$ with $i \neq j$, we have
\[ -c_{ij} = -c_{ji}=\#\{ a \in Q_1 \mid \{\mathrm{s}(a), \mathrm{t}(a)\} = \{i,j\}\}.\] 
A vertex $i$ is called a \emph{source} (resp.~\emph{sink}) of the quiver $Q$ if there is no arrow $a \in Q_1$ with $i = \mathrm{t}(a)$ (resp.~$i = \mathrm{s}(a)$).
A sequence $\ii = (i_1, i_2, \ldots, i_\ell)$ in $I$ is said to be \emph{adapted to $Q$} if the vertex $i_k$ is a source of the quiver $s_{i_{k-1}} \cdots s_{i_2}s_{i_1}Q$ for any $1 \le k \le \ell$, where $s_i Q$ denotes the quiver obtained from $Q$ by inverting the orientations of all the arrows incident to $i$.
\end{Def}

When our reduced word $\ii$ for $w$ is adapted to a quiver $Q$ of type $\fg$, we have a geometric realization of the deformed PBW theory $\{\tL_{\ii,j}\}_{j \in J}$ and their mixed products (see Section \ref{ssec:PrqH} below).    
In this case, we establish that our conjectures are true with the help of geometry. 

\begin{Thm} \label{Thm:adapted}
When our reduced word $\ii$ for $w$ is adapted to a quiver $Q$ of type $\fg$, Conjecture \ref{qHmix} holds, and hence Conjecture \ref{conj:qH} holds.
\end{Thm}

A proof will be given in Section \ref{ssec:PrqH} below.

\section{Preliminaries for geometric proofs}
\label{sec:preliminary}

In the remaining part of this paper, we prove our main Theorems \ref{simply-laced} and \ref{Thm:adapted} with the help of geometry.
In this section, before going into individual discussions, we recall some preliminary facts commonly used in the proofs. 
They are based on Grojnowski's unpublished note~\cite{Groj}, ``Fundamental Example'' of Bernstein-Lunts \cite{BL94} (also outlined in \cite{Groj}), and the hyperbolic localization theorem due to Braden \cite{Braden}.  

\subsection{Hard Lefschetz property}
\label{ssec:Groj1}

Let $\kk$ be a field and $z$ an indeterminate.
For a $\kk[z]$-module $M$, we often write $z_M \colon M \to M$ for the endomorphism given by the action of $z$. 
We endow the polynomial ring $\kk[z]$ with a $\Z$-grading by setting $\deg z \seq 2$.
Let $\kk[z] \gMod$ be the category of $\Z$-graded $\kk[z]$-modules. 
For $M \in \kk[z] \gMod$, its $n$-th graded piece is denoted by $M^n$. 
For $k \in \Z$, the grading shift functor $M \mapsto M\langle k \rangle$  is defined by  $(M\langle k \rangle)^n = M^{n - k}$  for any $n \in \Z$.
For each $n \in \Z$, we set $M^{\ge n} \seq \bigoplus_{k \ge n} M^k$, which is a graded $\kk[z]$-submodule of $M$.

\begin{Def}
We say that a module $M \in \kk[z] \gMod$ satisfies the \emph{hard Lefschetz property} if the endomorphism $z_M^n$ restricts to a $\kk$-linear isomorphism $M^{-n} \xrightarrow{\sim} M^n$ for any $n \in \N$.
\end{Def}

\begin{Lem} \label{Lem:hL}
Let $M$ be a finitely generated $\Z$-graded $\kk[z]$-module satisfying the hard Lefschetz property.
For any $n \in \Z$, we have
\[ M^{\ge n} = \sum_{k,l \in \N \colon k  - l = n} \Image (z_M^k) \cap \Ker(z_M^{l+1}). \] 
\end{Lem}
\begin{proof}
Note that a finitely generated $\Z$-graded $\kk[z]$-module is bounded from below with all its graded pieces being finite-dimensional.
Since $M$ satisfies the hard Lefschetz property, it is finite-dimensional and decomposes into a finite direct sum of the modules of the form $M_p \seq (\C[z]/z^{p+1}\C[z])\langle  -p  \rangle$ for various $p \in \N$. 
Thus, it suffices to prove the assertion when $M = M_p$.
From the definition of $M_p$, we have $\Image(z_{M_p}^k) = z^k M_p$ and $\Ker(z_{M_p}^{l+1}) = z^{p-l}M_p$. 
Therefore, we have
\[\sum_{k-l=n} \Image(z_{M_p}^k) \cap \Ker(z_{M_p}^{l+1}) = \sum_{0 \le k \le p+n} z^{\max(k, p+n-k)}M_p = z^{\lceil (p+n)/2\rceil}M_p.\]
Observe that $z^k M_p = M_p^{\ge -p+2k}$ for any $k \in \N$. 
If $p+n$ is even, we have $2\lceil (p+n)/2\rceil = p+n$ and hence $z^{\lceil (p+n)/2\rceil}M_p = M_p^{\ge n}$, which implies the assertion.  
If $p+n$ is odd, we have  $2\lceil (p+n)/2\rceil = p+n+1$ and hence $z^{\lceil (p+n)/2\rceil}M_p =M_p^{\ge n+1}$, which also implies the assertion as $M_p^{n} = 0$ in this case.
\end{proof}

Assume that there is a short exact sequence 
\[ 0 \to \sM \to \coM \to N\langle  -1  \rangle \to 0 \]
in $\kk[z]\gMod$ satisfying the following three conditions:
\begin{itemize}
\item[(i)] The modules $\sM$ and $\coM$ are free of finite rank over $\kk[z]$;
\item[(ii)] Setting $\sbM \seq \sM/ z \sM$ and $\cbM \seq \coM/z \coM$, we have $(\sbM)^{-n} = 0$ and $(\cbM)^n = 0$ for any $n > 0$; 
\item[(iii)] The module $N$ satisfies the hard Lefschetz property.
\end{itemize}
In what follows, we regard $\sM$ as a $\kk[z]$-submodule of $\coM$ through the given injection. 

\begin{Lem}[\cite{Groj}] \label{Lem:Lef}
With the above assumption, the graded $\kk[z]$-module $L \seq \coM / z \sM$ satisfies the hard Lefschetz property.
\end{Lem}
\begin{proof}
From the condition~(i), the endomorphism $z_{\coM}$ is injective and hence we have $\sbM \simeq \Ker( z_L)$. 
In addition, we have the natural isomorphisms $\cbM \simeq \Cok (z_{L})$ and $N \langle  -1  \rangle \simeq L/\sbM$. 
These isomorphisms give the exact sequences
\[ 0 \to \sbM \to L \xrightarrow{a} N \langle   -1  \rangle \to 0 \quad \text{and} \quad 
0 \to N \langle  1  \rangle \xrightarrow{b} L \to \cbM  \to 0\]
in $\kk[z] \gMod$ satisfying $b \langle  -2  \rangle \circ a = z_L \colon L \to L \langle  -2  \rangle$. By the condition~(ii), for any $n > 0$, the homomorphisms $a$ and $b$ induce the $\kk$-linear isomorphisms $a_{-n} \colon L^{-n} \xrightarrow{\sim} N^{-n+1}$ and $b_n \colon N^{n-1} \xrightarrow{\sim} L^n$ respectively.
Now, for each $n > 0$, we have the commutative diagram 
\[\vcenter{
\xymatrix{
L^{-n} \ar[r]^-{z_{L}^{n-1}} \ar[d]_-{a_{-n}}& L^{n-2} \ar[r]^-{z_{L}} \ar[d]_-{a_{n-2}}& L^n \\
N^{-n+1} \ar[r]^-{z_{N}^{n-1}} & N^{n-1} \ar[ru]_-{b_n}&
}}
\]
with the bottom arrow being an isomorphism by the condition~(iii).
Therefore, the $\kk$-linear map $z_L^n = z_L \circ z_L^{n-1}$ gives an isomorphism $L^{-n} \xrightarrow{\sim} L^n$ for any $n > 0$.
\end{proof}

\subsection{Notation around equivariant sheaves}
\label{ssec:Groj2}
In this subsection, we assume that $\kk$ is a field of characteristic zero.  
Let $G$ be a complex linear algebraic group.
By a $G$-variety, we mean a complex algebraic variety endowed with an algebraic action of $G$.
For a $G$-variety $X$, let $D_G^b(X, \kk)$ denote the $G$-equivariant bounded derived category of constructible complexes of sheaves of $\kk$-vector spaces on $X$ in the sense of Bernstein-Lunts \cite{BL94}.
It is a $\kk$-linear triangulated category, whose shift is denoted by $[1]$.  
It is endowed with the perverse $t$-structure, whose heart $\Perv_G(X, \kk)$ is the category of the $G$-equivariant perverse sheaves. 
When $G$ is a trivial group $G =\{1\}$, we simply write $D^b(X, \kk)$ and $\Perv(X, \kk)$ dropping the symbol $G$.

For $\mathcal{F, G} \in D^b_G(X,\kk)$, we abbreviate $\Hom_{D_G^b(X, \kk)}(\mathcal{F,G})$ as $\Hom_{G}(\mathcal{F,G})$, and for $n \in \Z$, we set $\Hom_G^n(\mathcal{F,G}) \seq \Hom_G(\mathcal{F,G}[n])$.
Letting $\ul{\kk}_X$ be the constant $\kk$-sheaf on $X$, we set $\rH_G^n(\cF) \seq  \Hom_G^n(\ul{\kk}_X, \cF)$.
The $\Z$-graded $\kk$-vector spaces 
$\Hom_G^\bullet(\mathcal{F,G}) \seq \bigoplus_{n \in \Z} \Hom^n_G(\mathcal{F,G})$ and $\rH^\bullet_G(\mathcal{F}) \seq \bigoplus_{n \in \Z} \rH^n_G(\mathcal{F})$
are graded modules over $\rH_G^\bullet(\pt,\kk) = \rH_G^\bullet(\ul{\kk}_\pt)$ (the $G$-equivariant cohomology ring of a point).
 
The Verdier duality of $D_G^b(X, \kk)$ is denoted by $\bD_X$, or simply by $\bD$. 
For an equivariant morphism $f$ of $G$-varieties, we use the symbols $f^*, f_*, f^!, f_!$ for the associated functors of the $G$-equivariant derived categories.  
Given a homomorphism of algebraic groups $\varphi \colon G' \to G$, we regard $X$ as an $G'$-variety through $\varphi$. 
Then, we have a natural functor $\Res_\varphi \colon D^b_G(X, \kk) \to D^b_{G'}(X, \kk)$, which commutes with the Verdier duality and all the functors $f^*, f_*, f^!, f_!$ above, see~\cite[Proposition 7.2]{BL94}.  
When $\varphi$ is the trivial inclusion $\{1\} \hookrightarrow G$, the functor $\Res_\varphi$ is identical to the forgetful functor $\For \colon D^b_G(X, \kk) \to D^b(X, \kk)$.
When $G$ is connected, $\For$ induces a full embedding $\Perv_G(X, \kk) \hookrightarrow \Perv(X, \kk)$, see~\cite[Proposition 6.2.15]{AcharBook}, through which we think of $\Perv_G(X, \kk)$ as a full subcategory of $\Perv(X, \kk)$.

We denote by $\IC(X, \kk)$ the intersection cohomology complex of $X$.
This is a simple object of $\Perv_G(X, \kk)$.
We set $\IH^\bullet_G(X, \kk) \seq \rH^\bullet_G(\IC(X, \kk))$.
  
\subsection{``Fundamental Example" of Bernstein-Lunts} 
In the reminder of this section, we consider the following situation.
Let $E$ be a finite dimensional $\C$-vector space endowed with a linear action of a complex algebraic torus $T$. 
Let $X^*(T)$ (resp.~$X_*(T)$) denote the character (resp.~cocharacter) lattice of $T$. 
We have the weight space decomposition 
$E=\bigoplus_{\lambda \in X^*(T)}E_\lambda$.
We assume that the $T$-action on $E$ is \emph{attractive}, that is,   
\begin{equation} \label{eq:attractive}
\text{there exists $\rho^\vee \in X_*(T)$ such that $\langle \rho^\vee, \lambda \rangle > 0$ for any $\lambda \in \wt(E)$,}
\end{equation} 
where $\wt(E) \seq \{ \lambda \in X^*(T) \mid E_\lambda \neq \{ 0 \}\}$ and $\langle -, -\rangle \colon X_*(T) \times X^*(T) \to \Z$ denotes the natural pairing.
It particularly implies that the $T$-fixed locus $E^T$ consists of a single point $0 \in E$ and $\lim_{s \to 0}\rho^\vee(s) \cdot x =0$ for all $x\in E$.

Let $i \colon \{0\} \to E$ and $p \colon E \to \{0\}$ be the obvious morphisms.  
Applying $p_*$ and $p_!$ respectively to the adjunction morphisms $\id \to i_* i^*$ and $i_!i^! \to \id$, we get the natural morphisms 
\begin{equation} \label{eq:attrmor}
p_* \to i^* \quad \text{and} \quad i^! \to p_!
\end{equation}
of functors from $D^b_T(E, \kk)$ to $D^b_T(\{0\}, \kk)$.

\begin{Prop}[{\cite[Proposition 2.3]{FW14}}]
\label{Prop:attr}
The morphisms in \eqref{eq:attrmor} are isomorphisms.
\end{Prop}

In what follows, we fix a cocharacter $\rho^\vee$ satisfying \eqref{eq:attractive} and regard $E$ as a $\Gm$-variety through $\rho^\vee \colon \Gm \to T$. 
We make an identification $\rH^\bullet_{\Gm}(\pt, \kk) = \kk[z]$ with $\deg z = 2$. 
Note that the condition~\eqref{eq:attractive} particularly implies that the stabilizer in $\Gm$ of a point $x \in E \setminus \{0\}$ is always finite.
For any closed $\Gm$-subvariety $X \subset E$, we consider the quotient $\bP_{\rho^\vee}X \seq (X \setminus \{0\})/\Gm$, which is projective as a closed subvariety of the weighted projective space $\bP_{\rho^\vee}E$. 

\begin{Prop}[\cite{BL94}]
\label{Prop:BL}
For any $\Gm$-stable closed variety $X$ of $E$, we have an isomorphism 
\[ \IH_{\Gm}^\bullet(X\setminus\{0\}, \kk) \simeq \IH^\bullet(\bP_{\rho^\vee}X, \kk)\langle  -1  \rangle \]
of finite-dimensional $\Z$-graded $\kk$-vector spaces, under which the action of $z \in \kk[z] = \rH^\bullet_{\Gm}(\pt, \kk)$ on the LHS corresponds the Lefschetz operator (i.e., multiplication by the first Chern class of an ample line bundle) on the RHS up to multiples in $\kk^\times$.
\end{Prop}
\begin{proof}
The existence of the isomorphism follows from \cite[Theorem 9.1]{BL94} (here, we need the assumption that $\kk$ is of characteristic zero).
The latter assertion is \cite[Lemma 14.5]{BL94}.
\end{proof}

Let $j \colon E \setminus \{0\} \hookrightarrow E$ be the open inclusion of the complement of $\{0\}$.
Let $X \subset E$ be a closed $T$-subvariety.
Applying $i^*$ to the standard exact triangle
\[ i_! i^! \IC(X,\kk) \to \IC(X,\kk) \to j_*j^* \IC(X,\kk) \xrightarrow{+1}, \]
we get the exact triangle
\begin{equation} \label{eq:fexd}
i^! \IC(X,\kk) \to i^* \IC(X,\kk) \to i^*j_*j^* \IC(X,\kk) \xrightarrow{+1} 
\end{equation}
in $D^b_T( \{0\}, \kk)$.
By applying $\rH^\bullet_{\Gm}(-) \circ \Res_{\rho^\vee}$ to the third term and using Propositions~\ref{Prop:attr} and \ref{Prop:BL}, we obtain the isomorphisms
\[ \rH^\bullet_{\Gm}(i^*j_*j^* \IC(X,\kk)) \simeq \rH_{\Gm}^\bullet((p\circ j)_*\IC(X \setminus \{0\}, \kk)) \simeq \IH^\bullet(\bP_{\rho^\vee}X, \kk)\langle  -1  \rangle. \]

\begin{Thm}[``Fundamental Example'' {\cite{BL94}}] \label{Thm:Fex}
Applying $\rH^\bullet_{\Gm}(-) \circ \Res_{\rho^\vee}$ to the exact triangle \eqref{eq:fexd} yields a short exact sequence
\begin{equation} \label{eq:fex}
0 \to \rH_{\Gm}^\bullet(i^! \IC(X,\kk)) \to \rH_{\Gm}^\bullet(i^* \IC(X,\kk)) \to \IH^\bullet(\bP_{\rho^\vee}X, \kk)\langle  -1  \rangle \to 0
\end{equation}
in $\kk[z] \gMod$ satisfying the conditions {\rm(i), (ii), (iii)} in Section \ref{ssec:Groj1} above. 
\end{Thm}
\begin{proof}
When $\kk$ is the field of real numbers, the assertion is proved in \cite[Section 14]{BL94}.
To deal with the general case, it is enough to consider the case when $\kk$ is the field of rational numbers.
For this case, we may employ the fact that $\IC(X,\kk)$ underlies a simple $\Gm$-equivariant mixed Hodge module of pure weight $0$, and both functors $i^*$ and $i^!$ preserve the purity thanks to Proposition~\ref{Prop:attr}. 
Then, it follows that the connecting homomorphisms in the long exact sequence obtained by applying $\rH_{\Gm}^0(-) \circ \Res_{\rho^\vee}$ to \eqref{eq:fexd} are all zero.
See \cite[Proof of Proposition 4.4]{BBDVW22} for more details. 
Together with Proposition~\ref{Prop:BL}, it leads to the desired short exact sequence~\eqref{eq:fex}. 
The conditions (i) and (ii) can be verified as a special case of Corollary~\ref{Cor:hr} below (see also Example~\ref{Ex:hr}).
The condition (iii) follows from the latter assertion of Proposition~\ref{Prop:BL} and the hard Lefschetz theorem for $\IH^\bullet(\bP_{\rho^\vee}X, \kk)$. 
\end{proof}

\subsection{Hyperbolic localization}
We finish this section by recalling an equivariant version of the \emph{hyperbolic localization theorem} due to Braden \cite{Braden}.
We keep the assumption from the previous subsection.
Let $\tau \in X_*(T)$ be a cocharacter of $T$.
We have the associated decomposition 
\begin{equation} \label{eq:wtdecomp}
E =E^+_\tau \oplus E^0_\tau \oplus E^-_\tau,
\end{equation} 
where the component $E^{\pm}_\tau$ (resp.~$E^0_\tau$) is the sum of weight spaces $E_\lambda$ satisfying $\pm\langle\tau, \lambda\rangle >0$ (resp.~$\langle \tau, \lambda \rangle = 0$).
Let $i^\pm_\tau \colon E_\tau^\pm \oplus E_\tau^0 \hookrightarrow E$ and $i^\pm_{\tau,0} \colon E^0_\tau \hookrightarrow E^\pm_\tau \oplus E^0_\tau$ be the inclusions.

\begin{Thm}[\cite{Braden}]
\label{Thm:hr}
For any cocharacter $\tau \in X_*(T)$, the followings hold.
\begin{enumerate}
\item There is a natural isomorphism $(i_{\tau,0}^+)^*(i_\tau^+)^! \simeq (i_{\tau,0}^-)^! (i_\tau^-)^*$ of functors from $D^b_T(E, \kk)$ to $D^b_T(E_\tau^0,\kk)$. 
\item \label{Thm:hr2} For any simple perverse sheaf $\cF \in \Perv_T(E, \kk)$, its image $(i_{\tau,0}^+)^*(i_\tau^+)^!\cF$ is a finite direct sum of shifted simple perverse sheaves on $E^0_\tau$.  
\end{enumerate}
\end{Thm}
We call the functor $(i_{\tau,0}^+)^*(i_\tau^+)^!$ the \emph{hyperbolic localization} associated with $\tau$.
\begin{proof}
By the similar argument as in \cite[\S2.6]{FW14}, one can easily lift the main theorems in \cite{Braden} to the equivariant setting, which proves the assertions.   
\end{proof}

We say that a cocharacter $\tau \in X_*(T)$ is \emph{generic} if $E_\tau^0 = \{ 0 \}$.

\begin{Ex}
\label{Ex:hr}
For example, $\tau = \pm \rho^\vee$ is a generic cocharacter.
In this case, we have $E_{\pm \rho^\vee}^{\pm} =E$ and $E_{\pm \rho^\vee}^\mp = \{0\}$. 
Therefore, we have $i^+_{\rho^\vee,0} = i^+_{- \rho^\vee} = i$, $i_{\rho^\vee}^+ = \id_E$ and $i_{-\rho^\vee,0}^+ = \id_{\{0\}}$. 
Thus, the functors $i^* = (i_{\rho^\vee,0}^+)^*(i^+_{\rho^\vee})^!$ and $i^! = (i_{-\rho^\vee,0}^+)^*(i^+_{-\rho^\vee})^!$ are  special cases of hyperbolic localization.
\end{Ex}

\begin{Cor}
\label{Cor:hr}
Let $\tau \in X_*(T)$ be generic. 
For any simple perverse sheaf $\cF \in \Perv_T(E)$, we have an isomorphism of $\Z$-graded $\kk[z]$-modules
\[ \rH_{\Gm}^\bullet((i_{\tau,0}^+)^*(i_\tau^+)^!\Res_{\rho^\vee}(\cF)) \simeq \rH^\bullet((i_{\tau,0}^+)^*(i_\tau^+)^!\For(\cF))\otimes \kk[z]. \]
\end{Cor}
\begin{proof}
Since $\cF$ is a simple perverse sheaf, Theorem~\ref{Thm:hr}~\eqref{Thm:hr2} enables us to find an isomorphism
\[ (i_{\tau,0}^+)^*(i_\tau^+)^!\cF \simeq \bigoplus_{k \in \Z} \rH^k((i_{\tau,0}^+)^*(i_\tau^+)^!\For(\cF))\otimes \ul{\kk}_{\{0\}}[-k]\] 
in $D^b_{T}(\{0\},\kk)$.
Applying $\rH_{\Gm}^\bullet(-) \circ \Res_{\rho^\vee}$, we obtain the desired isomorphism. 
\end{proof}

\section{Proof of Theorem \ref{simply-laced}}
\label{ssec:Prqaa}
In this section, we give a proof of Theorem \ref{simply-laced} using the geometric construction of $U_q(L\fg)$-modules due to Nakajima \cite{Nak01, Naktns}.
We retain the notation from Section \ref{fqla} above.
Throughout this section, we assume that our Lie algebra $\fg$ is of simply-laced type.

\subsection{Geometric construction of mixed tensor products}
Fix $\bd =(d_{i,p}) \in \N^{\oplus \hI}$. 
It determines a dominant monomial $Y^{\bd} = \prod_{(i,p) \in \hI}Y_{i,p}^{d_{i,p}} \in \cM^+$. 
Let $\qv(\bd)$ and $\qv_0(\bd)$ be the graded quiver varieties, smooth and affine respectively, associated with a $\hI$-graded $\C$-vector space $D = \bigoplus_{(i,p) \in \hI}D_{i,p}$ such that $\dim_\C D_{i,p} = d_{i,p}$. See \cite[Section 4]{Nak04} (and also \cite[Section 4.4]{Fuj}) for the definition.
They come with natural actions of the group $G_\bd := \prod_{(i,p) \in \hI} GL(D_{i,p})$ and there is a canonical $G_\bd$-equivariant proper morphism of varieties $\pi_\bd \colon \qv(\bd) \to \qv_0(\bd)$.  
Let $\qvs(\bd) := \qv(\bd) \times_{\qv_0(\bd)} \qv(\bd)$ be the Steinberg type variety.
The equivariant algebraic $K$-theory $K^{G_\bd}(\qvs(\bd))$ is an associative algebra with respect to the convolution. 
By Nakajima~\cite{Nak01}, for any group homomorphism $\varphi \colon G \to G_\bd$, there is a $\kk$-algebra homomorphism 
\begin{equation}\label{Nakhom}
U_q(L\fg) \to \widehat{K}^{G}\left(\qvs(\bd)\right)_\kk, 
\end{equation}
where $\widehat{K}^G(-)_\kk$ denote the completion of the equivariant $K$-theory $K^{G}(-)\otimes_\Z \kk$ with respect to the ideal of the representation ring 
$R(G) = K^G(\pt)$ formed by virtual $G$-representations of dimension $0$.
For the completion, see also \cite[Section 4.6]{Fuj}.
By the equivariant Riemann-Roch theorem, we have a homomorphism of $\kk$-algebras
\[\wh{K}^G\left(\qvs(\bd)\right)_\kk \to \wh{\rH}^G_\bullet\left(\qvs(\bd), \kk\right),\]
where the RHS is the convolution algebra of the completed $G$-equivariant Borel-Moore homology.
It is an algebra over the completion $\wh{\rH}^\bullet_G(\mathrm{pt}, \kk)$. 
Composed with the homomorphism \eqref{Nakhom}, we get a $\kk$-algebra homomorphism
\begin{equation}\label{NakhomH}
\Psi_{\bd, \varphi} \colon U_q(L\fg) \to \wh{\rH}^G_\bullet\left(\qvs(\bd), \kk\right).
\end{equation}

We consider an action of $\Gm$ on the vector space $D$ such that the $(i,p)$-component $D_{i,p}$ is of weight $-e(i,p)$ for each $(i,p) \in \hI$. 
It defines a group homomorphism $\rho^\vee \colon \Gm \to G_\bd$. 
In what follows, we consider the case $G = \Gm$ and $\varphi = \rho^\vee$ in \eqref{NakhomH}.
We identify the ring $\rH_{\Gm}^\bullet(\pt, \kk)$ with the polynomial ring $\kk[z]$ so that the indeterminate $z$ corresponds to the \emph{negative} fundamental weight of $\Gm$.
In particular, we have the identification $\Oo = \kk [\![z]\!]= \wh{\rH}_{\Gm}^\bullet(\pt, \kk)$.

From now on, through the bijection $e \colon \hI \to J \subset \Z$, we identify $\bd$ with an element $\bd \in \N^{\oplus J}$.
Namely, we set $d_j = d_{i,p}$ if $j = e(i,p)$. 
In the similar way, we identify $D$ with a $J$-graded vector space. 
Consider the action of the symmetric group $\mathfrak{S}_{d}$ on the set $J^\bd$ by place permutations, where $d \seq \sum_{j \in J} d_j$. 
Let $\bep_c = (j_1, \ldots, j_{d})$ denote the unique  costandard  sequence in $J^\bd$.
We fix a basis $\{v_1, \ldots, v_{d} \}$ of $D$ such that $v_k \in D_{j_k}$. 
This yields a maximal torus $T_\bd$ of $G_\bd$ consisting of diagonal matrices with respect to the basis. 
Note that the homomorphism $\rho^\vee \colon \Gm \to G_\bd$ factors through $T_\bd$ and hence $\rho^\vee \in X_*(T_\bd)$. 
For each sequence $\bep \in J^\bd$, let $\sigma_\bep \in \mathfrak{S}_{d}$ denote the element of the smallest length such that 
$\bep  = (j_{\sigma_\bep(1)}, \ldots, j_{\sigma_\bep(d)}).$
Then we define $\tau_\bep \in X_*(T_\bd)$ by $\tau_\bep(t)\cdot v_{\sigma_\bep(k)} = t^{ -k } v_{\sigma_\bep(k)}$ for $1 \le k \le d$. 
Following Nakajima~\cite{Naktns}, we consider the closed subvariety $\tv(\bep)$ of $\qv(\bd)$ defined by
\[ \tv(\bep) := \{x\in\qv(\bd)\mid\lim_{s\to0}\tau_\bep(s)\cdot\pi_\bd(x)=0\in\qv_0(\bd)\}.\]
The equivariant Borel-Moore homology $\rH^{\Gm}_\bullet(\tv(\bep), \kk)$ is a (left) module over the algebra $\rH^{\Gm}_\bullet\left(\qvs(\bd), \kk\right)$ by the convolution, where the $\Gm$-action is given by $\rho^\vee$.
Through the Nakajima homomorphism $\Psi_{\bd, \rho^\vee}$, we regard the completed $\Gm$-equivariant Borel-Moore homology $\wh{\rH}^{\Gm}_\bullet(\tv(\bep), \kk)$ as a $U_q(L\fg)_\Oo$-module.

\begin{Thm}[\cite{Naktns}]\label{Naktns}
For any $\bep \in J^\bd$, we have an isomorphism of $U_q(L\fg)_\Oo$-modules
\[\wh{\rH}^{\Gm}_\bullet(\tv(\bep), \kk) \simeq \tM(\bep), \]
which specialize to 
\[\rH_\bullet(\tv(\bep), \kk) \simeq M(\bep).\]
\end{Thm}

\subsection{Sheaf theoretic interpretation}
Recall that the affine graded quiver variety $\qv_0(\bd)$ has a canonical stratification whose stratum $\qvr(\bv,\bd)$ is labelled by $\bv = (v_{i,p}) \in \N^{\oplus \hI'}$, where $\hI' \seq (I \times \Z)\setminus \hI$.
See \cite[Section 4]{Nak04}.
Write $A^{-\bv} := \prod_{(i,p) \in \hI'}A_{i,p}^{-v_{i,p}} \in \cM$. 
A stratum $\qvr(\bv,\bd)$ is non-empty if and only if $Y^\bd A^{-\bv} \in \cM^+$ and the simple module $L(Y^\bd A^{-\bv})$ contributes as a composition factor of the standard module $M(\bd) = M(Y^\bd)$. 
In particular, we have only finitely many non-empty strata, including $\qvr(0,\bd) = \{ 0 \}.$

Let $\cA_\bd \seq (\pi_{\bd})_*\ul{\kk}_{\qv(\bd)}$ denote the (derived) push-forward of the constant sheaf on $\qv(\bd)$ along the proper morphism $\pi_{\bd}$. 
Then, by an equivariant version of \cite[Section 8.6]{CG}, we have an isomorphism of $\kk$-algebras 
\begin{equation}\label{Ginzburg}
\rH_\bullet^\Gm(\qvs(\bd), \kk) \simeq \Hom^\bullet_\Gm(\cA_\bd, \cA_\bd), 
\end{equation}
where the algebra structure on the RHS is given by the Yoneda product.
By \cite[Theorem 14.3.2]{Nak01}, we have a decomposition in $D^b_{\Gm}(\qv_0(\bd),\kk)$:
\begin{equation}\label{DT} 
\cA_\bd \simeq \bigoplus_{\bv} \mathrm{IC}(\bv,\bd) \otimes_\kk L^\bullet(\bv,\bd),
\end{equation}
where $\bv$ runs over all the elements of $\N^{\oplus \hI'}$ satisfying $\qvr(\bv,\bd) \neq \varnothing$, $\mathrm{IC}(\bv,\bd)$ denotes the intersection cohomology complex of the closure of the stratum $\qvr(\bv, \bd)$ (with coefficients in $\kk$), and $L^\bullet(\bv, \bd) \in D^{b}(\pt, \kk)$ is a non-zero object, which we regard as a non-zero finite-dimensional $\Z$-graded $\kk$-vector space. 
We consider the total perverse cohomology
\[ \ocA_\bd := \bigoplus_{k \in \Z} {}^p\mathcal{H}^k(\cA_\bd) = \bigoplus_{\bv} \mathrm{IC}(\bv,\bd) \otimes_\kk L(\bv, \bd),\]
where $L(\bv,\bd)$ denotes the underlying ungraded $\kk$-vector space of $L^\bullet(\bv, \bd)$. 
Since $\ocA_\bd$ is a semisimple perverse sheaf, its Yoneda algebra 
\[ A_\bd \seq \Hom^\bullet_\Gm(\ocA_\bd,\ocA_\bd)\]
is a non-negatively graded $\kk$-algebra whose degree zero component $A^0$ is isomorphic to the semisimple algebra $\bigoplus_{\bv} \End_{\kk}(L(\bv, \bd))$.
Let $\wh{A}_\bd \seq \prod_{n \ge 0}A_\bd^n$ denote the completion of $A_\bd$ along the grading. 
The set $\{L(\bv, \bd)\}_\bv$ gives a complete system of simple $\wh{A}_\bd$-modules.

From \eqref{Ginzburg} and \eqref{DT}, we obtain an isomorphism of $\kk$-algebras
\begin{equation}\label{isomHA}
\wh{\rH}_\bullet^\Gm(\qvs(\bd), \kk) \simeq \wh{A}_\bd.
\end{equation}
Composed with the Nakajima homomorphism~\eqref{Nakhom}, we get a $\kk$-algebra homomorphism $U_q(L\fg) \to \wh{A}_\bd$, through which we regard an $\wh{A}_\bd$-module as a $U_q(L\fg)$-module.

\begin{Thm}[{\cite[Theorem 14.3.2]{Nak01}}] \label{NakS14}
The simple $\wh{A}_\bd$-module $L(\bv, \bd)$ is isomorphic to the simple $U_q(L\fg)$-module $L(Y^\bd A^{-\bv})$.
\end{Thm}

Let $\Gamma$ be an infinite quiver whose set of vertices is $\hI$ and whose set of arrows $\Gamma_1$ is given by the following rule: the number of arrows
from $(i,p)$ to $(j,r)$ is equal to the pole order $\fo(V_{j,r}, V_{i,p})$ of the normalized $R$-matrix (cf.\ Section \ref{Rmat}).   
We define
\[ E(\bd) \seq \bigoplus_{x \in \Gamma_1} \Hom_\C(D_{\mathrm{s}(x)}, D_{\mathrm{t}(x)}), \]
where $\mathrm{s}(x)$ (resp.~$\mathrm{t}(x)$) denotes the source (resp.~target) of an arrow $x$.
The group $G_\bd$ acts on $E(\bd)$ by conjugation.
Note that $\Gm$ acts on $E(\bd)$ through $\rho^\vee \colon \Gm \to T_\bd$ with strictly positive weights because $\fo(V_{j,r}, V_{i,p})>0$ implies $r<p$. 
In other words, $\rho^\vee \in X_*(T_\bd)$ satisfies the condition~\eqref{eq:attractive}, and hence the $T_\bd$-action on $E(\bd)$ is attractive.

\begin{Thm}[\cite{KS, Fuj}]\label{KS}
The affine graded quiver variety $\qv_0(\bd)$ is $G_\bd$-equivariantly isomorphic to a $G_\bd$-stable closed subvariety of the affine space $E(\bd)$. 
\end{Thm}
In what follows, we identify $\qv_0(\bd)$ with the $G_\bd$-stable closed subvariety of $E(\bd)$.    

Recall the cocharacter $\tau_\bep \in X_*(T)$ for each $\bep \in J^\bd$.
Let 
\[E(\bep) \seq \{ x \in E(\bd) \mid\lim_{s\to0} \tau_\bep(s)\cdot x=0\}= E(\bd)_{\tau_\bep}^+\]
in the notation of \eqref{eq:wtdecomp}.
By definition, the variety $\tv(\bep)$ is identical to the fiber product $\qv(\bd) \times_{E(\bd)} E(\bep)$ arising from the canonical proper morphism $\pi_\bd \colon \qv(\bd) \to \qv_0(\bd) \subset E(\bd)$ and the inclusion $i_{\bep} \colon E(\bep) \hookrightarrow E(\bd)$.
Therefore, similarly to the isomorphism  in \eqref{Ginzburg}, we have
\[ \rH^\Gm_\bullet(\tv(\bep), \kk) = \rH^\Gm_\bullet(\qv(\bd) \times_{E(\bd)} E(\bep)) \simeq \Hom_\Gm^\bullet(i_{\bep*}\ul{\kk}_{E(\bep)}, \cA_\bd) \simeq \rH^\bullet_\Gm(p_{\bep*} i_\bep^! \cA_\bd)\]
as $A_\bd$-modules, where $p_\bep \colon E(\bep) \to \{0\}$ is the obvious morphism. 
Here the $\Z$-gradings are disregarded.
Let $i_{\bep, 0} \colon \{0\} \hookrightarrow E(\bep)$ be the inclusion.
By Proposition~\ref{Prop:attr}, we have an isomorphism $p_{\bep*} \simeq i_{\bep,0}^*$ of functors from $D^b_
\Gm(E(\bep), \kk)$ to $D^b_\Gm(\{0\}, \kk)$.
After the completion, we get 
\[\wh{\rH}^\Gm_\bullet(\tv(\bep),  \kk ) \simeq \wh{\rH}^\bullet_\Gm(p_{\bep*} i_\bep^! \ocA_\bd) \simeq \wh{\rH}^\bullet_\Gm(i_{\bep,0}^* i_\bep^! \ocA_\bd)\] 
as $\wh{A}_\bd$-modules. 
Combined with Theorem~\ref{Naktns}, we obtain the following.

\begin{Prop}\label{Nsheaf}
For each $\bep \in J^\bd$, we have an isomorphism of $U_q(L\fg)_\Oo$-modules
\[\wh{\rH}^\bullet_\Gm(i_{\bep,0}^* i_\bep^! \ocA_\bd) \simeq \tM(\bep), \]
which specialize to 
\[ \rH^\bullet(i_{\bep,0}^* i_\bep^! \ocA_\bd) \simeq M(\bep). \]
\end{Prop}  

\begin{Rem} \label{Rem:duality_g}
By construction, we have $E(\bd)^{+}_{\tau_{\bep}} = E(\bd)^{-}_{\tau_{\bep^\op}}$, where $\bep^\op$ is the opposite sequence of $\bep$.
By Theorem \ref{Thm:hr}, it implies an isomorphism $\bD(i_{\bep,0}^* i_\bep^! \ocA_\bd) \simeq i_{\bep^\op,0}^* i_{\bep^\op}^! \ocA_\bd$.
\end{Rem}

\subsection{Geometric interpretation of $R$-matrices}
Recall the preorder $\lesssim$ of the set $J^\bd$ from Section \ref{mjf}.
The following lemma is clear from the definition.

\begin{Lem}
For $\bep, \bep' \in J^\bd$, we have $E(\bep) \subset E(\bep')$ if and only if $\bep \lesssim \bep'$. 
For the standard (resp.~costandard) sequence $\bep_s$ (resp.~$\bep_c$), we have $E(\bep_s) = \{0\}$ (resp.~$E(\bep_c) = E(\bd)$).
\end{Lem}

For $\bep, \bep' \in J^\bd$ satisfying $\bep \lesssim \bep'$, let $i_{\bep', \bep} \colon E(\bep) \hookrightarrow E(\bep')$ denote the inclusion.
Note that we have $i_{\bep, 0} = i_{\bep, \bep_s}$ and $i_\bep = i_{\bep_c, \bep}$ for any $\bep \in J^\bd$.
We have the following diagram of inclusions:
$$
\xymatrix{
& E(\bep_c) = E(\bd) & \\
E(\bep)
\ar@{->}[ur]^-{i_\bep} \ar@{->}[rr]^-{i_{\bep', \bep}}
&& E(\bep'). \ar@{->}[ul]_-{i_{\bep'}} \\
& E(\bep_s) = \{0\}
\ar@{->}[ur]_-{i_{\bep',0}} \ar@{->}[ul]^-{i_{\bep,0}}& 
}
$$
The canonical morphism of functors $i_{\bep', \bep}^! \to i_{\bep', \bep}^*$ induces a morphism
\begin{equation} \label{eq:ibep}
i_{\bep,0}^* i_\bep^! \ocA_\bd = i_{\bep,0}^* i_{\bep', \bep}^! i_{\bep'}^! \ocA_\bd \to i_{\bep,0}^* i_{\bep', \bep}^* i_{\bep'}^! \ocA_\bd = i_{\bep',0}^* i_{\bep'}^! \ocA_\bd.
\end{equation}
Taking the cohomology, we obtain a homomorphism 
\[ \cR_{\bep', \bep} \colon \rH^\bullet_\Gm(i_{\bep,0}^* i_\bep^! \ocA_\bd) \to \rH^\bullet_\Gm(i_{\bep',0}^* i_{\bep'}^! \ocA_\bd) \]
of graded $A_\bd$-modules. From the construction, it satisfies 
\begin{enumerate}
\item $\cR_{\bep'', \bep'} \circ \cR_{\bep', \bep} = \cR_{\bep'', \bep}$ if $\bep \lesssim \bep' \lesssim \bep''$;
\item $\cR_{\bep, \bep'} \circ \cR_{\bep', \bep} = \id$ if $\bep \sim \bep'$. 
\end{enumerate}

\begin{Prop}\label{tau}
Let $\bep, \bep' \in J^\bd$ satisfying $\bep \lesssim \bep'$.
Through the isomorphism in Proposition~\ref{Nsheaf}, the completion $\wh{\cR}_{\bep', \bep}$ of the homomorphism $\cR_{\bep', \bep}$ is identified with the intertwiner $R_{\bep',\bep}$, i.e., the following diagram commutes up to multiples in $\Oo^\times$: 
\begin{equation} \label{eq:Rgeom}
\vcenter{
\xymatrix{
\wh{\rH}^\bullet_\Gm(i_{\bep,0}^* i_\bep^! \ocA_\bd) \ar@{->}[d]_-{\simeq} \ar@{->}[r]^-{\wh{\cR}_{\bep', \bep}}
& \wh{\rH}^\bullet_\Gm(i_{\bep',0}^* i_{\bep'}^! \ocA_\bd)  \ar@{->}[d]^-{\simeq}
\\
\tM(\bep)  \ar@{->}[r]^-{R_{\bep', \bep}}
& \tM(\bep'),
}}
\end{equation}
where the vertical arrows are the isomorphisms in Proposition~\ref{Nsheaf}.
\end{Prop}
\begin{proof}
Since the $U_q(L\fg)_\Kk$-module $\tM(\bep)_\Kk \simeq \tM(\bep')_\Kk$ is simple and we have $\beta(\bep', \bep) = 0$ by the normality, the intertwiner $R_{\bep', \bep}$ is characterized as a unique (up to multiples in $\Oo^\times$) $U_q(L\fg)_\Oo$-homomorphism $\tM(\bep) \to \tM(\bep')$ whose specialization at $z=0$ is non-zero. 
By construction, $\wh{\cR}_{\bep',\bep}$ is a $U_q(L\fg)_\Oo$-homomorphism.
Thus, it suffices to show that the specialization at $z=0$ of $\cR_{\bep',\bep}$ is non-zero. 
By Corollary~\ref{Cor:hr}, the specialization of $\cR_{\bep', \bep}$ is the homomorphism $\rH^\bullet(i_{\bep,0}^* i_\bep^! \ocA_\bd) \to \rH^\bullet(i_{\bep',0}^* i_{\bep'}^! \ocA_\bd)$ arising from \eqref{eq:ibep}.  
Since $\ocA_\bd$ contains $\ul{\kk}_{\{0\}}$ as a summand, it is non-zero.
\end{proof}

\begin{Cor} 
\label{Cor:Rinj}
For any $\bep, \bep' \in J^\bd$ satisfying $\bep \lesssim \bep'$, the homomorphism $\cR_{\bep', \bep}$ is injective.
\end{Cor}
\begin{proof}
Since $\rH_\Gm^\bullet(i_{\bep,0}^* i_\bep^! \mathcal{A})$ is a free $\kk[z]$-module of finite rank by Corollary~\ref{Cor:hr}, it suffices to show that the completion $\wh{\cR}_{\bep', \bep}$ is injective.
This latter assertion follows from Proposition~\ref{tau} and  the injectivity of the intertwiner $R_{\bep', \bep}$.
\end{proof}

\subsection{Proof of Theorem \ref{simply-laced}}  
\label{Ssec:Jantzeng}
Now we shall prove the following crucial result using the facts from Section \ref{sec:preliminary}.

\begin{Thm} \label{Amain}
For any $\bd \in \N^{\oplus \hI}$, $\bep \in J^\bd$ and $n \in \Z$, the isomorphism $M(\bep) \simeq \rH^\bullet(i_{\bep,0}^* i_\bep^! \ocA_\bd)$ in Proposition~\ref{Nsheaf} induces the isomorphism
\[ F_nM(\bep) \simeq \rH^{\ge n}(i_{\bep,0}^* i_\bep^! \ocA_\bd).\]
In particular, we have the following equality in $K(\Cc_{\Z})_t$:
\begin{equation} \label{KL}
[M(\bep)]_t = \sum_{\bv} \left(\sum_{n \in \Z} t^n\dim_\kk \rH^n(i_{\bep,0}^* i_\bep^! \mathrm{IC}(\bv, \bd))\right) [L(Y^\bd A^{-\bv})]. 
\end{equation}
\end{Thm}
\begin{proof}
In this proof, to lighten the notation, for each $\bep \in J^\bd$, we set 
\[\rH(\bep) :=  \rH^\bullet_\Gm(i_{\bep,0}^* i_\bep^! \ocA_\bd) \]
and regard it as a graded $A_\bd$-submodule of $\rH(\bep_c)$ through the injective homomorphism $\cR_{\bep_c, \bep} \colon \rH(\bep) \to \rH(\bep_c)$ (cf.\ Corollary \ref{Cor:Rinj}).
Thus, we have the inclusions of graded $A_\bd$-submodules $\rH(\bep_s) \subset \rH(\bep) \subset \rH(\bep_c)$ for any $\bep \in J^\bd$.
By Corollary~\ref{Cor:hr}, $\rH(\bep)$ is a graded free $\kk[z]$-module of finite rank and the quotient $\rH(\bep) / z \rH(\bep)$ is identical to the non-equivariant cohomology $\rH^\bullet(i_{\bep,0}^* i_\bep^! \ocA_\bd)$ as a graded $\kk$-vector space.
By Theorem~\ref{Thm:Fex} and Lemma~\ref{Lem:Lef}, the graded $\kk[z]$-module $L \seq \rH(\bep_c)/z \rH(\bep_s)$ satisfies the hard Lefschetz property.

By the definition \eqref{eq:Filt}, the filter submodule $F_nM(\bep)$ is the image of 
\[ F_n \tM(\bep) \seq \tM(\bep) \cap \sum_{k \in \Z} \left( z^k R_{\bep, \bep_s}\tM(\bep_s) \cap z^{n-k}R^{-1}_{\bep_c, \bep} \tM(\bep_c)\right) \]
under the evaluation map $\ev_{z=0} \colon \tM(\bep) \to M(\bep) = \tM(\bep)/z \tM(\bep)$.
Consider a quotient map 
\[ f \colon \tM(\bep) \to \tM(\bep)/zR_{\bep, \bep_s}\tM(\bep_s) \simeq \wh{\rH}(\bep)/z\wh{\rH}(\bep_s) \simeq \rH(\bep)/z\rH(\bep_s) \subset \rH(\bep_c)/z\rH(\bep_s) = L,\]
where the first isomorphism is induced by the one in Proposition~\ref{Nsheaf}.
We have 
\begin{align} 
f (F_n \tM(\bep)) &= \left(\sum_{k  - l = n} \Ker(z_L^{l+1}) \cap \Image (z_L^k) \right)\cap (\rH(\bep)/z\rH(\bep_s)) \\ 
&= L^{\ge n} \cap (\rH(\bep)/z\rH(\bep_s)) \\
&= \left(\rH(\bep)/z\rH(\bep_s)\right)^{\ge n}, 
\end{align} 
where the second equality is due to Lemma~\ref{Lem:hL}.  
Letting $g \colon L = \rH(\bep_c)/z\rH(\bep_s) \to \rH(\bep_c)/z\rH(\bep)$ be the quotient map, we obtain
\[ F_nM(\bep) = \ev_{z=0}(F_n\tM(\bep)) \simeq g\left(f(F_n \tM(\bep))\right) = g\left(\left(\rH(\bep)/z\rH(\bep_s)\right)^{\ge n}\right) = \rH^\bullet(i_{\bep,0}^* i_\bep^! \ocA_\bd)^{\ge n},\]
which proves the former assertion.
The other assertion \eqref{KL} follows from the former one together with the definition of $\ocA_\bd$ and Theorem~\ref{NakS14}.
\end{proof}

\begin{Cor} \label{Cor:ss}
For any $\bep$ and $n \in \Z$, the filtration layer $\Gr_n^FM(\bep) = F_{n}M(\bep)/F_{n+1}M(\bep)$ is a semisimple $U_q(L\fg)$-module.
\end{Cor}
\begin{proof}
This is because the $U_q(L\fg)$-action on $\Gr_n^FM(\bep)$ factors through the action of the semisimple algebra $A_{ \bd }^0  = \Hom^0_\Gm(\ocA_\bd,\ocA_\bd) $ by the above Theorem~\ref{Amain}.
\end{proof}

On the other hand, the geometric construction of the quantum Grothendieck ring $K_t(\Cc_\Z)$ due to Varagnolo-Vasserot \cite{VV} implies the following. 

\begin{Thm}[{\cite{VV03}}] \label{Thm:VVqgr}
For any $\bd \in \N^{\oplus \hI}$ and $\bep \in J^\bd$, we have the following equality in $K_t(\Cc_\Z)$:
\begin{equation} \label{VV}
E_t(\bep) = \sum_{\bv} \left(\sum_{n \in \Z} t^n \dim_\kk \rH^n(i_{\bep,0}^* i_\bep^! \mathrm{IC}(\bv, \bd))\right) L_t(Y^\bd A^{-\bv}). 
\end{equation}
\end{Thm}
\begin{proof}
This is a direct consequence of the geometric definition of $K_t(\Cc_\Z)$ in \cite{VV03}.
See \cite[Section 5.5]{HL15} for a comparison with our algebraic definition in Section \ref{Ssec:qGr}. 
\end{proof}

Comparing \eqref{KL} with \eqref{VV}, we arrive at the desired equality \eqref{mixedet}. Thus, we have proved Theorem \ref{simply-laced}.

\subsection{Examples}
We give some explicit computations  of the hyperbolic localizations $i_{\bep,0}^*i_\bep^! \ocA_\bd$ in the simplest examples (and we check that we recover examples that we computed in the previous sections). 
Let $\fg = \mathfrak{sl}_2$. 
In this case, $\hI = \{ 1 \} \times 2\Z$, and $J=e(\hI)$ is the set of all odd integers. 
The quiver $\Gamma$ is depicted as 
 $$\raisebox{5mm}{
\scalebox{0.9}{
\xymatrix@!C=10mm@R=0mm{
&{(1,-4)}&(1,-2)&(1,0)&(1,2)&(1,4)& \\
\cdots \ar@{<-}[r] & \circ \ar@{<-}[r]  & \circ  \ar@{<-}[r] & \circ \ar@{<-}[r]  & \circ \ar@{<-}[r] & \circ \ar@{<-}[r] & \cdots
}}}.
$$ 
We have $S_j = L(Y_{1,j-1})$ for any $j \in J$.

\subsubsection{}
We consider the case of Section \ref{exfil2_1}, where $\bd \in \N^{\oplus\hI}$ is given by $d_{1,2} = 2$, $d_{1,0} = 1$ and $d_{1,2k} = 0$ if $k \not \in \{ 0,1\}$.
The affine graded quiver variety $\qv_0(\bd)$ coincides with the $2$-dimensional linear space $E \seq E(\bd) = \Hom_\C(\C^2, \C)$ and the stratification is given by $E = (E \setminus \{ 0 \}) \sqcup \{0\}$.
We have 
\[ \ocA_\bd \simeq (\ul{\kk}_{\{0\}} \otimes L) \oplus (\ul{\kk}_{E}[2] \otimes S_3), \]
where $L=L(Y_{1,2}^2 Y_{1,0})$.
The set $J^\bd$ consists of $3$ elements $\bep_s = (3,3,1) < \bep = (3,1,3) < \bep_c = (1,3,3)$
and $E(\bep)$ is a $1$-dimensional linear subspace of $E$.
We have 
\begin{align*}
i_{\bep_s,0}^* i_{\bep_s}^! \ocA_\bd =  i_{0}^! \ocA_\bd &= (\ul{\kk}_{\{0\}} \otimes L) \oplus (\ul{\kk}_{\{0\}} [-2] \otimes S_3), \\
i_{\bep,0}^* i_{\bep}^! \ocA_\bd &= (\ul{\kk}_{\{0\}} \otimes L) \oplus (\ul{\kk}_{\{0\}} \otimes S_3), \\
i_{\bep_c,0}^* i_{\bep_c}^! \ocA_\bd =  i_{0}^* \ocA_\bd &= (\ul{\kk}_{\{0\}} \otimes L) \oplus (\ul{\kk}_{\{0\}} [2] \otimes S_3),
\end{align*}
where (and hereafter) $i_0$ denotes the inclusion of the origin. 

\subsubsection{}
We consider the case of Section \ref{exfil2_2}, where $\bd \in \N^{\oplus\hI}$ is given by $d_{1,0} = d_{1,2} = d_{1,4} = 1$ and $d_{1,2k} = 0$ if $k \not \in \{ 0,1, 2\}$.
The affine graded quiver variety $\qv_0(\bd)$ coincides with the closed subvariety $X$ of the $2$-dimensional linear space $E(\bd) = \Hom_\C(\C, \C)^{\oplus 2} = \{ (a,b) \mid a,b \in \C \}$ defined by the equation $ab=0$. 
The stratification is given by 
$\qv_0(\bd) = X = X_a \sqcup X_b \sqcup \{0\}$, where $X_a \seq \{ (a,b) \in \C^2 \mid a \neq 0, b=0 \}$ and $X_b \seq \{ (a,b) \in \C^2 \mid a = 0, b\neq 0 \}$.
We have 
\[ \ocA_\bd \simeq (\ul{\kk}_{\{0\}} \otimes L) \oplus (\ul{\kk}_{\bar{X}_a}[1] \otimes S_5) \oplus (\ul{\kk}_{\bar{X}_b}[1] \otimes S_1), \]
where $L=L(Y_{1,0}Y_{1,2}Y_{1,4})$.
The set $J^\bd / \sim$ consists of $4$ equivalence classes represented by $\bep_s = (5,3,1), \bep_1 = (3,5,1), \bep_2 = (5,1,3)$, and $\bep_c = (1,3,5)$. 
We have $E(\bep_1) = \bar{X}_{b}$ and $E(\bep_2) = \bar{X}_{a}$.
Therefore, we can compute as:
\begin{align*}
i_{\bep_s,0}^* i_{\bep_s}^! \ocA_\bd =  i_{0}^! \ocA_\bd &= (\ul{\kk}_{\{0\}} \otimes L) \oplus (\ul{\kk}_{\{0\}} [-1] \otimes S_5) \oplus (\ul{\kk}_{\{0\}}[-1] \otimes S_1), \\
i_{\bep_1,0}^* i_{\bep_1}^! \ocA_\bd &= (\ul{\kk}_{\{0\}} \otimes L) \oplus (\ul{\kk}_{\{0\}} [-1] \otimes S_5) \oplus (\ul{\kk}_{\{0\}}[1] \otimes S_1),  \\
i_{\bep_2,0}^* i_{\bep_2}^! \ocA_\bd &= (\ul{\kk}_{\{0\}} \otimes L) \oplus (\ul{\kk}_{\{0\}} [1] \otimes S_5) \oplus (\ul{\kk}_{\{0\}}[-1] \otimes S_1),  \\
i_{\bep_c,0}^* i_{\bep_c}^! \ocA_\bd =  i_{0}^* \ocA_\bd &= (\ul{\kk}_{\{0\}} \otimes L) \oplus (\ul{\kk}_{\{0\}} [1] \otimes S_5) \oplus (\ul{\kk}_{\{0\}}[1] \otimes S_1).
\end{align*}

\subsubsection{}
We consider the case of Section \ref{exfil2_3}, where $\bd \in \N^{\oplus\hI}$ is given by $d_{1,0} = d_{1,2} = 2$ and $d_{1,2k} = 0$ if $k \not \in \{ 0,1\}$.
The affine graded quiver variety $\qv_0(\bd)$ coincides with the $4$-dimensional linear space 
$E \seq E(\bd) = \Hom_\C(\C^2, \C^2) \simeq \mathrm{Mat}_{2}(\C)$
 and the stratification is given by $E = (E \setminus X) \sqcup (X \setminus \{ 0 \}) \sqcup \{0\}$, where $X \seq \{ \begin{pmatrix} a & b \\ c & d \end{pmatrix} \in \mathrm{Mat}_{2}(\C) \mid ad = bc \}$.
We have 
\[ \ocA_\bd \simeq (\ul{\kk}_{\{0\}} \otimes L) \oplus (\IC(X, \kk) \otimes K) \oplus (\ul{\kk}_E[4] \otimes \mathbf{1}), \]
where $L=L(Y_{1,0}^2 Y_{1,2}^2)$, $K = L(Y_{1,0} Y_{1,2})$, and $\mathbf{1} = L(1)$.
Note that $\IC(X, \kk)$ fits into an exact triangle  (cf.\ \cite[Exercise 3.10.6]{AcharBook}): $\ul{\kk}_X[3] \to \IC(X, \kk) \to \ul{\kk}_{\{0\}}[1] \xrightarrow{+1}$.
The set $J^\bd$ consists of $6$ elements $\bep_s = (3,3,1,1)$, $\bep_1 = (3,1,3,1)$, $\bep_2 = (3,1,1,3)$, $\bep_3 = (1,3,3,1)$, $\bep_4 = (1,3,1,3)$, and $\bep_c = (1,1,3,3)$. 
We have the following commutative diagram of inclusions:
\[
\xymatrix@!C=30mm{ &E(\bep_c) = E& \\
& E(\bep_4) = \{ c = 0\}\ar[u]& \\
E(\bep_2) = \{ a= c = 0 \} \ar[ur] && E(\bep_3)= \{c = d = 0 \}. \ar[ul] \\
& E(\bep_1) = \{ a = c =d = 0 \}\ar[ur] \ar[ul]  &\\ 
& E(\bep_s) = \{ 0 \} \ar[u]& \\}
\]
By the fact that $E(\bep_k) \subset X$ for $k \in \{1,2,3 \}$ and Remark \ref{Rem:duality_g}, we can compute as:
\begin{align*}
i_{\bep_s,0}^* i_{\bep_s}^! \ocA_\bd =  i_{0}^! \ocA_\bd &= (\ul{\kk}_{\{0\}} \otimes L) \oplus (\ul{\kk}_{\{0\}} [-1] \oplus \ul{\kk}_{\{0\}}[-3]) \otimes K \oplus ( \ul{\kk}_{\{0\}} [-4]  \otimes \mathbf{1}), \\
i_{\bep_1,0}^* i_{\bep_1}^! \ocA_\bd &= (\ul{\kk}_{\{0\}} \otimes L) \oplus (\ul{\kk}_{\{0\}} [-1] \oplus \ul{\kk}_{\{0\}}[-1]) \otimes K \oplus ( \ul{\kk}_{\{0\}} [-2]  \otimes \mathbf{1}), \\
i_{\bep_2,0}^* i_{\bep_2}^! \ocA_\bd &= (\ul{\kk}_{\{0\}} \otimes L) \oplus (\ul{\kk}_{\{0\}} [-1] \oplus \ul{\kk}_{\{0\}}[1]) \otimes K \oplus ( \ul{\kk}_{\{0\}}[0]  \otimes \mathbf{1}), \\
i_{\bep_3,0}^* i_{\bep_3}^! \ocA_\bd &= (\ul{\kk}_{\{0\}} \otimes L) \oplus (\ul{\kk}_{\{0\}} [-1] \oplus \ul{\kk}_{\{0\}}[1]) \otimes K \oplus ( \ul{\kk}_{\{0\}} [0]  \otimes \mathbf{1}), \\
i_{\bep_4,0}^* i_{\bep_4}^! \ocA_\bd &= (\ul{\kk}_{\{0\}} \otimes L) \oplus (\ul{\kk}_{\{0\}} [1] \oplus \ul{\kk}_{\{0\}}[1]) \otimes K \oplus ( \ul{\kk}_{\{0\}} [2]  \otimes \mathbf{1}), \\
i_{\bep_c,0}^* i_{\bep_c}^! \ocA_\bd = i_{0}^* \ocA_\bd &= (\ul{\kk}_{\{0\}} \otimes L) \oplus (\ul{\kk}_{\{0\}} [1] \oplus \ul{\kk}_{\{0\}}[3]) \otimes K \oplus ( \ul{\kk}_{\{0\}} [4]  \otimes \mathbf{1}).
\end{align*}

\section{Proof of Theorem \ref{Thm:adapted}}
\label{ssec:PrqH}

In this section, we give a proof of Theorem \ref{Thm:adapted} using the geometric construction of the canonical bases due to Lusztig \cite{LusB} and the symmetric quiver Hecke algebras due to Varagnolo-Vasserot \cite{VV11}. 
We retain the notation from Section \ref{mjqha} above.
Throughout this section, we assume that $\ii = (i_1, \ldots, i_\ell) \in I^\ell$ is a reduced word for an element $w \in \sW$ adapted to a fixed quiver $Q$ of type $\fg$ (recall Definition \ref{Def:adapted}). 
We set $J \seq \{ k \in \Z \mid 1 \le k \le \ell\}$ as before.

\subsection{Lusztig's construction of canonical bases}
\label{ssec:Lusztig_construction}
First, we review the geometric construction of the canonical basis $\bB$ of $U_t^+(\fg)_{\Z[t^{\pm 1}]}$  due to Lusztig~\cite{LusB}.
For an $I$-graded $\C$-vector space $V = \bigoplus_{i \in I} V_i$, we set $\vdim V \seq \sum_{i \in I} (\dim_\C V_i) \alpha_i \in \sQ^+$.  
For each $\beta \in \sQ^+$, we fix an $I$-graded vector space $V^\beta= \bigoplus_{i \in I} V^\beta_{i}$ satisfying $\vdim V^\beta = \beta$.  
For $\beta, \beta' \in \sQ^+$, we set
\[
L(\beta, \beta') \seq \bigoplus_{i \in I} \Hom_\C(V^\beta_{i}, V^{\beta'}_{i}), \quad 
E(\beta, \beta') \seq \bigoplus_{h \in Q_1} \Hom_\C(V^\beta_{\mathrm{s}(h)}, V^{\beta'}_{\mathrm{t}(h)}).\]
The space $X(\beta) \seq E(\beta, \beta)$ consists of representations of the quiver $Q$ of dimension vector $\beta$, on which we have the natural conjugation action of the group 
$G(\beta) \seq \prod_{i \in I} GL(V^\beta_i).$
Note that the Lie algebra of $G(\beta)$ is identical to $L(\beta, \beta)$.
It is convenient to introduce the bilinear form $\langle-,-\rangle_Q \colon \sQ^+ \times \sQ^+ \to \Z$ defined by
\[
\langle \beta, \beta' \rangle_Q \seq \dim_\C L(\beta, \beta') - \dim_\C E(\beta, \beta').
\]
For any $\beta, \beta' \in \sQ$, we have 
\begin{equation} \label{eq:inner}
(\beta, \beta') = \langle \beta, \beta' \rangle_Q + \langle \beta', \beta \rangle_Q.
\end{equation}

For any finite sequence $\beta_1, \ldots, \beta_n \in \sQ^+$, we write $X(\beta_1, \ldots, \beta_n) \seq X(\beta_1) \times \cdots \times X(\beta_n)$ and $G(\beta_1, \ldots, \beta_n) \seq G(\beta_1) \times \cdots \times G(\beta_n)$.
Assume $\beta = \sum_{k=1}^n \beta_k$.
By \cite[Section 9]{LusB}, we have an adjoint pair of functors
\[ 
\xymatrix@C=20mm{ D^b_{ G(\beta_1, \ldots, \beta_n)}(X(\beta_1,\ldots,\beta_n), \kk) \ar@<3pt>[r]^-{\Ind_{\beta_1, \ldots, \beta_n}} & D_{G(\beta)}^b(X(\beta), \kk)  \ar@<3pt>@{->}[l]^-{\Res_{\beta_1, \ldots, \beta_n}}}, \]
with $\Ind_{\beta_1, \ldots, \beta_n}$ being left adjoint to $\Res_{\beta_1, \ldots, \beta_n}$.
We shall recall their construction.
Choose an identification $V^\beta = V^{\beta_1} \oplus \cdots \oplus V^{\beta_n}$ of $I$-graded vector spaces. (The resulting functors do not depend on this choice up to isomorphism.)
We define an $I$-graded flag $V^\beta = F_0 \supset F_1 \supset \cdots \supset F_n = \{0\}$ given by $F_k = \bigoplus_{l > k}V^{\beta_l}$.
Let $F(\beta_1, \ldots, \beta_n)$ be the subvariety of $X(\beta)$ consisting of representations $x$ satisfying $x(F_k) \subset F_k$ for $1 \le k \le n$.
Consider the following diagram 
\begin{equation} \label{eq:resdiag}
\xymatrix{
X(\beta_1, \ldots, \beta_n)& \ar@{->}[l]_-{\kappa} F(\beta_1, \ldots, \beta_n) \ar@{->}[r]^-{\iota} & X(\beta)
},
\end{equation}
where $\iota$ is the inclusion, and $\kappa$ is given by $x \mapsto (x |_{F_{k-1}} \mod F_k)_{1 \le k \le n}$.
On the other hand, let $P \subset G(\beta)$ denote the stabilizer of the fixed flag $F_\bullet$, and $U$ its unipotent radical so that $P/U \simeq G(\beta_1, \ldots, \beta_n)$. We consider the following diagram
\[ 
\xymatrix{
X(\beta_1, \ldots, \beta_n)& \ar[l]_-{p_1} G(\beta) \times^{U} F(\beta_1, \ldots, \beta_n) \ar[r]^-{p_2} & \tF(\beta_1, \ldots, \beta_n) \ar[r]^-{p_3} & X(\beta)
},
\] 
where $\tF(\beta_1, \ldots, \beta_n) = G(\beta) \times^{P} F(\beta_1, \ldots, \beta_n)$, $p_1(g,x) = \kappa(x)$, $p_2(g,x) = (g,x)$, $p_3(g,x) = g\cdot\iota(x)$ for $g \in G(\beta)$ and $x \in F(\beta_1, \ldots, \beta_n)$.
Note that $p_1$ is smooth, $p_2$ is a $G(\beta_1, \ldots, \beta_n)$-torsor, $p_3$ is proper. 
Then, we define
\[
\Ind_{\beta_1, \ldots, \beta_n}\cF \seq p_{3*}\cF' [ 2 \dim U -c ] \quad \text{ and } \quad 
\Res_{\beta_1, \ldots, \beta_n}\cG \seq \kappa_* \iota^! \cG [c],
\]
where $\cF'$ is a unique $G(\beta)$-equivariant complex on $G(\beta) \times^{U} F(\beta_1, \ldots, \beta_n)$ satisfying $p_2^*\cF' \simeq p_1^* \cF$, and $c \seq \sum_{1 \le j < k \le n}\langle \beta_j, \beta_k \rangle_Q$.
When $n=2$, we write 
\[\cF_1 \star \cF_2 \seq \Ind_{\beta_1, \beta_2} (\cF_1 \boxtimes \cF_2)\] for $\cF_k \in D^b_{G(\beta_k)}(X(\beta_k), \kk)$, $k \in \{1,2\}$.
Then, we have the (strong) associativity 
\[ \Ind_{\beta_1, \ldots, \beta_n}(\cF \boxtimes \cG) \simeq \Ind_{\beta_1, \ldots, \beta_k}(\cF) \star \Ind_{\beta_{k+1},\ldots, \beta_n}(\cG).  \]

Let $\cC_\beta \seq \ul{\kk}_{X(\beta)}[\dim X(\beta)]$ be the constant perverse sheaf on $X(\beta)$.
Recall the notation $I^\beta$ from Section \ref{Ssec:qH}.
We set
\begin{equation} \label{eq:Lbeta} 
\cL_\beta \seq \bigoplus_{\nu \in I^\beta} \cL_{\nu}, \qquad \cL_\nu \seq \cC_{\alpha_{\nu_1}} \star \cC_{\alpha_{\nu_2}} \star  \cdots \star \cC_{\alpha_{\nu_{|\beta|}}} = (p_3)_* \ul{\kk}_{\tF_\nu}[\dim \tF_\nu]
\end{equation}
where $\tF_\nu \seq \tF(\alpha_{\nu_1},\ldots,\alpha_{\nu_{|\beta|}})$.
By the decomposition theorem, the complex $\cL_\beta$ is a finite direst sum of shifts of simple perverse sheaves on $X(\beta)$.
Let $\scrQ_\beta$ be the smallest additive, strictly full subcategory of $D^b_{G(\beta)}(X(\beta), \kk)$ that contains $\cL_\beta$ and is closed under taking shifts and direct summands.  
One can show that the category $\scrQ \seq \bigoplus_{\beta \in \cQ^+} \scrQ_\beta$ is stable under the functors $\Ind_{\beta_1, \ldots, \beta_n}$ and $\Res_{\beta_1,\ldots,\beta_n}$, and hence the operation $\star$ defines the structure of $\Z[t^{\pm 1}]$-algebra on the Grothendieck group
$K(\scrQ) = \bigoplus_{\beta \in \sQ^+}K(\scrQ_\beta)$, where the action of $t^{\pm 1}$ corresponds to the cohomological degree shift $[\mp 1]$.

\begin{Thm}[\cite{LusB}] \label{Thm:Lus}
There is a unique isomorphism of $\Z[t^{\pm 1}]$-algebras
\[\chi \colon U^+_t(\fg)_{\Z[t^{\pm 1}]}  \simeq  K(\scrQ)\qquad \text{given by $\chi(e_i) = [\cC_{\alpha_i}]$ for any $i \in I$},\]
through which the homomorphism $\mathrm{r}$ corresponds to $\bigoplus_{\beta, \beta' \in \sQ^+} [\Res_{\beta, 
\beta'}]$, and the involution $\iota$ corresponds to the Verdier duality $\bigoplus_{\beta \in \sQ^+} [\bD_{X(\beta)}]$.
\end{Thm}
\begin{Rem}
The functor $\Res_{\beta, \beta'}$ is isomorphic to $\bD_{X(\beta, \beta')} \circ \Res_{\mathbf{T}, \mathbf{W}}^{\mathbf{V}} \circ \bD_{X(\beta + \beta')}$ in Lusztig's notation \cite[9.2.10]{LusB} with $\mathbf{T} = V^{\beta}, \mathbf{W} = V^{\beta'}, \mathbf{V} = V^{\beta + \beta'}$. 
\end{Rem}

By construction, the algebra $K(\scrQ)$ has a basis $\scrP$ over $\Z[t^{\pm1}]$ consisting of the classes of simple perverse sheaves. 
The canonical basis $\bB$ of $U_t^+(\fg)_{\Z[t^{\pm 1}]}$ is defined by $\bB \seq \chi^{-1}(\scrP)$.
Recall the dual canonical basis $\bB^*$ and its subset $\bB^*(w) = \{ B^*_\ii(\bd) \mid \bd \in \N^{\oplus J} \}$ from Theorem~\ref{Thm:Kimura}.  
For each $\bd \in \N^{\oplus J}$, let $B_\ii(\bd)$ denote the element of $\bB$ dual to $B^*_\ii(\bd)$.
We write $\IC(\bd)$ for a unique simple perverse sheaf in $\scrP$ satisfying $\chi(B_\ii(\bd)) = [\IC(\bd)]$.

\subsection{IC-sheaves corresponding to real positive roots}
Let $\Rep(Q)$ be the category of representations of the quiver $Q$ over $\C$.
We abbreviate $\Hom_{\Rep(Q)}(x,y)$ as $\Hom_Q(x,y)$.

For each real positive root $\alpha \in \sR^+$, there exists an indecomposable representation $x(\alpha)$ of the quiver $Q$, uniquely up to isomorphism, by Kac's theorem.
In what follows, we fix such a representation $x(\alpha)$ for each $\alpha \in \sR^+$. 
We often regard $x(\alpha)$ as a geometric point of the affine space $X(\alpha)$. 
The orbit $O(\alpha) \seq G(\alpha) \cdot x(\alpha)$ is dense in $X(\alpha)$.
Since $\Stab_{G(\alpha)}x(\alpha) = \End_{Q}(x(\alpha))^\times \simeq \C^\times$ is connected, every simple $G(\alpha)$-equivariant perverse sheaf whose support contains $O(\alpha)$ is isomorphic to the constant one $\cC_\alpha$.   
For simplicity, we will use the abbreviation
\begin{align}
\dhom(\alpha, \beta) &\seq \dim_\C \Hom_Q(x(\alpha), x(\beta)), \\
\dext(\alpha, \beta) &\seq \dim_\C \Ext^1_Q(x(\alpha), x(\beta)).
\end{align} 
With these notations, we have
\begin{equation} \label{eq:Eu}
\langle \alpha, \beta \rangle_ Q = \dhom(\alpha, \beta) - \dext(\alpha, \beta)
\end{equation}
for any $\alpha, \beta \in \sR^+$. 

Recall that we defined the positive root $\alpha_{\ii, j} \seq s_{i_1} s_{i_2} \cdots s_{i_{j-1}} (\alpha_{i_j})$ for each $j \in J$.
The next lemma is standard.

\begin{Lem} \label{Lem:dhom}
The followings hold.
\begin{enumerate}
\item
For $j,k \in J$, we have
\begin{align}
\dhom(\alpha_{\ii, j}, \alpha_{\ii, k}) & = 0 \quad \text{if $j < k$}, \\
\dext(\alpha_{\ii,j}, \alpha_{\ii, k}) &= 0 \quad \text{if $j \ge k$}.
\end{align}
\item
The additive full subcategory $\mathrm{add} \{ x(\alpha_{\ii,j})\}_{j \in J}$ of $\Rep (Q)$ consisting of representations isomorphic to finite direct sums of the indecomposable representations $\{ x(\alpha_{\ii, j})\}_{j \in J}$ is closed under extensions.
\end{enumerate}
\end{Lem}
\begin{proof}
For a source $i \in I$ of a quiver $Q$, we have Bernstein-Gelfand-Ponomarev's reflection functors
$\Sigma_i \colon \Rep (s_i Q) \to \Rep (Q)$
and $\Sigma_i^* \colon \Rep(Q) \to \Rep(s_i Q)$.
If $x$ is an indecomposable representation in $\Rep(s_iQ)$ (resp.~$\Rep(Q)$) of dimension vector $\alpha \in \sQ^+$, its reflection $\Sigma_i x$ (resp.~$\Sigma_i^*x$) is indecomposable of dimension vector $s_i \alpha$ if $\alpha \neq \alpha_i$, and zero otherwise (cf.~\cite[Theorem 4.3.9]{DW}). 
In particular, we have
\[ x(\alpha_{\ii, j}) \simeq \Sigma_{i_1} \Sigma_{i_2} \cdots \Sigma_{i_{j-1}} x(\alpha_{i_j})\]
for each $j \in J$. 
Therefore, the assertion (1) follows from the adjunction isomorphism $\Hom_{Q}(x, \Sigma_i y) \simeq \Hom_{s_i Q}(\Sigma_i^*x,y)$ and its Auslander-Reiten dual $\Ext^1_Q(\Sigma_i y, x) \simeq \Ext^1_{s_i Q}(y, \Sigma_i^* x)$ (cf.~\cite[Exercise 4.3.6]{DW}).
Moreover, we see that the category $\mathrm{add} \{ x(\alpha_{\ii,j})\}_{j \in J}$ coincides with the kernel of the right exact functor $\Sigma_{i_\ell}^* \cdots \Sigma_{i_2}^* \Sigma_{i_1}^* \colon \Rep(Q) \to \Rep(s_{i_\ell} \cdots s_{i_2} s_{i_1}Q)$, which implies the assertion (2). 
\end{proof}

The following result is due to Lusztig.

\begin{Prop}[\cite{Lus97}] \label{Prop:Lus}
For each $j \in J$, we have $\IC(\bdelta_j) = \cC_{\alpha_{\ii, j}}$.
\end{Prop}  
\begin{proof}
This follows from \cite[9.4]{Lus97}, which shows that the correspondence among $\mathscr{P}$ induced from the reflection functor $\Sigma_i$  (considered in the proof of Lemma \ref{Lem:dhom})  coincides through $\chi$ with the one among $\mathsf{B}$ induced from the braid symmetry $T_i$.
See \cite[Section 3]{Kato20} and \cite[Theorem 3.5]{XZ17} for some more details.
\end{proof}
For each $\beta \in \sQ^+$, we set
\[\mathrm{KP}_{\ii}(\beta) \seq  \{ \bd = (d_j)_{j \in J} \in \N^{\oplus J} \mid \textstyle \sum_{j \in J} d_j \alpha_{\ii, j} = \beta\}.\] 
The perverse sheaf $\IC(\bd)$ lives in $X(\beta)$ (i.e., belongs to $\scrQ_\beta$) if and only if $\bd \in \mathrm{KP}_\ii(\beta)$.

\subsection{Geometric interpretation of some structure constants}
Recall the the mixed product $\tE^*_\ii(\bep)$ in the quantum unipotent coordinate ring defined in \eqref{eq:Mit}.
The purpose of this and next subsections is to describe $\tE^*_\ii(\bep)$ in terms of the intersection cohomology.
Main results are Propositions \ref{Prop:Mit_geom} and \ref{Prop:Mtg_slice}.
They are analogous to Theorem \ref{Thm:VVqgr} above for the quantum Grothendieck rings.

Let $\beta \in \sQ^+$ and 
fix $\bd = (d_j)_{j \in J}\in \mathrm{KP}_\ii(\beta)$.
We choose an identification
\[V^\beta = (V^{\alpha_{\ii, 1}})^{\oplus d_1} \oplus \cdots \oplus (V^{\alpha_{\ii, \ell}})^{\oplus d_\ell} =
\bigoplus_{j \in J} V^{\alpha_{\ii, j}}\otimes D_j,\] 
where $D_j$ is a $\C$-vector space of dimension $d_j$ (``space of multiplicity'').
Let $G_\bd  \seq \prod_{j \in J} GL(D_j)$.
We have an injective homomorphism $G_\bd  \hookrightarrow G(\beta)$ given by $(g_j)_{j \in J} \mapsto (\id_{V^{\alpha_{\ii, j}}} \otimes g_j)_{j \in J}$, through which we regard $G_\bd$ as a subgroup of $G(\beta)$.
Thus, the group $G_\bd$ acts on  $V^\beta$.

Let $\bep_c = (j_1, \ldots, j_{d})$ denote the unique costandard sequence in $J^\bd$.
We fix a basis $\{v_1, \ldots, v_{d} \}$ of the vector space $\bigoplus_{j \in J}D_j$ such that $v_k \in D_{j_k}$. 
This yields a maximal torus $T_\bd$ of $G_\bd$ consisting of diagonal matrices with respect to the basis. 
The fixed locus $X(\beta)^{T_\bd}$ is identical to the space
\[ X(\bd) \seq X(\alpha_{\ii, 1})^{d_1} \times X(\alpha_{\ii, 2})^{d_2} \times \cdots \times X(\alpha_{\ii, \ell})^{d_\ell}. \]
The quiver representation
\[x(\bd) \seq x(\alpha_{\ii, 1})^{\oplus d_1} \oplus x(\alpha_{\ii, 2})^{\oplus d_2} \oplus\cdots \oplus x(\alpha_{\ii, \ell})^{\oplus d_\ell}\] 
is regarded as a geometric point of $X(\bd) = X(\beta)^{T_\bd}\subset X(\beta)$. 
Let 
\[ i_{x(\bd)} \colon \{x(\bd)\} \hookrightarrow X(\bd) \] denote the inclusion.

Let $d \seq \sum_{j \in J} d_j$. 
The symmetric group $\mathfrak{S}_{d}$ acts on the set $J^\bd$ by place permutations.
For each sequence $\bep \in J^\bd$, let $\sigma_\bep \in \mathfrak{S}_{d}$ denote the element of the smallest length such that $\bep  = (j_{\sigma_\bep(1)}, \ldots, j_{\sigma_\bep(d)}).$
Then, we define a cocharacter $\tau_\bep \in X_*(T_\bd)$ by $\tau_\bep(s)\cdot v_{\sigma_\bep(k)} = s^k v_{\sigma_\bep(k)}$ for $1 \le k \le d$. 
In the notation of \eqref{eq:wtdecomp}, we have
\begin{align}
X(\beta)_{\tau_\bep}^0 &= \bigoplus_{1 \le k \le d} \bigoplus_{h \in Q_1} \Hom_\C(V_{\mathrm{s}(h)}^{\alpha_{\ii, \ep_k}}\otimes \C v_{\sigma_{\bep}(k)}, V_{\mathrm{t}(h)}^{\alpha_{\ii, \ep_k}}\otimes \C v_{\sigma_{\bep}(k)}) = X(\bd), \\
X(\beta)_{\tau_\bep}^+ &= \bigoplus_{1 \le k < l\le d} \bigoplus_{h \in Q_1} \Hom_\C(V_{\mathrm{s}(h)}^{\alpha_{\ii, \ep_k}}\otimes \C v_{\sigma_{\bep}(k)}, V_{\mathrm{t}(h)}^{\alpha_{\ii, \ep_l}}\otimes \C v_{\sigma_{\bep}(l)}).
\end{align}
In particular, we have an isomorphism
\[ 
F(\bep) \seq X(\beta)_{\tau_\bep}^0 \oplus X(\beta)_{\tau_\bep}^+ \simeq F(\alpha_{\ii, \ep_1}, \alpha_{\ii, \ep_2}, \cdots, \alpha_{\ii, \ep_{d}}).
\]
Let
\[
\xymatrix{
X(\bd)& \ar@{->}[l]_-{\kappa_\bep} F(\bep) \ar@{->}[r]^-{\iota_\bep} & X(\beta)
}
\]
be the diagram defined as in \eqref{eq:resdiag}. 

\begin{Prop} \label{Prop:Mit_geom}
For any $\bep = (\ep_1, \ldots, \ep_d)\in J^\bd$, we have
\[
\tE^*_{\ii}(\bep) = t^{c(\bep)} \sum_{\bd' \in \mathrm{KP}_\ii(\beta)} \left(\sum_{n \in \Z} t^n \dim_\kk \rH^n(i_{x(\bd)}^! \kappa_{\bep*} \iota_\bep^! \IC(\bd')) \right) \tB^*_\ii(\bd')
\]
in the quantum unipotent coordinate ring $A_t[N(w)]_{\Z[t^{\pm 1/2}]}$, where
\[ c(\bep) = - \dim X(\bd) - \sum_{1 \le k < l \le d}(\dhom(\alpha_{\ii, \ep_k}, \alpha_{\ii,\ep_l}) - \dhom(\alpha_{\ii, \ep_l}, \alpha_{\ii, \ep_k}) ).\]
\end{Prop}
\begin{proof}
By definition, we have
\[ \tE^*_{\ii}(\bep) = t^{c_1} E^*_{\ii, \ep_1} \cdots E^*_{\ii, \ep_{d}} = t^{c_2} \sum_{\bd' \in \mathrm{KP}_\ii(\beta)} \langle E^*_{\ii, \ep_1} \cdots E^*_{\ii, \ep_d}, B_\ii(\bd') \rangle \tB^*_\ii(\bd'), \]
where
\begin{align}
c_1 &= \sum_{1 \le k < l \le d}\gamma_\ii(\bdelta_{\ep_k}, \bdelta_{\ep_l}) - \frac{1}{4} \sum_{1 \le k \le d} (\alpha_{\ii, \ep_k}, \alpha_{\ii, \ep_k}), \\
c_2 &= c_1 + \frac{1}{4}(\beta, \beta) = \sum_{1 \le k < l \le d} \left(\gamma_\ii(\bdelta_{\ep_k}, \bdelta_{\ep_l}) + \frac{1}{2}(\alpha_{\ii, \ep_k}, \alpha_{\ii, \ep_l})\right).
\end{align}
By Theorem~\ref{Thm:Lus}, for each $\bd' \in \mathrm{KP}_\ii(\beta)$, we have
\[
\langle E^*_{\ii, \ep_1} \cdots E^*_{\ii, \ep_{d}}, B_\ii(\bd') \rangle
= \langle E^*_{\ii, \ep_1} \otimes \cdots \otimes E^*_{\ii, \ep_{d}}, \chi^{-1}[\Res_{\alpha_{\ii, \ep_1}, \ldots, \alpha_{\ii, \ep_{d}}} \IC(\bd')] \rangle.
\]
This is equal to the graded multiplicity of the constant perverse sheaf 
\[\cC_{\alpha_{\ii, \ep_1}}\boxtimes \cdots \boxtimes \cC_{\alpha_{\ii, \ep_{d}}} \simeq \ul{\kk}_{X(\bd)}[\dim X(\bd)]\] 
in $\Res_{\alpha_{\ii, \ep_1}, \ldots, \alpha_{\ii, \ep_d}} \IC(\bd')$ by Proposition~\ref{Prop:Lus}.   
Since we have 
\[i_{x(\bd)}^!( \ul{\kk}_{X(\bd)}[\dim X(\bd)]) = \ul{\kk}_{\{x(\bd)\}} [-\dim X(\bd)]\] 
and $\ul{\kk}_{X(\bd)}[\dim X(\bd)]$ is the unique simple $G_\bd$-equivariant perverse sheaf on $X(\bd)$ with a non-trivial (co)stalk at $x(\bd)$, the graded multiplicity in question can be computed as the Poincar\'e polynomial of 
\[i_{x(\bd)}^! \Res_{\alpha_{\ii, \ep_1}, \ldots, \alpha_{\ii, \ep_{d}}} \IC(\bd') [\dim X(\bd)] = i_{x(\bd)}^! \kappa_{\bep*} \iota_{\bep}^! \IC(\bd')[c_3], \]
where
\[c_3 = \dim X(\bd) + \sum_{1 \le k < l \le d} \langle \alpha_{\ii, \ep_k}, \alpha_{\ii, \ep_l} \rangle_Q.\]
Therefore, we get
\[\tE^*_{\ii}(\bep) = t^{c_2 - c_3} \sum_{\bd' \in \mathrm{KP}_\ii(\beta)} \left(\sum_{n \in \Z} t^n \dim_\kk \rH^n(i_{x(\bd)}^! \kappa_{\bep*} \iota_\bep^! \IC(\bd')) \right) \tB^*_\ii(\bd').\]
It remains to observe
\begin{align}
 c_2 - c_3 &= - \dim X(\bd) +   \sum_{1 \le k < l \le d} \left(\gamma_\ii(\bdelta_{\ep_k}, \bdelta_{\ep_l}) - \frac{1}{2}(\langle \alpha_{\ii, \ep_k}, \alpha_{\ii, \ep_l}\rangle_Q -  \langle \alpha_{\ii, \ep_l}, \alpha_{\ii, \ep_k} \rangle_Q )\right) \\
&= - \dim X(\bd) +   \sum_{1 \le k < l \le d} \left(
-\langle \alpha_{\ii, \ep_k}, \alpha_{\ii,\ep_l} \rangle_{Q} + \langle \alpha_{\ii, \ep_l}, \alpha_{\ii,\ep_k} \rangle_{Q}  - \dext(\alpha_{\ii, \ep_k}, \alpha_{\ii, \ep_l}) +\dext(\alpha_{\ii,\ep_l}, \alpha_{\ii, \ep_k})\right) \\
&= - \dim X(\bd) - \sum_{1 \le k < l \le d}(\dhom(\alpha_{\ii, \ep_k}, \alpha_{\ii,\ep_l}) - \dhom(\alpha_{\ii, \ep_l}, \alpha_{\ii, \ep_k}) ),
\end{align}
where the first equality follows from \eqref{eq:inner}, the second one follows from Lemma~\ref{Lem:gamma-ext} below, and the last one follows from \eqref{eq:Eu}.
\end{proof}

\begin{Lem} \label{Lem:gamma-ext}
For any $j,k \in J$, we have
\[ \gamma_{\ii}(\bdelta_j, \bdelta_k) = -\frac{1}{2}(\langle \alpha_{\ii, j}, \alpha_{\ii, k} \rangle_{Q} - \langle \alpha_{\ii, k}, \alpha_{\ii, j} \rangle_{Q} ) - \dext(\alpha_{\ii, j}, \alpha_{\ii, k}) +\dext(\alpha_{\ii, k}, \alpha_{\ii, j}). \]
\end{Lem}
\begin{proof}
Since both sides of the desired equality are skew-symmetric, we may assume that $j < k$. 
Then, we have $\gamma_{\ii}(\bdelta_j, \bdelta_k) = (\alpha_{\ii, j}, \alpha_{\ii, k})/2$ by definition.
Now, the result follows from the formulas \eqref{eq:inner}, \eqref{eq:Eu}, and Lemma~\ref{Lem:dhom}. 
\end{proof}

\subsection{Lusztig's transversal slice} 
\label{Ssec:slice}
In this subsection, we restrict the above geometric setting to a certain transversal slice $S(\bd)$ in $X(\beta)$ considered by Lusztig \cite[Section 10]{Lus90}.
This is an important step to apply the facts from Section \ref{sec:preliminary}. 

To define the  transversal  slice $S(\bd)$, first we recall the following general fact about quiver representations. 
Let $x \in X(\beta), x' \in X(\beta')$ be two representations of $Q$.
We have an exact sequence of $\C$-vector spaces
\[ 0 \to \Hom_{Q}(x, x') 
\to L(\beta, \beta')
\to E(\beta, \beta')
\to \Ext^1_Q(x, x') 
\to 0,\] 
where the middle map is given by 
\[
L(\beta, \beta') \ni \varphi = (\varphi_i)_{i \in I} \mapsto  (\varphi_{\mathrm{t}(h)} x_{h} - x'_{h} \varphi_{\mathrm{s}(h)} )_{h \in Q_1} \in E(\beta, \beta').
\] 
Note that the equality \eqref{eq:Eu} follows from this.

Now, we retain the notation from the previous subsection and consider the special case when $x = x' = x(\bd)$ and $\beta = \beta' = \sum_{j \in J}d_j \alpha_{\ii, j}$ to get the exact sequence
\begin{equation} \label{eq:HLXE}
0 \to \Hom_{Q}(x(\bd), x(\bd)) 
\to L(\beta, \beta)
\xrightarrow{\xi} X(\beta)
\to \Ext^1_Q(x(\bd), x(\bd)) 
\to 0.
\end{equation} 
Note that  the middle map $\xi$ is $\Stab_{G(\beta)} x(\bd)$-equivariant, and hence, as $G_\bd \subset \Stab_{G(\beta)} x(\bd)$, it is $G_\bd$-equivariant.
Since $G_\bd$ is a reductive group, we can find a $G_\bd$-stable linear subspace $E(\bd)$ of $X(\beta)$ such that \[X(\beta) = \Image \xi \oplus E(\bd)\] 
as $G_\bd$-representation.
By \eqref{eq:HLXE}, we have $E(\bd) \simeq \Ext^1_{Q}(x(\bd), x(\bd))$ as $\C$-vector spaces. 
Let 
\[S(\bd) \seq x(\bd) + E(\bd)\]
be the affine subspace of $X(\beta)$. 
Note that each geometric point of $S(\bd)$ is a quiver representation obtained as an extension of indecomposable ones $\{x(\alpha_{\ii, j})\}_{j \in J}$.   
Lemma~\ref{Lem:dhom}~(2) implies that we have a finite stratification 
\begin{equation} \label{eq:strS}
S(\bd) = \bigsqcup_{\bd' \in \mathrm{KP}_\ii(\beta)} O(\bd') \cap S(\bd), \quad \text{where $O(\bd') \seq G(\beta) \cdot x(\bd')$}. 
\end{equation} 
The variety $S(\bd)$ is a transversal slice through $x(\bd)$, meaning that it intersects transversally with each orbit $O(\bd')$, $\bd' \in \mathrm{KP}_\ii(\beta)$.

For each sequence $\bep \in J^\bd$, we define
\begin{align}
X(\bep) &\seq x(\bd) + X(\beta)^+_{\tau_\bep} = \kappa_\bep^{-1}(x(\bd)), \\
S(\bep) &\seq x(\bd) + E(\bd)^+_{\tau_\bep} = S(\bd) \cap F(\bep). 
\end{align}
We have a commutative diagram:
\begin{equation} \label{eq:diagS}
\vcenter{
\xymatrix{
X(\bd)  
& \ar@{->}[l]_-{\kappa_\bep} F(\bep) \ar@{->}[r]^-{\iota_\bep}
& X(\beta) 
\\
\{x(\bd)\} \ar@{->}[u]_-{i_{x(\bd)}} \ar@<-3pt>@{->}[r]_-{i_2}
& \ar@<-3pt>@{->}[l]_-{p} X(\bep) \ar@{->}[u]_-{i_1} 
& 
\\
\{x(\bd)\} \ar@{=}[u] \ar@{->}[r]^-{i_{\bep, x(\bd)}}
&  S(\bep)\ar@{->}[r]^-{i_\bep} \ar@{->}[u]_-{i_3}
& S(\bd). \ar@{->}[uu]_{i_{S(\bd)}} 
}}
\end{equation}
Here the arrow $p$ is the obvious map, and the arrows $i_1, i_2, i_3, i_{S(\bd)}, i_{\bep, x(\bd)}, i_{\bep}$ are the inclusions.
Note that the upper left square and the right square are both cartesian.
All the varieties in the diagram~\eqref{eq:diagS} are stable under the action of the maximal torus $T_\bd \subset G_\bd$ and all the morphisms in the diagram~\eqref{eq:diagS} are $T_\bd$-equivariant.

\begin{Lem} \label{Lem:slice}
We have a natural isomorphism
\begin{equation} \label{eq:iki=}
i_{x(\bd)}^! \kappa_{\bep *} \iota_\bep^! \simeq i_{\bep, x(\bd)}^* i_{\bep}^!  i_{S(\bd)}^! [\dim X(\beta) - \dim S(\bd) + c(\bep)] 
\end{equation}
of functors from $D^b_{G(\beta)}(X(\beta), \kk)$ to $D^b_{T_\bd} (\{x(\bd)\}, \kk)$.
\end{Lem}
\begin{proof}
By the base change and Proposition~\ref{Prop:attr}, we have
\begin{equation} \label{eq:basechange}
i_{x(\bd)}^! \kappa_{\bep *} \iota_\bep^! 
\simeq  p_* i_1^! \iota_\bep ^!
\simeq i_2^* i_1^! \iota_\bep ^!
\simeq i_{\bep, x(\bd)}^* i_3^* i_1^! \iota_\bep ^!
\end{equation}
in the notation from the  diagram  \eqref{eq:diagS}. 
Let $U_\bep$ be the  unipotent subgroup of $G(\beta)$ whose Lie algebra is $L(\beta, \beta)^+_{\tau_\bep}$.
The varieties $F(\bep)$ and $X(\bep)$  are  stable under the action of $U_\bep$, and hence they are $(U_\bep \rtimes T_\bd)$-varieties.
In particular, for any $\mathcal{F} \in D^b_{G(\beta)}(X(\beta), \kk)$, the $!$-restriction $i_1^! \iota_\bep^! \mathcal{F}$ can be seen as an object of $D^b_{U_\bep \rtimes T_\bd}(X(\bep), \kk)$.
We shall show a natural isomorphism
\begin{equation} \label{eq:i3}
i_3^*  \simeq i_3^!  [2 (\dim X(\bep) - \dim S(\bep))] 
\end{equation}
as functors from $D^b_{U_\bep \rtimes T_\bd}(X(\bep), \kk)$ to $D^b_{T_\bd}(S(\bep), \kk)$.
Consider the factorization $i_3 = \pi_3 \circ s_3$:
\[
\xymatrix{
 S(\bep) \ar[rd]_-{s_3} \ar[rr]^-{i_3} & & X(\bep), \\
 & U_\bep \times S(\bep) \ar[ur]_-{\pi_3} &
}
\]
where $s_3$ and $\pi_3$ are $G_\bd$-equivariant morphisms defined by $s_3 (x) \seq (1, x)$ and $\pi_3(g, x) \seq g\cdot x$.
The morphism $\pi_3$ is a locally trivial fibration. 
Indeed, its differential at the point $(1, x(\bd))$ is naturally identified with the linear map 
\[ L(\beta, \beta)^+_{\tau_\bep} \oplus E(\bd)^+_{\tau_\bep} \to X(\beta)^+_{\tau_\bep} \quad \text{given by $(u,v) \mapsto \xi(u) + v$} \]
in the notation of \eqref{eq:HLXE}.
This is surjective thanks to the exactness of the sequence obtained from \eqref{eq:HLXE} by taking $(-)^+_{\tau_\bep}$-parts. 
Since the action of $\Gm$ given by $\tau_\bep$ contracts the variety $U_\bep \times S(\bep)$ (resp.~$X(\bep)$) to the single point $(1, x(\bd))$ (resp.~$x(\bd)$), it follows that the morphism $\pi_3$ is surjective and its differential is surjective at any points.
Thus, $\pi_3$ is a locally trivial fibration with smooth fibers, and hence we have 
\[ \pi_3^* \simeq \pi_3^![2(\dim X(\bep) - \dim (U_\bep \times S(\bep)))]\]
as functors from $D^b_{U_\bep \rtimes T_\bd}(X(\bep), \kk)$ to $D^b_{U_\bep \rtimes T_\bd}(U_\bep \times S(\bep), \kk)$.
On the other hand, we have the induction equivalence 
\[ s_3^* \simeq s_3^! [2 (\dim (U_\bep \times S(\bep)) - \dim S(\bep))] \colon D^b_{U_\bep \rtimes T_\bd}(U_\bep \times S(\bep), \kk) \xrightarrow{\sim} D^b_{T_\bd}(S(\bep), \kk).\]
Combining the above isomorphisms with the natural isomorphisms $i_3^* \simeq s_3^* \pi_3^*$ and $i_3^! \simeq s_3^! \pi_3^!$, we arrive at the isomorphism \eqref{eq:i3}.

Now, the isomorphisms \eqref{eq:basechange} and \eqref{eq:i3} yield an isomorphism
\[ i^!_{x(\bd)} \kappa_\bep \iota^!_\bep \simeq i^*_{\bep, x(\bd)} i_3^! i_1^! \iota_\bep^! [2(\dim X(\bep) - \dim S(\bep))] \simeq i^*_{\bep, x(\bd)} i_\bep^! i_{S(\bd)}^! [2(\dim X(\bep) - \dim S(\bep))]. \]
It remains to check that the number $2(\dim X(\bep) - \dim S(\bep))$ coincides with $\dim X(\beta) - \dim S(\bd) + c(\bep)$.
This is done by noting the equalities
\begin{align}
\dim X(\bep) &= \dim X(\beta)^+_{\tau_\bep}, 
& \dim X(\beta) &= \dim X(\beta)^+_{\tau_\bep} + \dim X(\bd) + \dim X(\beta)^-_{\tau_\bep}, \\
\dim S(\bep) & = \dim E(\bd)^+_{\tau_\bep}, 
& \dim S(\bd) &= \dim E(\bd) = \dim E(\bd)^+_{\tau_\bep} + \dim E(\bd)^-_{\tau_\bep},
\end{align}
and computing as follows:
\begin{align} 
& 2(\dim X(\bep) - \dim S(\bep)) - (\dim X(\beta) - \dim S(\bd)) \\
&=  - \dim X(\bd) + (\dim X(\beta)^+_{\tau_\bep} - \dim E(\bd)^+_{\tau_\bep}) - (\dim X(\beta)^-_{\tau_\bep} - \dim E(\bd)^-_{\tau_\bep}) \\
&= - \dim X(\bd) + (\dim L(\beta, \beta)^+_{\tau_\bep} - \dim H(\bd)^+_{\tau_\bep}) - (\dim L(\beta, \beta)^-_{\tau_\bep} - \dim H(\bd)^-_{\tau_\bep}) \\
&= - \dim X(\bd) - \dim H(\bd)^+_{\tau_\bep} +  \dim H(\bd)^-_{\tau_\bep} = c(\bep), \end{align}
where we put $H(\bd) \seq \Hom_Q(x(\bd), x(\bd))$.
Here, the second equality follows from the exactness of the sequence \eqref{eq:HLXE}, and the third one is due to an obvious equality $\dim L(\beta, \beta)^+_{\tau_\bep} = \dim L(\beta, \beta)^-_{\tau_\bep}$.   
Thus, we obtain the desired isomorphism \eqref{eq:iki=}.
\end{proof}

\begin{Prop} \label{Prop:ICa}
For any $\bd \in \N^{\oplus J}$, we have
\[ \IC(\bd) = \IC(\ol{O(\bd)}, \kk).  \]
\end{Prop}
\begin{proof}
First we show $\IC(\ol{O(\bd)}, \kk) \in \scrP$.
Let $\bep_c = (j_1, \ldots, j_{d})$ be the unique costandard sequence in $J^\bd$ as before.
Then $\bep_s = \bep_c^\op = (j_{d}, \ldots, j_1)$ is the unique standard sequence in $J^\bd$.
The image of the proper map 
\[ p_3 \colon \tF(\bep_s) \seq  \tF(\alpha_{\ii, j_{d}}, \ldots, \alpha_{\ii, j_1}) \to X(\beta),\]
which appeared in the definition of the induction functor $\Ind_{\alpha_{\ii, j_{d}}, \ldots, \alpha_{\ii, j_1}}$, contains a dense subset consisting of quiver representations $x \in X(\beta)$ which respects an $I$-graded flag $V^\beta = F^{d} \supset F^{d-1} \supset \cdots \supset F^1 \supset F^0 = 0$ and satisfies $x |_{F^k/F^{k-1}} \simeq x(\alpha_{\ii, j_k})$ for $1 \le k \le d$.  
Since $\dext(\alpha_{\ii, j_l}, \alpha_{\ii, j_k}) = 0$ for $1 \le k \le l \le d$ by Lemma~\ref{Lem:dhom} (1), such a representation $x$ is always isomorphic to $x(\bd)$, and hence $p_3(\tF(\bep_s)) = \ol{O(\bd)}$.
Thus, the support of the object
\[ \cC_{\alpha_{\ii, j_{d}}} \star \cdots \star \cC_{\alpha_{\ii, j_1}} =  (p_3)_* \ul{\kk}_{\tF(\bep_s)} [\dim \tF(\bep_s)]\]
coincides with the orbit closure $\ol{O(\bd)}$.
Since the complex $\IC(\ol{O(\bd)}, \kk)$ is the unique simple $G(\beta)$-equivariant perverse sheaf on $X(\beta)$ whose support coincides with $\ol{O(\bd)}$, some of its shifts must contribute to $\cC_{\alpha_{\ii, j_{d}}} \star \cdots \star \cC_{\alpha_{\ii, j_1}}$ as direct summands. 
By Theorem~\ref{Thm:Lus} and Proposition~\ref{Prop:Lus}, the object $\cC_{\alpha_{\ii, j_{d}}} \star \cdots \star \cC_{\alpha_{\ii, j_1}}$ belongs to the category $\scrQ$. 
Therefore, $\IC(\ol{O(\bd)}, \kk)$ belongs to $\scrP$.    

Now, in order to verify $\IC(\bd) = \IC(\ol{O(\bd)}, \kk)$, it suffice to show the equality
\begin{equation} \label{eq:EvsIC}
\langle E^*_\ii(\bd), \chi^{-1}[\IC(\ol{O(\bd)}, \kk)]\rangle = 1
\end{equation} 
by the characterization of $B_\ii^*(\bd)$ in Theorem~\ref{Thm:Kimura}.  
By the same computation as in the proof of Proposition~\ref{Prop:Mit_geom}, the LHS of \eqref{eq:EvsIC} is equal to 
\begin{align}
& t^{c(\bep_s)} \sum_{n \in \Z} t^n \dim_\kk \rH^n(i_{x(\bd)}^! \kappa_{\bep_s *} \iota_{\bep_s}^! \IC(\ol{O(\bd)}, \kk))\\
&= \sum_{n \in \Z} t^n \dim_\kk \rH^n(i_{\bep_s}^! i_{S(\bd)}^! \IC(\ol{O(\bd)}, \kk)[\dim X(\beta) - \dim S(\bd)]) \\
&=  \sum_{n \in \Z} t^n \dim_\kk \rH^n((i_{S(\bd)} \circ i_{\bep_s})^! \IC(\ol{O(\bd)}, \kk)[\dim O(\bd)]),
\end{align}
where the first equality is due to Lemma~\ref{Lem:slice}.
Note that $S({\bep_s}) = \{x(\bd)\}$ and $i_{S(\bd)} \circ i_{\bep_s}$ is the inclusion $\{ x(\bd)\} \hookrightarrow X(\beta)$.
In particular, we have 
\[(i_{S(\bd)} \circ i_{\bep_s})^! \IC(\ol{O(\bd)}, \kk)[\dim O(\bd)] \simeq \ul{\kk}_{\{x(\bd)\}}\] 
and hence the desired equality \eqref{eq:EvsIC} follows.
\end{proof}

\begin{Prop} 
\label{Prop:Mtg_slice}
For any $\bd \in \mathrm{KP}_\ii(\beta)$ and $\bep \in J^\bd$, we have the equality
\begin{equation} \label{eq:Mtg_slice}
\tE_{\ii}^*(\bep) = \sum_{\bd' \in \mathrm{KP}_\ii(\beta)} \left(\sum_{n \in \Z} t^n \dim_\kk \rH^n(i_{\bep, x(\bd)}^* i_{\bep}^! \IC(\ol{O(\bd')}\cap S(\bd), \kk)) \right) \tB^*_\ii(\bd')
\end{equation}
 in the quantum unipotent coordinate ring $A_t[N(w)]_{\Z[t^{\pm 1/2}]}$.
\end{Prop}
\begin{proof}
Since $S(\bd)$ is a transversal slice, we have
\[ i_{S(\bd)}^! \IC(\ol{O(\bd')}, \kk) [\dim X(\beta) - \dim S(\bd)] \simeq \IC(\ol{O(\bd')} \cap S(\bd), \kk) \]
for any $\bd' \in \mathrm{KP}_\ii(\beta)$ (cf.~\cite[Theorem 5.4.1]{GM2}).
Therefore, the assertion follows from Proposition~\ref{Prop:Mit_geom} together with Lemma~\ref{Lem:slice} and Proposition~\ref{Prop:ICa}. 
\end{proof}

\subsection{Geometric realization of symmetric quiver Hecke algebras}
Now, we briefly review the geometric interpretation of the symmetric quiver Hecke algebras  due  to Varagnolo-Vasserot \cite{VV11}.
Let $\beta \in \sQ^+$.
Recall the complex $\cL_\beta$ defined in \eqref{eq:Lbeta}.

\begin{Thm}[{\cite{VV11}}] \label{Thm:VVg}
There is an isomorphism of graded $\kk$-algebras
\begin{equation} \label{eq:VVg}
H_{\beta} \simeq \Hom^\bullet_{G_\beta}(\cL_\beta, \cL_\beta). \end{equation}
\end{Thm}

For each $b \in \bB$, we fix a representative $\IC_b \in \scrQ$ of the class $\chi(b)$.
Proposition~\ref{Prop:ICa} implies $\IC_b \simeq \IC(\ol{O(\bd)}, \kk)$ if $b = B_\ii(\bd)$ for some $\bd \in \N^{\oplus J}$.
By the decomposition theorem, we have 
\begin{equation} \label{eq:DTL}
\cL_\beta \simeq \bigoplus_{b \in \bB_\beta} \IC_b \otimes_\kk L_b^\bullet
\end{equation}
for some finite-dimensional self-dual  $\Z$- graded vector space $L_b^\bullet$,  where $\bB_\beta \seq \bB \cap U_t^+(\fg)_\beta$. 
Through the isomorphism \eqref{eq:VVg} in Theorem~\ref{Thm:VVg}, we can regard $L_b$ as a graded simple $H_\beta$-module. 
The set $\{ L_b^\bullet \mid b \in \bB_\beta \}$ gives a complete system of representatives of the self-dual simple isomorphism classes of the category $\Mm_{\fdim,\beta}^\bullet$.   
Under the isomorphism \eqref{eq:VV} in Theorem~\ref{Thm:VV}, the class $[L_b^\bullet]$ corresponds to the dual element $b^* \in \bB^*$.
Taking the total perverse cohomology, we define 
\[ \bar{\cL}_\beta := \bigoplus_{k \in \Z} {}^p\mathcal{H}^k(\cL_\beta) = \bigoplus_{b \in \bB_\beta} \IC_b \otimes_\kk L_b,\]
where $L_b$ denotes the ungraded finite-dimensional $\C$-vector space obtained from $L_b^\bullet$ by forgetting the grading.  
Since $\bar{\cL}_\beta$ is a semisimple perverse sheaf, its Yoneda algebra 
\[ \Hom^\bullet_{G_\beta}(\bar{\cL}_\beta,\bar{\cL}_\beta)\]
is non-negatively graded, whose degree zero part is isomorphic to the semisimple algebra $\bigoplus_{b \in \bB_\beta} \End_\kk(L_b)$. 
Let $\Hom^\bullet_{G_\beta}(\bar{\cL}_\beta, \bar{\cL}_\beta)^\wedge$ denote its completion along the grading.

\begin{Cor}
There is an isomorphism of $\kk$-algebras
\[ \hH_{\beta} \simeq \Hom^\bullet_{G_\beta}(\bar{\cL}_\beta, \bar{\cL}_\beta)^\wedge. \]
The set $\{ L_b \mid b \in \bB \}$ gives a complete system of representative of the simple isomorphism classes of the category $\Mm^{\nilp}_{\fdim} = \hH \mof$.
Through the isomorphism in Corollary~\ref{Cor:VV}, the class $[L_b]$ corresponds to the specialized element $b^*|_{t=1}$. 
\end{Cor}

When $b = B_\ii(\bd)$ for some $\bd \in \N^{\oplus J}$, we write $L_\ii(\bd)$ for $L_b$.
Note that this notation is compatible with the previous one in Section \ref{ssec:pbwqH}.

\subsection{Geometric realization of mixed convolution products}
Let $\beta \in \sQ^+$ and $\bd \in \mathrm{KP}_\ii(\beta)$.
In this subsection, we establish a geometric realization of the mixed products $M_\ii(\bep)$ and their deformations $\tM_\ii(\bep)$ for $\bep \in J^\bd$.

Let $M^\bullet$ be a graded $H_\beta$-module and $z$ an indeterminate of degree $2$. 
Endow the graded $\kk$-vector space $M^\bullet[z] \seq M^\bullet \otimes \kk[z]$ with an $H_\beta$-module structure by the same formulas as \eqref{eq:affx} with $a(z)=z$.
The resulting graded $H_\beta$-module $M^\bullet[z]$ is called the \emph{affinization} of $M^\bullet$.
Note that, for any $j \in J$, we have an isomorphism
\[ \tL_{\ii,j} = (L_{\ii,j})_{jz} \simeq L^\bullet_{\ii,j}[z] \otimes_{\kk[z]} \Oo \]
of $\hH_{\alpha_{\ii, j}}$-modules, where $\kk[z] \to \Oo = \kk[\![z]\!]$ is given by $z \mapsto jz$.
A proof of the following lemma is given later in Section~\ref{Ssec:pfLem}.

\begin{Lem} \label{Lem:Lg}
For each $j \in J$, we have an isomorphism of graded $H_{\alpha_{\ii, j}}$-modules
\[ L^\bullet_{\ii,j}[z] \simeq \Hom^\bullet_{G(\alpha_{\ii, j})}((i_{O(\alpha_{\ii, j})})_! \ul{\kk}_{O(\alpha_{\ii, j})}, \cL_{\alpha_{\ii, j}})\langle \dim X(\alpha_{\ii, j}) \rangle, \]
where $i_{O(\alpha_{\ii, j})} \colon O(\alpha_{\ii, j}) \hookrightarrow X(\alpha_{\ii, j})$ denotes the inclusion.
\end{Lem}

We retain the notation from the previous subsections.
Let us consider a cocharacter $\rho^\vee \in X_*(T_\bd)$ given by 
\[ \rho^\vee(s) |_{D_j} =s^j \cdot \id_{D_j}\]
for any $j \in J$.
In what follows, we regard a $T_\bd$-variety as a $\Gm$-variety through $\rho^\vee \colon \Gm \to T_\bd$. 

\begin{Prop} \label{Prop:Mig}
For each $\bep \in J^\bd$, we have an isomorphism of $\hH_\beta$-modules
\[\tM_{\ii}(\bep) \simeq \wh{\rH}^\bullet_{\Gm}(i_{x(\bd)}^!\kappa_{\bep *} \iota_\bep^! \bar{\cL}_\beta), \]
which specializes to
\[ M_{\ii}(\bep) \simeq \rH^\bullet(i_{x(\bd)}^!\kappa_{\bep *} \iota_\bep^! \bar{\cL}_\beta).\]
\end{Prop}

\begin{proof}
From the definition, we have
\begin{align} \tM_{\ii} (\bep) \simeq ( L^\bullet_{\ii,\ep_1}[z_1] \star \cdots \star L^\bullet_{\ii,\ep_{d}}[z_{d}] ) \otimes_{\C[z_1,  \ldots, z_{d}]} \Oo.
\end{align}
Here the tensor product $- \otimes_{\kk[z_1,  \ldots, z_{d}]} \Oo$ is taken with respect to the $\kk$-algebra homomorphism $\kk[z_1, \ldots, z_{d}] \to \Oo = \kk[\![ z ]\!]$ given by $z_k \mapsto \ep_k z$ for $1 \le k \le d$, which is identified with the homomorphism $\rH^\bullet_{T_\bd}(\mathrm{pt}, \kk) \to \wh{\rH}_{\Gm}^\bullet(\mathrm{pt}, \kk)$ induced from the cocharacter $\rho^\vee \colon \Gm \to T_\bd$.
Unpacking the definition, we have
\[
  L^\bullet_{\ii,\ep_1}[z_1] \star \cdots \star L^\bullet_{\ii,\ep_{d}}[z_{d}] =
H_\beta e(\bep) \otimes_{H_\bep}  \left( L^\bullet_{\ii,\ep_1}[z_1] \otimes \cdots \otimes L^\bullet_{\ii,\ep_{d}}[z_{d}] \right), 
\]
where we abbreviate $e(\bep) \seq e(\alpha_{\ii, \ep_1}, \ldots, \alpha_{\ii, \ep_{d}})$ and  $H_\bep \seq H_{\alpha_{\ii, \ep_1}, \ldots, \alpha_{\ii, \ep_{d}}}$.
Thanks to Theorem~\ref{Thm:VVg}, Lemma~\ref{Lem:Lg} and \cite[Proposition 6.7.5]{AcharBook}, we have the graded isomorphisms
\begin{gather} 
H_\beta e(\bep)  \simeq \Hom^\bullet_{G(\beta)} (\cL_{\alpha_{\ii, \ep_1}} \star \cdots \star \cL_{\alpha_{\ii, \ep_{d}}}, \cL_\beta) \simeq \Hom^\bullet_{G^\bd}(\cL^\bd, \kappa_{\bep*} \iota_\bep^! \cL_\beta) \langle  - c  \rangle, 
\\
H_\bep  \simeq \Hom^\bullet_{G^\bd}(\cL^\bd, \cL^\bd),
\\
 L^\bullet_{\ii,\ep_1}[z_1] \otimes \cdots \otimes L^\bullet_{\ii,\ep_{d}}[z_{d}]   \simeq \Hom_{G^\bd}((i_{O^\bd})_! \ul{\kk}_{O^\bd}, \cL^\bd) \langle \dim X(\bd) \rangle,
\end{gather}
where $G^\bd \seq G(\alpha_{\ii, 1})^{d_1} \times \cdots \times G(\alpha_{\ii, \ell})^{d_\ell}$, $O^\bd \seq O(\alpha_{\ii, 1})^{d_1} \times \cdots \times O(\alpha_{\ii, \ell})^{d_\ell} \subset X(\bd)$, $\cL^\bd \seq \cL_{\alpha_{\ii, 1}}^{\boxtimes d_1} \boxtimes \cdots \boxtimes \cL_{\alpha_{\ii, \ell}}^{\boxtimes d_\ell} \in D^b_{G^\bd}(X(\bd), \kk)$ and $c \seq \sum_{1 \le k < l \le d} \langle \alpha_{\ii, \ep_k}, \alpha_{\ii, \ep_l} \rangle_Q$.
(Be aware that the group $G^\bd$ is different from the group $G_\bd$ from the previous sections. In fact, we have $G_\bd \cap G^\bd = T_\bd$.) 
Note that both $\cL^\bd$ and $\kappa_{\bep *} \iota_\bep^! \cL_\beta$ belong to the category $\scrQ^\bd \seq \scrQ_{\alpha_{\ii, 1}}^{\boxtimes d_1} \boxtimes \cdots \boxtimes \scrQ_{\alpha_{\ii, \ell}}^{\boxtimes d_\ell} \subset D^b_{G^\bd}(X(\bd), \kk)$, and every indecomposable object of $\scrQ^\bd$ appears up to shift as a direct summand of the object $\cL^\bd$ by definition.  
Therefore, we have 
\begin{align}
& L^\bullet_{\ii,\ep_1}[z_1] \star \cdots \star L^\bullet_{\ii,\ep_{d}}[z_{d}]  \\
& \simeq 
\Hom^\bullet_{G^\bd}(\cL^\bd, \kappa_{\bep*} \iota_\bep^! \cL_\beta) \langle  -c  \rangle \otimes_{\Hom^\bullet_{G^\bd}(\cL^\bd, \cL^\bd)} 
\Hom_{G^\bd}((i_{O^\bd})_! \ul{\kk}_{O^\bd}, \cL^\bd) \langle \dim X(\bd) \rangle \\
& \simeq \Hom^\bullet_{G^\bd}((i_{O^\bd})_! \ul{\kk}_{O^\bd}, \kappa_{\bep*} \iota_\bep^! \cL_\beta)\langle  \dim X(\bd) -c  \rangle \\
& \simeq \rH^\bullet_{G^\bd}(i_{O^\bd}^! \kappa_{\bep*} \iota_\bep^! \cL_\beta )\langle  \dim X(\bd) - c  \rangle \\
&\simeq \rH^\bullet_{T_\bd}(i_{x(\bd)}^! \kappa_{\bep*}\iota_\bep^! \cL_\beta)\langle  \dim X(\bd) -c  \rangle
\end{align}
as graded $H_\beta$-modules,
where the last equality is due to the induction equivalence with $\Stab_{G^\bd} x(\bd) = T_\bd$.
As a consequence, we obtain
\[ \tM_{\ii} (\bep) \simeq \rH^\bullet_{T_\bd}(i_{x(\bd)}^! \kappa_{\bep*}\iota_\bep^! \cL_\beta)\otimes_{\rH^\bullet_{T_\bd}(\mathrm{pt}, \kk)} \wh{\rH}^\bullet_{\Gm}(\mathrm{pt}, \kk) \simeq \wh{\rH}_{\Gm}^\bullet(i_{x(\bd)}^! \kappa_{\bep*}\iota_\bep^! \bar{ \cL}_\beta) \]
as $\wh{H}_\beta$-modules.
Here the last isomorphism follows from \cite[Lemma 6.7.4]{AcharBook}.
The same lemma in \cite{AcharBook} also yields the specialized isomorphism
\[ M_\ii(\bep) \simeq \rH_{\Gm}^\bullet(i_{x(\bd)}^! \kappa_{\bep*}\iota_\bep^! \bar{ \cL_\beta}) \otimes_{\rH^\bullet_{\C^\times}(\mathrm{pt}, \kk)} \kk \simeq \rH^\bullet(i_{x(\bd)}^! \kappa_{\bep*}\iota_\bep^! \bar{ \cL}_\beta),  \] 
which completes the proof.
 \end{proof}
Now, restricting to the transversal slice $S(\bd) \subset X(\beta)$ considered in Section~\ref{Ssec:slice}, we define
\begin{align} 
\ocA_\bd &\seq i_{S(\bd)}^! \bar{\cL}_\beta [\dim X(\beta) - \dim S(\bd)] \\
&\simeq \bigoplus_{\bd' \in \mathrm{KP}_\ii(\beta)} \IC(\ol{O(\bd')} \cap S(\bd), \kk) \otimes L_\ii(\bd'), \label{eq:DecA} 
\end{align}
which is a $T_\bd$-equivariant semisimple perverse sheaf. 
Here, the isomorphism is due to \eqref{eq:strS}, Proposition~\ref{Prop:ICa} and \cite[Theorem 5.4.1]{GM2}.
The functor $i_{S(\bd)}^!$ induces a $\kk$-algebra homomorphism 
\begin{equation} \label{eq:rest_hom}
\hH_\beta \simeq \Hom^\bullet_{G(\beta)}(\bar{\cL}_\beta, \bar{\cL}_\beta)^\wedge  
\to
\Hom^\bullet_{T_\bd}(\ocA_\bd, \ocA_\bd)^\wedge.
\end{equation}

\begin{Prop} \label{Prop:M=H}
For each $\bep \in J^\bd$, we have an isomorphism of $\hH_\beta$-modules
\[\tM_{\ii}(\bep) \simeq \wh{\rH}^\bullet_{\Gm}(i_{\bep, x(\bd)}^* i_\bep^! \ocA_\bd), \]
 where the $\hH_\beta$-module structure on the RHS is given through the homomorphism \eqref{eq:rest_hom}.
Specializing at $z=0$, we obtain
\[ M_{\ii}(\bep) \simeq \rH^\bullet(i_{\bep, x(\bd)}^* i_\bep^! \ocA_\bd).\]
\end{Prop}
\begin{proof}
The assertion follows from Proposition~\ref{Prop:Mig} together with Lemma~\ref{Lem:slice}.
\end{proof}

\subsection{Geometric interpretation of $R$-matrices}
In this subsection, we establish a geometric interpretation of the renormalized $R$-matrices between the deformed mixed tensor products, analogous to Theorem \ref{tau} for the quantum loop algebras.   
First, we need a lemma.
 Recall the quantity $\alpha(j,k)$ for $j,k \in J$ from \eqref{eq:dd}. 
\begin{Lem} \label{Lem:ddqH}
For any $j,k \in J$, we have
\[
\alpha(j,k) = \dext(\alpha_{\ii, j}, \alpha_{\ii, k}) + \dext(\alpha_{\ii, k}, \alpha_{\ii, j}).
\]

\end{Lem}
\begin{proof}
We may assume $j > k$ without loss of generality.
Then, we have $\dext(\alpha_{\ii, j}, \alpha_{\ii, k}) = 0$ by Lemma~\ref{Lem:dhom}~(1).
We have to show $\alpha(j,k) = e \seq \dext(\alpha_{\ii, k}, \alpha_{\ii, j})$.
Let $\bd = \bdelta_j + \bdelta_k \in \N^{\oplus J}$, $\bep = \bep_s(\bd) = (j,k)$ and $\bep' = \bep_c(\bd) =(k,j)$.
We abbreviate $E = E(\bd) \simeq \C^e$, $x = x(\bd)$, and $S = S(\bd) = x + E$.
We have $S(\bep) = \{ x \}$ and  $S(\bep') = S$.
The action of $s \in \C^\times$ on $E$ (through $\rho^\vee$) is simply the multiplication by $s^{j-k}$. 
By Proposition~\ref{Prop:M=H} and Proposition~\ref{Prop:attr}, we have isomorphisms
\begin{align}
\tL_{\ii, j} \star_\Oo \tL_{\ii, k} 
&\simeq \wh{\rH}^\bullet_{\C^\times}(i^! \ocA_\bd) 
\simeq \Hom^\bullet_{\C^\times}(\ul{\kk}_{\{x\}}, i^!\ocA_\bd)^\wedge
\simeq \Hom^\bullet_{\C^\times}(i_{*}i^* \ul{\kk}_{S},\ocA_\bd)^\wedge,
\\
\tL_{\ii, k} \star_\Oo \tL_{\ii, j} 
&\simeq \wh{\rH}^\bullet_{\C^\times}(i^* \ocA_\bd)
\simeq \Hom^\bullet_{\C^\times}(\ul{\kk}_{\{x\}}, p_{*}\ocA_\bd)^\wedge
\simeq \Hom^\bullet_{\C^\times}(\ul{\kk}_{S},\ocA_\bd)^\wedge,
\end{align} 
where $i \colon \{ x \} \hookrightarrow S$ and $p \colon S \to \{x\}$ are the trivial maps. 
The adjunction morphisms
\[ \ul{\kk}_{S} \to i_* i^* \ul{\kk}_{S}, \quad \text{ and } \quad i_* i^* \ul{\kk}_{S} \simeq i_! i^! \ul{\kk}_{S} [2e] \to \ul{\kk}_{S} [2e]\]
respectively give rise to the homomorphisms
\begin{align} 
\cR_{j,k} &\colon \Hom^\bullet_{\C^\times}(i_{*}i^* \ul{\kk}_{S},\ocA_\bd) \to \Hom^\bullet_{\C^\times}(\ul{\kk}_{S},\ocA_\bd), \\ 
\cR_{k,j} &\colon 
 \Hom^\bullet_{\C^\times}(\ul{\kk}_{S},\ocA_\bd) \to \Hom^\bullet_{\C^\times}(i_{*}i^* \ul{\kk}_{S},\ocA_\bd) \langle  -2e  \rangle
\end{align}
of graded $\Hom^\bullet_{T_\bd}(\ocA_\bd, \ocA_\bd)$-modules. 
Let $\bar{\cR}_{j,k}$ and $\bar{\cR}_{k,j}$ denote their specializations at $z=0$ respectively, which are obtained simply by forgetting the $\C^\times$-equivariance. 
Since $S$ has the finite stratification \eqref{eq:strS}, there is a unique $\bd' \in \mathrm{KP}_\ii(\beta)$ with $\beta = \alpha_{\ii, j} + \alpha_{\ii, k}$ such that $S = \ol{O(\bd')} \cap S$.
Then, the decomposition \eqref{eq:DecA} tells us that the perverse sheaf $\ocA_\bd$ contains both $\ul{\kk}_{\{x\}} = i_*i^*\ul{\kk}_{S}$ and $\kk_{S}[e]$ as summands.  
Thus, it follows that the specializations $\bar{\cR}_{j,k}$ and $\bar{\cR}_{k,j}$ are both non-zero, and hence the completions of $\cR_{j,k}$ and $\cR_{k,j}$ are identical to the renormalized $R$-matrices $R_{j,k}$ and $R_{k,j}$ respectively (up to multiples in $\Oo^\times$).  
In particular, we have $\cR_{k,j} \circ \cR_{j,k} = a z^{\alpha(j,k)} \id$ for some $a \in \kk^\times$. 
On the other hand, $\cR_{k,j} \circ \cR_{j,k}$ is a graded homomorphism of degree $2e$ by construction.
Therefore, we get $\alpha(j,k) = e$ as desired.
\end{proof}

Recall the preorder $\lesssim$ of $J^\bd$ from Section \ref{mjf}.
Lemmas \ref{Lem:dhom} (1) and \ref{Lem:ddqH} implies the following.

\begin{Cor}
For $\bep, \bep' \in J^\bd$, we have $S(\bep) \subset S(\bep')$ if and only if $\bep \lesssim \bep'$. 
For the standard (resp.~costandard) sequence $\bep_s$ (resp.~$\bep_c$), we have $S(\bep_s) = \{x(\bd)\}$ (resp.~$S(\bep_c) = S(\bd)$).
\end{Cor}

For $\bep, \bep' \in J^\bd$ satisfying $\bep \lesssim \bep'$, let $i_{\bep', \bep} \colon S(\bep) \hookrightarrow S(\bep')$ denote the inclusion. 
Note that we have $i_{\bep, x(\bd)} = i_{\bep, \bep_s}$ and $i_\bep = i_{\bep_c, \bep}$ for any $\bep \in J^\bd$.
We have the following diagram of inclusions
$$
\xymatrix{
& S(\bep_c) = S(\bd)& \\
S(\bep)
\ar@{->}[ur]^-{i_\bep} \ar@{->}[rr]^-{i_{\bep', \bep}}
&& S(\bep'). \ar@{->}[ul]_-{i_{\bep'}} \\
& S(\bep_s) = \{x(\bd)\}
\ar@{->}[ur]_-{i_{\bep',x(\bd)}} \ar@{->}[ul]^-{i_{\bep,x(\bd)}}& 
}
$$
Then the canonical morphism of functors $i_{\bep', \bep}^! \to i_{\bep', \bep}^*$ induces a morphism
\[ i_{\bep,x(\bd)}^* i_\bep^! \ocA_\bd = i_{\bep,x(\bd)}^* i_{\bep', \bep}^! i_{\bep'}^! \ocA_\bd \to i_{\bep,x(\bd)}^* i_{\bep', \bep}^* i_{\bep'}^! \ocA_\bd = i_{\bep',x(\bd)}^* i_{\bep'}^! \ocA_\bd.\]
Taking the cohomology, we obtain a homomorphism of graded $H_\beta$-modules:
\[ \cR_{\bep', \bep} \colon \rH^\bullet_\Gm(i_{\bep,0}^* i_\bep^! \ocA_\bd) \to \rH^\bullet_\Gm(i_{\bep',0}^* i_{\bep'}^! \ocA_\bd). \]
A proof of the following proposition can be the same as Proposition~\ref{tau}.

\begin{Prop}
\label{Prop:Rmatg_qH}
Let $\bep, \bep' \in J^\bd$ satisfying $\bep \lesssim \bep'$.
The completion $\wh{\cR}_{\bep', \bep}$ of the homomorphism $\cR_{\bep', \bep}$ is identical to the intertwiner $R_{\bep',\bep}$ up to multiples in $\Oo^\times$. 
\end{Prop}

\subsection{Proof of Theorem \ref{Thm:adapted}}
Now, we are ready to prove our main theorem. 
One can verify the following analog of Theorem~\ref{Amain} by the same argument using Propositions \ref{Prop:M=H} and \ref{Prop:Rmatg_qH} instead of Propositions \ref{Nsheaf} and \ref{tau} respectively. 

\begin{Thm} 
\label{Thm:main_qH}
Assume that $\ii$ is a reduced word for $w$ adapted to the quiver $Q$. 
For any $\bd \in \mathrm{KP}_\ii(\beta)$ and $\bep \in J^\bd$, we have
the following equality in $K(\Cc_w)_t$:
\begin{equation} \label{eq:Mt_for_qH}
[M_\ii(\bep)]_t = \sum_{\bd' \in \mathrm{KP}_\ii(\beta)} \left( \sum_{n \in \Z} t^n \dim_\kk \rH^n (i_{\bep,x(\bd)}^* i_\bep^! \IC(\ol{O(\bd')}\cap S(\bd), \kk)) \right) [L_\ii(\bd')]. 
\end{equation}
\end{Thm}

Comparing \eqref{eq:Mtg_slice} with \eqref{eq:Mt_for_qH}, we obtain the desired equality \eqref{eq:mixqH} when our reduced word  $\ii$ is adapted to $Q$.
Thus, we have proved Theorem \ref{Thm:adapted}.

As a byproduct of the proof, we also obtain the following analog of Corollary \ref{Cor:ss}. 
\begin{Cor} 
When $\ii$ is adapted to a quiver of type $\fg$,  the filtration layer $\Gr_n^FM_\ii(\bep) = F_{n}M_\ii(\bep)/F_{n+1}M_\ii(\bep)$ is a semisimple $\hH$-module for any $\bep \in J^d$ and $n \in \Z$.
\end{Cor}

\subsection{Proof of Lemma~\ref{Lem:Lg}}
\label{Ssec:pfLem}
In this subsection,  we give a proof of Lemma~\ref{Lem:Lg} above.
First we recall the construction of the isomorphism~\eqref{eq:VVg}. 
Let $\beta \in \sQ^+$ and recall the $G(\beta)$-variety $\tF_\nu$, which we identify with the variety of pairs $(x, F^\bullet)$ of $x \in X(\beta)$ and $I$-graded flag $F^\bullet = (V^\beta = F^0 \supset F^1 \supset \cdots \supset F^{|\beta|} = 0)$ such that $\vdim F^{n-1}/ F^n = \alpha_{\nu_n}$ and $x(F^n) \subset F^n$ for any $1 \le n \le |\beta|$.  
Then the proper morphism $p_3 \colon \tF_\nu \to X(\beta)$ is simply the projection $(x, F^\bullet) \mapsto x$.
We consider the convolution algebra of the $G(\beta)$-equivariant Borel-Moore homologies:
\[ Z_\beta \seq \bigoplus_{\nu, \nu' \in I^\beta} \rH^{G(\beta)}_\bullet(\tF_\nu \times_{X(\beta)} \tF_{\nu'}, \kk).\]
With this notation, the isomorphism \eqref{eq:VVg} in Theorem~\ref{Thm:VVg} is constructed as the composition of two isomorphisms of $\kk$-algebras:
\begin{equation} \label{eq:VVg2}
H_\beta \simeq Z_\beta \simeq \Hom^\bullet_{G(\beta)}(\cL_\beta, \cL_\beta).
\end{equation}
Through the first isomorphism $H_\beta \simeq Z_\beta$, the idempotent $e(\nu)$ goes to the fundamental class of the diagonal $[\Delta(\tF_\nu)] = \Delta_*[\tF_\nu]$, where $\Delta$ is the diagonal embedding, and the element $\tau_k e(\nu)$ goes to the fundamental class of a certain subvariety of $\tF_{\sigma_k \nu} \times \tF_\nu$.  The element $x_n e(\nu)$ goes to $\Delta_* c_1^{G(\beta)}(\mathcal{O}_\nu(n))$, where $c_1^{G(\beta)}(\mathcal{O}_\nu(n))$ denotes the first equivariant Chern class of the $G(\beta)$-equivariant line bundle $\mathcal{O}_\nu(n)$ on $\tF_\nu$ whose fiber at $(x,F^\bullet)$ is $F^{n-1}/F^n$.
The second isomorphism $Z_\beta \simeq \Hom^\bullet_{G(\beta)}(\cL_\beta, \cL_\beta)$ is an equivariant version of the isomorphism in \cite[Section 8.6]{CG}.

In what follows, we fix $j \in J$ and put $\beta = \alpha_{\ii, j}$ for the sake of brevity. 
Let $M^\bullet$ denote the RHS of the desired isomorphism.
We have
\begin{align}
M^\bullet & =
\Hom^\bullet_{G(\beta)}((i_{O(\beta)})_! \ul{\kk}_{O(\beta)}, \cL_{\beta})\langle \dim X(\beta) \rangle \\
& \simeq \rH^\bullet_{G(\beta)}(i_{O(\beta)}^* \cL_\beta [-\dim X(\beta)]) \\
& \simeq \rH^\bullet_{\C^\times}(i_{x(\beta)}^* \cL_\beta[- \dim X(\beta)]),
\end{align}
where the last isomorphism comes from the induction equivalence together with $\Stab_{G(\beta)} x(\beta) = \C^\times \id_{V^\beta} \simeq \C^\times$.
By the definition of $\cL_\beta$ and the decomposition theorem, we have
\begin{equation} \label{eq:iL=Hk}
i_{x(\beta)}^* \cL_\beta[- \dim X(\beta)] \simeq \bigoplus_{\nu \in I^\beta} \rH^\bullet(\tF_\nu (x(\beta)), \kk) \langle  d(\nu, \beta)  \rangle\otimes \ul{\kk}_{\{ x(\beta)\}}, 
\end{equation}
where $\tF_\nu (x(\beta)) \seq p_3^{-1}(x(\beta)) \subset \tF_\nu$ denotes the variety  of $I$-graded flags stable under $x(\beta)$, and $d(\nu, \beta) \seq - \dim \tF_\nu + \dim X(\beta)$. 
Note that $i_{x(\beta)}^* \IC(\bd)$ is isomorphic to $\ul{\kk}_{\{x(\beta)\} }[\dim X(\beta)]$ if $\bd = \bdelta_j$, and zero otherwise.
Thus, the decomposition \eqref{eq:DTL} implies an isomorphism
\[i_{x(\beta)}^* \cL_\beta[- \dim X(\beta)] \simeq L^\bullet_{\ii,  j } \otimes \ul{\kk}_{\{ x(\beta)\}}.
\]
Comparing this with \eqref{eq:iL=Hk}, we get an isomorphism
\begin{equation} \label{eq:L=H}
L^\bullet_{\ii, j} \simeq \bigoplus_{\nu \in I^\beta} \rH^\bullet(\tF_\nu(x(\beta)), \kk) \langle  d(\nu, \beta)  \rangle 
\end{equation}
of graded vector spaces. 
By construction, the $H_\beta$-action on $L_{\ii,j}^\bullet$ corresponds to the convolution action of $Z_\beta$ on the RHS of \eqref{eq:L=H}. 
On the other hand, from \eqref{eq:iL=Hk}, we get
\begin{equation} \label{eq:M=H}
 M^\bullet \simeq \bigoplus_{\nu \in I^\beta} \rH^\bullet(\tF_\nu (x(\beta)), \kk) \langle  d(\nu, \beta)  \rangle\otimes \rH^\bullet_{\C^\times}(\pt, \kk). 
\end{equation}
Through the isomorphisms  $H_\beta \simeq Z_\beta$  and \eqref{eq:M=H}, the $H_\beta$-action on $M^\bullet$ is translated into the convolution action of $Z_\beta$ on the RHS of \eqref{eq:M=H}. 
In particular, the action of the element $x_n e(\nu)$ on $M^\bullet$ corresponds to the multiplication of the equivariant Chern class $c_1^{\C^\times}(\mathcal{O}_\nu(n)|_{\tF_\nu(x(\beta))})$ on the RHS of \eqref{eq:M=H}.
Let us make an identification $\rH^\bullet_{\C^\times}(\pt) = \kk[z]$ such that $z$ represents the fundamental weight of $\C^\times$. 
Since the $\C^\times$-action on the fibers of $\mathcal{O}_\nu(n)$ is of weight $1$, we have 
\[c_1^{\C^\times}(\mathcal{O}_\nu(n)|_{\tF_\nu(x(\beta))}) = c_1(\mathcal{O}_\nu(n)|_{\tF_\nu(x(\beta))}) + z.\]
This matches with the formula \eqref{eq:affx} (with $a(z) = z$) defining the action of $x_n$ on the affinization $L^\bullet_{\ii,j}[z]$.
Thus, the isomorphisms \eqref{eq:L=H} and \eqref{eq:M=H} yield an isomorphism $M^\bullet \simeq L^\bullet_{\ii, j} [z]$ of graded $H_\beta$-modules, which completes the proof.



\end{document}